\documentclass[12pt,leqno]{amsart}

\usepackage[showframe=false,marginparsep=0in,
top=0.9in, 
text={5.85in,9.3in},
left=1.325in, 
bindingoffset=0in,twoside=true]{geometry} 
\usepackage{amsmath,amssymb}

\usepackage{hyperref}	% support for hypertext in LaTeX
\hypersetup{
	pdfauthor={...},
	pdftitle={...},
	pdfsubject={...},
	urlcolor=blue!75!black,
	colorlinks=true,%
	linkcolor=blue!75!black,%
	citecolor=blue!75!black,%
	filecolor=blue!75!black,%
	menucolor=blue!75!black,%
}	% setup for hyperref package

\usepackage[initials, alphabetic, msc-links, nobysame, short-publishers]{amsrefs}

\usepackage{tikz}
\usepackage{tikz-cd}
\usetikzlibrary{arrows,backgrounds,decorations.pathmorphing,decorations.pathreplacing,positioning,fit,petri,graphs,automata,intersections,shapes,calc,tikzmark}

\usepackage{bbm}		% "Blackboard-style" cm fonts
\usepackage{mathtools}	% provides psmallmatrix
\usepackage{etoolbox}	% used for figure and table numbering

\usepackage{multirow}	% tabular cells spanning multiple rows
\usepackage{float}		% provides the H float modifier option
\usepackage{subcaption} % provides subfigures
\usepackage{wrapfig}	% produces figures which text can flow around
\usepackage{longtable}	% allow tables to flow over page boundaries
\usepackage{caption}

\usepackage{stmaryrd}		% \llbracket and \rrbracket
\usepackage{colonequals}	% \colonequals and \equalscolon
\usepackage{adjustbox}		% for scaling diagrams
\usepackage{nicefrac}		% for nice fractions

\usepackage[color=yellow,disable]{todonotes}

\usepackage[inline]{enumitem}

\usepackage{xcolor}
\usepackage{array}

\setlist[itemize]{leftmargin=*}
\setlist[enumerate]{leftmargin=*}	% margin indent to zero in lists

\tikzset{rotate/.style={anchor=south, rotate=90, inner sep=.5mm}}
\tikzset{
	labl/.style={anchor=south, rotate=90, inner sep=.5mm}
}
\makeatletter
\newcommand*\bigcdot{\mathpalette\bigcdot@{.75}}
\newcommand*\bigcdot@[2]{\mathbin{\vcenter{\hbox{\scalebox{#2}{$\m@th#1\bullet$}}}}}
\makeatother
\newcommand{\bt}{\bullet}
\newcommand{\blank}{\mathstrut_{\text{\textemdash}}}
\newcommand{\Dcat}{\mathcal{D}} % comma category
\newcommand{\Xcat}{\mathcal{X}} % category
\newcommand{\Xset}{\mathcal{X}} % category
\newcommand{\Yset}{\mathcal{Y}} % category

\newcommand{\Hcat}{\mathcal{H}}

\newcommand{\A}{\Lambda}
\newcommand{\Ainf}{\Lambda_{e}}
\newcommand{\B}{\Gamma}
\newcommand{\Hrd}{\Gamma}
\newcommand{\AI}{A}

\newcommand{\Gr}{\mathrm{G}}

\newcommand{\CP}{P}

\newcommand{\Rx}{\mathsf R}
\newcommand{\KK}{\mathsf K}
\newcommand{\mx}{\mathfrak m}
\newcommand{\px}{\mathfrak{p}}
\newcommand{\kk}{\mathbf{k}}

\newcommand{\N}{\mathbb{N}}
\newcommand{\Z}{\mathbb{Z}}
\newcommand{\Q}{\mathbb{Q}}

\newcommand{\op}{op}
\newcommand{\sgn}{\mathrm{sgn}}

\newcommand{\inv}{\xi}

\newcommand{\AG}{\mathcal{AG}} % multi-set of AG-invariants

\newcommand{\Exc}{\mathcal{E}}
\newcommand{\Ho}{H} % homology
\newcommand{\MD}{\mathbbm{D}} % Matlis duality
\newcommand{\SD}{\mathbbm{S}} % Serre duality
\newcommand{\Tw}{\mathbbm{T}} % Twist functor

\DeclareMathOperator{\lotimes}{\overset{\mathbf{L}}{\otimes}_{\A}}

\DeclareMathOperator{\md}{\mathsf{mod}}
\DeclareMathOperator{\Md}{\mathsf{Mod}}

\DeclareMathOperator{\proj}{\mathsf{proj}}

\DeclareMathOperator{\add}{\mathsf{add}}
\DeclareMathOperator{\ind}{\mathsf{ind}}
\DeclareMathOperator{\simp}{\mathsf{sim}}

\DeclareMathOperator{\Mat}{\mathsf{Mat}}
\DeclareMathOperator{\GL}{\mathsf{GL}}

\DeclareMathOperator{\bc}{\mathsf{bc}}
\DeclareMathOperator{\pc}{\mathsf{p}}

\DeclareMathOperator{\D}{\mathrm{D}}
\DeclareMathOperator{\Db}{\D^{\mathrm{b}}}
\DeclareMathOperator{\Dsg}{\D_{\mathrm{sg}}}

\newcommand{\Dh}{\D^-_{\AI}}
\DeclareMathOperator{\Hot}{\mathsf{K}}
\DeclareMathOperator{\Hotac}{\mathsf{K}_{\mathsf{ac}}}

\DeclareMathOperator{\per}{\mathsf{per}}
\DeclareMathOperator{\perfd}{\per_{\mathsf{fd}}}

\DeclareMathOperator{\im}{\mathsf{im}}

\DeclareMathOperator{\rad}{\mathsf{rad}}
\DeclareMathOperator{\head}{\mathsf{top}}

\DeclareMathOperator{\prdim}{\mathsf{pr.dim}}
\DeclareMathOperator{\gldim}{\mathsf{gldim}}

\DeclareMathOperator{\End}{\operatorname{End}}
\DeclareMathOperator{\Hom}{\operatorname{Hom}}
\DeclareMathOperator{\HomD}{\Hom}
\DeclareMathOperator{\HomDc}{\Hom}
\DeclareMathOperator{\EndD}{\End}
\DeclareMathOperator{\Ext}{\operatorname{Ext}}
\DeclareMathOperator{\Tor}{\operatorname{Tor}}

\newcommand{\id}{\mathsf{id}}

\DeclareMathOperator{\tria}{\mathsf{tria}}
\DeclareMathOperator{\exc}{\mathsf{exc}}

\newcommand{\Fcat}{\mathcal{F}}
\newcommand{\Scat}{\mathcal{S}}
\newcommand{\Tau}{\mathcal{T}}
\newcommand{\Ucat}{\mathcal{U}}
\newcommand{\Tcat}{\mathcal{T}}

\DeclareMathOperator{\rk}{ \mathsf{rk}}

\DeclareMathOperator{\CYdim}{ \mathrm{CY}-\dim}
\newcommand{\krdim}{d}

\newcommand{\Surf}{\mathsf{S}}

\newcommand{\itop}{\mathrm G}
\newcommand{\imid}{\mathrm I}
\newcommand{\jtop}{\mathrm F}
\newcommand{\jmid}{\mathrm J}

\newcommand{\aisle}{\mathcal{U}}
\newcommand{\coaisle}{\mathcal{V}}

\numberwithin{equation}{section}

\AtBeginEnvironment{figure}{\stepcounter{equation}}
\AtBeginEnvironment{table}{\stepcounter{equation}}

\theoremstyle{plain}
\newtheorem{thm}[equation]{Theorem}
\newtheorem{lem}[equation]{Lemma}
\newtheorem{prp}[equation]{Proposition}
\newtheorem{cor}[equation]{Corollary}

\newtheorem{dfn}[equation]{Definition}

\newtheorem{rmk}[equation]{Remark}
\newtheorem{notation}[equation]{Notation}

{\begin{tikzcd}[ampersand replacement=\&,baseline=0pt,
		cells={rectangle, inner sep=1pt, outer sep=1pt, anchor=center, minimum width=0.0cm, minimum height=0.0cm},
		labels={rectangle, font=\normalsize}, 	
		arrows={->,semithick, >=stealth', bend left=45},
		]}{\end{tikzcd}}

\newenvironment{td}[0]		% standard for ring and category diagrams
{\begin{tikzcd}[ampersand replacement=\&, cells={outer sep=2pt, inner sep=2pt}, 
		]}{\end{tikzcd}}

\newenvironment{td-long}[0]		% standard for ring and category diagrams
{\begin{tikzcd}[ampersand replacement=\&, cells={outer sep=2pt, inner sep=2pt}, column sep=1.5cm,
		]}{\end{tikzcd}}

\newenvironment{td-sp}[0]
{\begin{tikzcd}[ampersand replacement=\&, cells={outer sep=2pt, inner sep=2pt},
		row sep=10pt,
		]}		
	{\end{tikzcd}}

\definecolor{darkred}{rgb}{0.5, 0, 0}

\author{Wassilij Gnedin}
\address{Universität Paderborn, 
	Institut für Mathematik,
	Warburger Straße 100,
	33098 Paderborn, Germany}
\email{wassilij.gnedin@math.uni-paderborn.de}

\begin{document}
	\title{Derived invariants of gentle orders}
	
	\begin{abstract}
This article is concerned with the derived representation theory of certain infinite-dimensional gentle algebras called gentle orders. 
For a gentle order, we provide a factorization of the derived Nakayama functor, study its fractionally Calabi-Yau objects and exceptional cycles, and establish that certain combinatorial invariants of its underlying quiver are derived invariants, analogous to results for finite-dimensional gentle algebras.
	\end{abstract}
	\maketitle
	
	\tableofcontents

%\input{tex/01-preliminaries}
%\newpage
%\section{Introduction}

%
Finite-dimensional gentle algebras have been intensively studied in recent years \cites{OPS, APS, Opper, CJS}, in particular, due to their appearance in homological mirror symmetry \cites{HKK,LP}.
The derived representation theory of finite-dimensional gentle algebras is well-understood.
%Although defined in purely combinatorial terms, the class of finite-dimensional gentle algebras is preserved under derived equivalences \cite{SZ}. 
In \cite{AG}, Avella-Allaminos and Geiß introduced certain derived invariants 
for a finite-dimensional gentle algebra which can be computed from its underlying quiver.
To study the auto-equivalences of bounded derived categories of certain gentle algebras, Broomhead, Pauksztello and Ploog introduced the notion of exceptional cycles \cite{BPP}.
% generalizing that of spherical objects by Seidel and Thomas \cite{ST},
In \cite{GZ} Guo and Zhang gave a partial classification of the exceptional cycles  
in the bounded derived category of a finite-dimensional gentle algebra and related them to AG-invariants.
%Building on previous works \cites{BM, BD, CPS, ALP} on the explicit combinatorics of  of a finite-dimensional algebra, Opper, Plamondon and Schroll described a geometric model 
Extending work of \cite{HKK}, Lekili and Polishchuk established derived invariants for homologically smooth graded gentle algebras in terms of geometric invariants of an associated graded surface \cite{LP}. 
Independently, Opper, Plamondon and Schroll provided a geometric model for the bounded derived category of  a finite-dimensional gentle algebra \cite{OPS}.
% using previous combinatorial results due to \cites{BM, BD, CPS, ALP}.
These geometric models were successfully employed to deduce a complete derived equivalence classification
%\cites{APS,O}  as well as 
of homologically smooth and proper graded gentle algebras \cites{APS, Opper, JSW}.
%l as to the study of recollements for the latter class \cite{CJS}.

%By \cites{BM,BD} gave a complete classification of the indecomposable objects in the bounded derived category of a finite-dimensional gentle algebra,  and $\cite{ALP}$ provided bases for their morphism spaces
% The representation theory of gentle algebras received new attention after \cite{HKK} related to the surface geometry.

The present article is concerned with a special class of semiperfect infinite-dimensional gentle algebras called \emph{gentle orders}.
Their underlying quivers are distinguished by the property that any arrow lies on an oriented cycle without relations.
Examples of gentle orders have  appeared in connection with Lie theory 
\cites{GP, Khoroshkin}, 
%have motivated the development of techniques leading to a complete classification of indecomposable modules of finite-dimensional gentle algebras \cite{BR}
and are related to a generalized notion of string algebras introduced by Bennett-Tennenhaus \cite{Bennett-Tennenhaus1}.

In contrast to finite-dimensional gentle algebras, the derived representation theory of gentle orders is less understood at the present. 
By work of Burban and Drozd \cite{BD} there is a complete classification of indecomposable objects in the right-bounded derived category of a gentle order.
Work of Palu, Pilaud and Plamondon \cite{PPP} associates a natural marked surface to any gentle order.
The aim of this work is to provide analogues of some of the above-mentioned results for 
the derived representation theory of any gentle order $\A$.

%\cite{LP} extended the notion of AG-invariants 
%to infinite-dimensional graded gentle algebras and showed that they form derived invariants in case of finite global dimension.
%
%However, the connection between the symplectic geometry of the decorated marked surface and 
%the bounded derived category of finitely generated modules over a gentle order
%is conjectural and the geometric model for the latter remains to be fully developed. 
%

By work of Iyama and Reiten \cite{IR}, the derived category of a gentle order admits a \emph{relative Serre functor} $\SD$, an auto-equivalence which restricts to a Serre functor on a certain Hom-finite subcategory of perfect complexes of $\A$. Their work implies that the functor $\SD$ is given by the composition $\nu \circ [1]$ of the derived Nakayama functor $\nu$ with the shift functor.
The action of the functor $\nu$ on the right-bounded derived category $\Dcat_{\A} \colonequals \D^-(\md \A)$ of $\A$ can be computed using previous work \cite{GIK}.
The first main result of  this article concerns a more conceptual description.
\begin{thm}[Theorem~\ref{thm:nu-factors}] \label{thm:A}
	For any gentle order $\A$ the derived Nakayama functor admits a factorization 
	$$
	\nu  \cong \inv^* \circ \Tw_b \circ \ldots \circ \Tw_2 \circ \Tw_1 \colon \begin{td} \Dcat_{\A} \ar{r}{\sim} \& \Dcat_\A \end{td}
	$$
	of commuting functors, where $\inv^*$ is the auto-equivalence induced by an involution of the ring $\A$ and $\Tw_1, \Tw_2, \ldots, \Tw_b$ denote standard equivalences associated to certain exceptional cycles.
%	 in $\perfd \A$.
	\end{thm}

The gentle order $\A$ has a maximal idempotent $e$
such that the projective module $B \colonequals \A e$ is $\nu$-invariant.
%such that $e \A e$ has infinite global dimension 
This idempotent gives rise to 
triangulated subcategories
$$
\begin{td}
	\aisle_{\A}  \colonequals	\Hot^-(\add \A e) 
	\ar[yshift=0pt,hookrightarrow]{r}
	\& 
	\Dcat_{\A} = \D^-(\md \A)
	\ar[yshift=0pt,<-,hookleftarrow]{r} \&  
	\Dh(\md \A)\equalscolon \coaisle_{\A} 
\end{td}$$
where $A$ is a certain hereditary quotient of $\A$.
The pair $(\aisle_{\A},\coaisle_{\B})$
forms a stable $t$-structure  of the category $\Dcat_{\A}$, 
which yields a useful perspective  to study
homologically distinguished objects.
%The aisle $\aisle_{\A}$ is big in the sense that it contains all band complexes,
%while the $\coaisle_{\A}$ is representation-discrete.
%\todo{reference}
%This stable $t$-structure yields a useful division of the category

%The next statement gives a description of fractionally Calabi-Yau objects and exceptional cycles.
\begin{thm}[Propositions~\ref{prp:recol4} and \ref{prp:exc}]\label{thm:B}
	For any gentle order $\A$ the following statements hold.
	\begin{enumerate}
		\item Any fractionally Calabi-Yau object in $\Dcat_\A$ is contained in $\aisle_{\A} \cup \coaisle_{\A}$.\\ More precisely the following holds.
		\begin{enumerate}
		\item The aisle $\aisle_{\A}$ is equal to the full subcategory of $\nu$-periodic objects in $\Dcat_\A$.
		\item Any non-$\nu$-periodic fractionally Calabi-Yau object is contained in $\coaisle_{\A}$.
			\end{enumerate}
		\item An object $X$ in $\Dcat_\A$ gives rise to an exceptional cycle if and only if 
		it is isomorphic to a shift of a simple $\A$-module contained in $\coaisle_{\A}$
		or $X \in \Ucat_{\A}$ and
$X$ is $1$-spherical or induces an exceptional $2$-cycle.
		\end{enumerate}
\end{thm}

%The underlying quiver $(Q,I)$ of the gentle order $\A$ can be either `decomposed' in permitted cycles without relations. Alternatively, $(Q,I)$ can be decomposed into forbidden cycles which have a maximal amount of zero relations and a mixed type of cycles which we will call AG-cycles. 
%The cycles give rise to numerical data called AG-invariants.
%Moreover, 

The homological interpretation of the aisle $\aisle_{\A}$ as $\nu$-periodic objects allows to derive 
the first part in the next statement
concerning derived invariants of gentle orders.
\begin{thm}[Proposition~\ref{prp:der-gen}~\eqref{der-gen1} and Theorem~\ref{thm:der-inv}]\label{thm:C}
	For any derived equivalent ring-indecomposable gentle orders $\A$ and $\B$ the following statements hold.
	\begin{enumerate}
		\item There are equivalences $\aisle_{\A} \overset{\sim}{\longrightarrow} \aisle_{\B}$ 
		and $\coaisle_{\A} \overset{\sim}{\longrightarrow} \coaisle_{\B}$.
		\item
%	\begin{align*}
%		\AG_1(\A) = \AG_1(\B), && \AG_{2}(\A) = \AG_{2}(\B),
%		&&\pc_{\A} = \pc_{\B}, && \bc_{\A} = \bc_\B.
%	\end{align*}
%	that is, 
The orders $\A$ and $\B$ have the same multisets of AG-invariants, the same number of permitted cycles, and the same bicolorability parameter.
	\end{enumerate}
\end{thm}
%In more detail, the quiver $(Q,I)$ underlying a gentle order can be viewed as gluing of permitted cycles. 
%The notions of AG-invariants for gentle orders is motivated by the works \cites{AG,LP} and given by certain numerical data  associated to other types of cycles in the quiver $(Q,I)$.
%The bicolorability parameter is a binary parameter which determines whether there exists a $2$-coloring of the arrows in $(Q,I)$ reflecting its zero relations.

 In forthcoming work, we will provide an explicit description of the subcategory $\coaisle_{\A}$ and extend the results above to graded skew-gentle orders.
 
 This article is structured as follows. Section~\ref{sec:prelim} collects preliminaries on exceptional cycles in triangulated categories and the derived Morita theory of orders.
 Section~\ref{sec:quiver} introduces the quiver-theoretic notions relevant for gentle orders.
 Sectoin~\ref{sec:cartan} is provides a graph-theoretic approach to study the Cartan matrix of a gentle order. Section~\ref{sec:factor} is concerned with the proof of Theorem~\ref{thm:A}. 
Section~\ref{sec:invar} contains the main results and the proofs of Theorems~\ref{thm:B} and \ref{thm:C}.
In Section~\ref{sec:tru} we consider the notion of truncated ribbon graphs and consider translations between quiver-theoretic, graph-theoretic and derived invariants.

\subsubsection*{Acknowledgement} This work was supported by the German Research Foundation SFB-TRR 358/1 2023– 491392403. 
I would like to thank Igor Burban and Kyungmin Rho for helpful discussions of results related to this work.

\section{Preliminaries on derived categories of orders}
%In this section, we recall exceptional  triangulated categories 

\label{sec:prelim}

\subsubsection*{Conventions on categories and complexes}
Any \emph{full subcategory} of another category is meant to be a full subcategory closed under isomorphisms.
Any functor between triangulated categories is assumed to be an exact, additive functor.

For a set of isomorphism classes of objects $\Ucat$ in a triangulated category $\Tcat$ we denote 
by $\tria(\Ucat)$ the smallest triangulated subcategory containing $\Ucat$,
by $\mathstrut^{\perp} \Ucat$ the full subcategory in $\Tcat$ given by objects $X$ such that $\Hom_{\Tcat}(X,U)=0$ for all $U \in \Ucat$,
and by  $U^{\perp}$ 
the full subcategory given by objects $Y$ of $\Tcat$ with $\Hom_{\Tcat}(U,Y)=0$ for all $U \in \Ucat$.

Modules over any ring are assumed to be \emph{left modules}.
By a \emph{complex} of modules we mean a chain complex 
%$X = (X_n, \partial_n \colon X_n \to X_{n-1})_{n \in \Z}$.
\begin{align*}
	\begin{td}
		X=( 	\ldots	X_{i+1} \ar{r}{\partial_{i+1}} \& X_i \ar{r}{\partial_i} \& X_{i-1} \ldots )
	\end{td}
\end{align*}
The shift of the complex $X$ is denoted by $X[1]$, moves $X$ one step to the left
and changes the sign in each differential.

%For a ring $\A$ and any two objects $X,Y \in \D(\Md \A)$ we set $
%\Hom^{\bt}_{\A} (X,Y) \colonequals \sum_{d\in\Z} \HomD(X,Y[d])$.

\subsection{Serre functor and fractionally Calabi--Yau objects}
\label{subsec:Serre}
Throughout this section, let $\Dcat$ denote a $\kk$-linear triangulated category satisfying the following assumptions.
\begin{itemize}[label=--]
	\item $\Dcat$ is \emph{Hom-finite} in the sense that for any objects $X,Y$ from $\Dcat$
	the $\kk$-linear space of morphisms $X\to Y$, denoted by $\HomD(X,Y)$
	in the following, is finite-dimensional.
	\item $\Dcat$  has a \emph{Serre functor} $\SD$, that is, an equivalence $\SD \colon \Dcat \to \Dcat$ 
%$\SD\colon \Dcat \overset{\sim}{\to} \Dcat$
such that for any objects $X,Y$ from $\Dcat$
there is a bifunctorial isomorphism 
\begin{align} \label{eq:S-dual}
	\HomD(X,Y) \cong \MD \HomD(Y,\SD(X)),\quad \text{where }	\MD \colonequals \Hom_{\kk}(\blank,\kk)\,.
\end{align}
\end{itemize}
%The functor $\SD$ is then an equivalence.
%\todo{reference}.

\begin{dfn}\label{dfn:frac-CY}
	For $(m,n)\in \Z\times \Z\backslash\{(0,0)\}$ a non-zero object $X$ from $\Dcat$ is called
	\emph{$(m,n)$-Calabi--Yau} if $X[m] \cong \SD^n(X)$. The \emph{Calabi--Yau-dimension}  $\CYdim(X)$ is defined as such a pair $(m,n)$ with minimal integer $n \geq 0$. 
The object	$X$ is \emph{$m$-Calabi--Yau} if it is $(m,1)$-Calabi--Yau.
\end{dfn}

\begin{rmk}
%	\begin{enumerate}
%		\item
		If $X$ is $(m,n)$-Calabi--Yau, then $X$ is also $(qm,qn)$-Calabi--Yau for any integer $q \neq 0$. The converse is not true (see 
		e.g. Lemma~\ref{lem:hered-CY}).
%		\item Usually, fractionally Calabi-Yau objects are defined within a $\kk$-linear Hom-finite triangulated category with a Serre functor.
%		For our purposes, it will be convenient to consider fractionally Calabi-Yau objects in the $\kk$-linear triangulated category $\D^-(\md \A)$ which has only a relative Serre functor. \todo{rephrase}
%	\end{enumerate}
	%	provides $(\ell,n)$-Calabi-Yau objects which are not $1$-Calabi-Yau assuming $n >
	%	let $\Gamma$ be a hereditary $R$-order 
	%	over a complete discrete valuation ring $R$
	%	with $n$ simples modules up to isomorphism, then any object in $\Dbfd(\md \Gamma)$ is $(n,n)$-Calabi-Yau, but not $1$-Calabi-Yau.
\end{rmk}

		The next statement is a variation of a known observation \cite{BK}*{Lemma~5.3}.
\begin{lem}\label{lem:CY}
	Let $X \in \Dcat$ be $(m,n)$-Calabi--Yau and $Y \in \Dcat$ be $(p,q)$-Calabi--Yau such that $\triangle \colonequals np -m q \neq 0$
	and $\HomDc(X,Y[i]) = 0$ for any $i \ll 0$ or $i \gg 0$.
	Then $X \in \mathstrut^{\perp} Y$.
	%	Then $\HomDc(X,Y[j])=0$ for all $j \in \Z$.
\end{lem}
\begin{proof}
	Let $i \in \Z$.
	Since $\SD^{\pm qn}(X) \cong X[\pm mq]$ and $\SD^{\pm nq}(Y) \cong Y[\pm np]$
	it holds that
	\begin{align*}
		\HomDc(X,Y[i]) \cong \HomDc(\SD^{\pm qn}(X),\SD^{\pm nq}(Y)[i]) \cong 
		%		\HomDc(X[\pm mq],Y[j\pm np]) \cong 
		\HomDc(X,Y[i\pm\triangle]).
	\end{align*}
	It follows that
	$\HomDc(X,Y[i]) \cong \HomDc(X,Y[i +  \triangle \ell])$ for any $\ell \in \Z$, which implies the claim.
\end{proof}

\subsection{Exceptional cycles}
%Let $\Dcat$ be a $\kk$-linear triangulated category with a Serre functor $\SD$.
The following notion is equivalent to one introduced by Broomhead, Pauksztello and Ploog \cite{BPP}*{Definition~4.2}.
\begin{dfn}\label{dfn:ex-cycle}
	A sequence of objects $\Exc=(E_1, E_2, \ldots, E_n)$ with $n \geq 1$ from $\Dcat$ is an \emph{exceptional $n$-cycle} if it satisfies the following conditions.
	\begin{enumerate}[label=$\mathsf{(E\arabic*)}$,]
		\item \label{E1} For any index $1 \leq j < n$ there is an integer $m_j \in \Z$ with $\SD(E_j) \cong E_{j+1}[m_j]$, and, similarly, there is an integer $m_n \in \Z$ such that $\SD(E_n) \cong E_1[m_n]$ in $\Dcat$. 
		%		\item \label{E2} For each index $1 \leq i \leq n$ it holds that $\sum_{j=1}^n \sum_{d\in\Z} \dim \HomDc(E_i,E_j[d]) = 2$.
		\item 
		\label{E2} 
		%For each index $1 \leq i \leq n$ 
		It holds that $\sum_{p=0}^{n-1} \sum_{i\in\Z} \dim \HomDc(E_1,\SD^p(E_1)[i]) = 2$.
	\end{enumerate}
	%The exceptional cycle is \emph{irredundant} if 
	%$E_n \notin \tria (E_1,\ldots E_n)$.
\end{dfn}
We will need the following observations and additional notions for an exceptional cycle $\Exc$.
\begin{itemize}
	\item Each object $E_j$ in an exceptional cycle $\Exc$ is fractionally 
			%		For each index $1 \leq j \leq n$ the object $E_j$ is fractionally 
			$(m,n)$-Calabi--Yau with $m \colonequals \sum_{j=1}^{n} m_j$.
Moreover, all objects in $\Exc$ have the same Calabi-Yau dimension, which will be denoted by
%	Calabi-Yau of the same Calabi-Yau dimension, which will be denoted by
$\CYdim \Exc$.
\item Vice versa, an $(m,n)$-Calabi--Yau object $E$ appears in an exceptional cycle if and only if 
$n > 0$ and $E$ satisfies condition \ref{E2}.
\item We will call an exceptional cycle \emph{repetition-free} if $E_1 \not\cong E_j[i]$ for any index $2 \leq j \leq n$ and any integer $i \in \Z$. 
\end{itemize}

\begin{rmk}
	%	For any exceptional $n$-cycle as above the following holds.
	We recall a few further remarks on an exceptional $n$-cycle as above.
	\begin{enumerate}
		%		\item More precisely, 
		%		For each index $1 \leq j \leq n$ the object $E_j$ is fractionally $(m,n)$-Calabi--Yau with $m \colonequals \sum_{j=1}^{n} m_j$.
		%		We will call $(m,n)$ the \emph{Calabi--Yau dimension of the exceptional cycle $\Exc$}.
		%		Vice versa, an $(m,n)$-Calabi--Yau object $E$ appears in an exceptional cycle if and only if 
		%		$n > 0$ and $E$ satisfies condition \ref{E2}.
		\item The terminology can be explained as follows.
		\begin{itemize}
			\item If $n > 1$, each object $E_j$ is \emph{exceptional}, that is, $\sum_{i\in\Z} \dim \HomDc(E_j,E_j[i]) =1$.
			\item If $n=1$, the object $E_1$ is not exceptional but constitutes  a \emph{spherical object} in the sense of Seidel and Thomas \cite{ST}. 
			%	Following \cite{BPP}, we will call such an object a \emph{$1$-exceptional cycle} nonetheless.
		\end{itemize}
		%\item 
		%\begin{rmk}
		%	For any exceptional $n$-cycle as above the following holds.
		%	\begin{enumerate}
			%$$\sum_{i=1}^n \sum_{d\in \Z}\dim \Hom_{\Tau}(E,\SD^i(E)[d]) = 2.$$
			\item 
			If, in addition to the previous assumptions, the category $\Dcat$ is algebraic and indecomposable, any exceptional $n$-cycle $\Exc$ in $\Dcat$ 
			gives rise to an equivalence called a \emph{twist functor}
			$ \Tw_{\mathcal{E}} \colon \begin{td} \Dcat \ar{r}{\sim} \& \Dcat \end{td},$
			see \cite{ST} for $n = 1$ and \cite{BPP} for general $n$.	
%			\todo{define twist functors}
%			This fact is our motivation to study exceptional cycles.
%, however we will not use this result and instead study certain auto-equivalences of $\D(\Md \A)$ which restrict to twist functors from exceptional cycles on $\perfd \A$.  This approach is inspired by \cite{IR} and circumvents technical issues with extension of the definition of twist functors to $\D(\Md \A)$.
			%		$$ \Tw_{\mathcal{E}} \colon \begin{td} \Dcat \overset{\sim}{\longrightarrow} \Tau$$
			\end{enumerate}
		\end{rmk}
		
		Following \cite{GZ}*{\S~2.2} we recall a notion of equivalence of exceptional cycles.
		\begin{dfn}\label{dfn:eq-exc}
			Two exceptional cycles 
			$(E_1,\ldots E_n)$
			and $(E'_1,\ldots, E'_{n'})$ in $\Dcat$ are \emph{equivalent}
			if $n=n'$ 
			and there is an index $1 \leq j \leq n$ and an integer $i \in \Z$  such that
			$E'_1 \cong E_j[i]$,
			or, equivalently,
			%	$n= n'$ and
			there are integers $i_1,\ldots i_n \in \Z$
			and a cyclic permutation $\pi \in S_n$ such that
			\begin{align*}
				(E'_1,E'_2\ldots E'_{n'})
				=(E_{\pi(1)}[i_1],E_{\pi(2)}[i_2]\ldots E_{\pi(n)}[i_n])
			\end{align*}
		\end{dfn}
		
		\begin{lem}\label{lem:vanish}
			Let $E \in \Dcat$ be a fractionally $(m,n)$-Calabi--Yau object satisfying \ref{E2}
			and $q$ a non-zero integer.
			%	 appearing in an exceptional cycle
			%	of Calabi--Yau dimension $(m,n)$.
			%with Calabi-Yau dimension $(m,n)$
			%and let $i$ be a non-zero integer.	
			%	and $i$ a non-zero integer. Set $(m,n) \colonequals \CYdim X$.
			%	For any non-zero integer $i$ 
			%	it holds that
			Then
			%	\begin{align*}
				$		\HomDc(\tau^{q}(E),E) \neq 0$
				%	\end{align*}
			if and only if  the following conditions hold.
			\begin{itemize}
				%		[label=--]
				\item if $n=1$, then $m =1$, that is, $E$ is $1$-Calabi--Yau.
				\item if $n>1$, then 
				$m=n$ and $n$ divides $q$,
				or  $(m,n) \in \{(-q,-q-1), (q,q+1)\}$.
			\end{itemize}
		\end{lem}
		\begin{proof}
			In the notations above, for any $p, i\in \Z$ 
			using $\SD^n(E) \cong E[m]$ it holds that
			\begin{align}\label{eq:nu-hom2}
				\HomDc(E,\SD^p(E)[i]) \cong \HomDc(E,\SD^r(E)[i+\ell m])
			\end{align}
			where $\ell \in\Z$ and $0 \leq r <n$ are determined by the equality $p = \ell n+r$.
			On the other hand, since $\tau  = \SD \circ[-1]$ 
			there is an isomorphism
			\begin{align}\label{eq:nu-hom}
				\HomDc(\tau^q(E),E) 
				%	&\cong \HomDc(\SD^q(E)[-\ell],\SD^n(E)[-m]) \\
				&\cong \HomDc(E,\SD^{n-q}(E)[q-m])
			\end{align}
			\begin{itemize}
				\item Assume that $n=1$.
				Since $E$ satisfies \ref{E2} and $(\ell,r)=(p,0)$,	the spaces in \eqref{eq:nu-hom2}
				are not zero if and only if $i \in \{ -pm, (1-p)m \}$.
				With $(p,i) = (1-q, q-m)$ it follows that the spaces in \eqref{eq:nu-hom} are not zero
				if and only if $q \in \{qm, (1+q)m\}$ if and only if $m =1$, where the last equivalence is true since
				$q \neq 0$ and  $m \in \Z$.
				\item Assume that $n > 1$.
				Since $E$ satisfies \ref{E2}
				%		$\sum_{j=0}^{n-1} \sum_{d\in \Z}\dim \Hom(E,\SD^j(E)[d]) = 2$ 
				the spaces in \eqref{eq:nu-hom2} are not zero if and only if $i= -\ell m$ and $r\in \{0,1\}$.
				Setting $(p,i) = (n-q, q-m)$ it follows that the spaces in \eqref{eq:nu-hom} are not zero if and only if 
				$q = (1-\ell)m$ and $(1-\ell)(n-m) \in \{0,1\}$. As $q\neq0$, the last condition holds if and only if $m=n \mid q$ or $q = \pm m = \pm n - 1$.
				%where the last equivalence uses that $i \neq 0$.
				This shows the claim in case $n > 1$.
				\qedhere
			\end{itemize}
			
		\end{proof}

		%\begin{ex}
		%	For a $\Gamma_n$ be the hereditary $\Rx$-order with $n$ simples, the sequence $(S_1,S_2\ldots S_n)$ of simple modules is an exceptional $n$-cycle.
		%\end{ex}
%		
%		
%
%		
%		The next simple observation shall motivate further contents of this article.
%		\todo{review}
%		
%		
%		
%		For any $(p,q) \in \Z \times \N$ let $	\CYcat(p,q) $ denote the full subcategory of objects in $\Dcat$ 
%		which are fractionally $(p,q)$-Calabi--Yau. 
%		%we set 
%		%\tfrac{p}{q} \in \mathbb{Q}^*$ we set
%		%\begin{align*} 
%		%	\CYcat(p,q) \colonequals \left\{X \in \D^-(\md \A) \mid 
%		%	X[p] \cong \SD^q(X) 
%		%	%	\text{ fractionally Calabi-Yau} \mid 
%		%	%\mid
%		%	%	\CYdim X = (p,q) 
%		%	\right\},
%		%	\end{align*}
%	Let us briefly consider the two full subcategories
%	\begin{align*} 
%		\Tcat \colonequals \bigcup_{n\in \N} \CYcat(n,n)&&
%		\Fcat \colonequals \bigcup_{p \neq q} \CYcat(p,q)
%		%		 \exists (m,n) \in \Z \times \N\colon X[m] \cong \SD^n(X)\text{ and } \tfrac{m}{n} = \tfrac{p}{q} \right\}
%	\end{align*}
%	which are given by the $\nu$-periodic objects and all the remaining fractionally Calabi-Yau objects, respectively.
%	Lemma~\ref{lem:CY} implies that $\Hom(\Tcat, \Fcat) = 0$. This property can be viewed as the first step in an attempt to `decompose' $\D^-(\md \A)$ from smaller derived invariant subcategories, which will be made rigorous in the next subsection.
%	
%	One of the goals of this article is to deduce such a decomposition
%	in the case that $\A$ is a gentle order.
%	A priori, there seems to be no reason that why any of the two subcategories should be triangulated. 
	
	\subsection{Auslander--Reiten translation}
The \emph{Auslander-Reiten translation} on $\Dcat$ is defined by the equivalence
	\begin{align*}
		\tau  \colonequals \SD \circ [-1]    \Dcat \overset{\sim}{\longrightarrow} \Dcat, 
%		\qquad
%		X \longmapsto \tau(X) \colonequals 
		%	\nu (X)[d-1] = 
%		\SD(X)[-1].
	\end{align*}
	For any $X$  from $\Dcat$, the Serre duality \eqref{eq:S-dual} yields that
	\begin{align*}
		\EndD(X) \cong \MD \HomD(X, \tau(X)[1])
	\end{align*}
	If $X$ is indecomposable, arguments due to Happel \cite{Happel} imply that there is an \emph{Auslander-Reiten triangle} beginning in $X$, that is, a non-split triangle
	$$
	\begin{td}
		X \ar{r}{f} \& Y \ar{r} \& Z \ar{r} \& X[1]
	\end{td}
	$$
	in the category $\perfd(\A)$ such that $Z$ is indecomposable and for any non-left-invertible morphism $g \colon X \to Y'$
	there exists a morphism $g' \colon Y' \to Y$ with $g = g' f$. 
	In fact, it holds that $Z \cong \tau^{-1}(X)$.
%	In particular, the category $\Dcat$ has an \emph{Auslander-Reiten quiver} $\Gamma_{\Dcat}$, we refer to \cite{Happel} for the corresponding notions.	
%	\todo{elaborate}.

\subsection{Main setup}
\label{subsec:setup}
%Let us recall that a \emph{Noetherian $\Rx$-algebra} $\A$ is given by 
% a ring homomorphism $\Rx \to \A$ which factors through the center of $\A$ and $\A$ is finitely generated as an $\Rx$-module.

The statements of the present section assume the following conditions on an $\Rx$-algebra $\A$.
%Let $\A$ be a \emph{Noetherian $R$-algebra}, that is, there is a ring homomorphism $R \to \A$ which factors through the center of $\A$ and $\A$ is finitely generated when viewed as an $R$-module.
\begin{itemize}[label=--]
	\item 
	$\Rx$ is a complete regular local $\kk$-algebra of Krull dimension $\krdim$, that is, $\Rx$ is the ring of formal power series $\kk\llbracket x_1,\ldots x_{\krdim} \rrbracket$ if $\krdim > 0$ respectively $\Rx$ is the field $\kk$ otherwise.
	We will denote by $\KK$ the field of fractions of the ring $\Rx$.            
\item  $\A$ is an \emph{$\Rx$-order}, that is, which means that there is a ring homomorphism $\Rx \to \A$ which factors through the center of $\A$ and $\A$ is finitely generated and free  as an $\Rx$-module.
\item $\A$ is \emph{isolated}, that is, 
for any non-maximal ideal $\px$ of $\Rx$ the ring $\A_{\px} \colonequals \A \otimes_{\Rx} \Rx_{\px}$ satisfies
$\gldim \A_{\px} = \dim \Rx_{\px}$.
\item Moreover, $\A$ is a \emph{Gorenstein algebra} in the sense of Iyengar and Krause \cite{IK}, 
which means that for any prime ideal $\mathfrak{p}$ of $\Rx$ the ring $\A_{\px}$ has finite injective dimension as left and right regular module.
	\item The base field $\kk$ is a splitting field for $\A$.
\end{itemize}
Since the ring $\A$ is a \emph{Noetherian $\Rx$-algebra} (that is, $\A$ is finitely generated as an $\Rx$-module), the ring $\A$ is Noetherian and semiperfect (\cite{Lam}*{(23.13)}). For simplicity of the presentation, we assume also the following.
\begin{itemize}[label=--]
	\item $\A$ is basic and ring-indecomposable.
	\end{itemize}

\subsection{Derived categories}
\label{subsec:main-setup}
Associated to $\A$ we will consider several triangulated categories.
\begin{itemize}[label=--]
	\item The unbounded derived category $\D(\Md \A)$ of $\A$-modules.
	For any objects $X$, $Y$ from $\D(\Md \A)$ we denote their morphism space by $\HomD(X,Y)$. 
	\item Its full subcategory given by the right-bounded derived category $\D^-(\md \A)$ of finitely generated $\A$-modules, which can be identified with the homotopy category $\Hot^-(\proj \A)$ 
of complexes of finitely generated projective $\A$-modules.
	\item The \emph{singularity category} $\Dsg(\A)$ which is defined as the Verdier quotient of the bounded derived category $\Db(\md \A)$ by the subcategory $\per(\A)$ of perfect complexes. 
	\item The full subcategory $\perfd(\A)$  given by perfect complexes $X$ such that the homology $\Ho_i(X)$ is finite-dimensional at each degree $i \in \Z$. This subcategory is \emph{Hom-finite} in the sense that  for any two of its objects $X$, $Y$ the $\kk$-linear space $\HomD(X,Y)$ is finite-dimensional. 
\end{itemize}

\begin{thm}\label{thm:der-eq}
	Assume that the Noetherian $\Rx$-algebra $\A$ is \emph{derived equivalent} to a ring $\B$, that is, there is 
	an equivalence of triangulated categories
	$$
	\begin{td}
		\Phi \colon \D(\Md \A) \ar{r}{\sim} \& \D(\Md \B).
	\end{td}$$
Then $\B$ is a Noetherian $\Rx$-algebra and $\Phi$	induces equivalences of triangulated categories
	\begin{align*}
		\begin{array}{ccc}
			\begin{td}
				\D^-(\md \A) \ar{r}{\sim} \& \D^-(\md \B)\end{td},& 		\begin{td} \per(\A) \ar{r}{\sim} \& \per(\B) \end{td}, & \multirow{2}{*}{\begin{td}	\Dsg(\A) \ar{r}{\sim} \& \Dsg(\B) \end{td}} \\
			\begin{td} \Db(\md \A) \ar{r}{\sim} \& \Db(\md \B)\end{td}, & 	\begin{td}	\perfd(\A) \ar{r}{\sim} \& \perfd(\B) \end{td} 
		\end{array}
	\end{align*}
	%
	%Vice versa, any equivalence of one of the pairs of subcategories restricts respectively extends to an equivalence of any other pair.
\end{thm}
\begin{proof}
	%		Let $\Phi$ denote the initial equivalence. 
	The ring  $\B$ is a Noetherian $\Rx$-algebra by \cite{Rickard}*{Proposition~9.4}.
	The equivalence
	$\Phi$ restricts to an equivalence of the subcategories 
	$\D^-(\Md \blank)$ by \cite{Krause1}*{Proof of Theorem~9.2.4},
	to $\D^-(\md \blank)$ and $\Db(\md \blank)$ by \cite{Rickard}*{Proposition~8.1},
	and to $\per (\blank)$ by \cite{Rickard}*{Section 6} or \cite{Krause1}*{Theorem~9.2.4}.
	As  observed in \cite{Rickard3}*{Corollary 2.2}, there is an equivalence of singularity categories by the universal property of the Verdier quotient.
	
	Note that a perfect complex $X$ of $\A$
is an object of $\perfd(\A)$ if and only if
	for any tilting complex $T$ of $\A$
	it holds that $\sum_{i\in\Z}\dim \HomD(T,X[i]) < \infty$. This implies that $\Phi$ and $\Phi^{-1}$ restrict to  $\perfd(\blank)$ as well.
\end{proof}
%\subsection{Bounded derived categories}

\subsection{The relative Serre functor}

The structure morphism $\Rx \to \A$ gives rise to the $\A$-bimodule
$\omega \colonequals \Hom_{\Rx}(\A,\Rx)$, which is called the \emph{canonical bimodule} and induces the functors
\begin{align}\label{eq:nak}
	\nu \colonequals \omega \lotimes \blank, 
	\quad  \SD \colonequals \nu \circ [\krdim] \colon \D(\Md \A) \longrightarrow \D(\Md \A).
\end{align}
The functor $\nu$ is called the \emph{derived Nakayama functor}, whereas the functor $\SD$ is a \emph{relative Serre functor} in the following sense.
\begin{enumerate}[label=(\alph*)]
	\item Since $\A$ is a Gorenstein algebra, the functor $\SD$ is an equivalence by \cite{IK}*{Theorem~4.5}.
	\item 
	By Theorem~\ref{thm:der-eq}, the functor $\SD$ restricts to an auto-equivalence of the 
	Hom-finite subcategory $\perfd(\A)$. This restriction of $\SD$ is the Serre functor on the category $\perfd \A$ by \cite{IR}*{Theorem~3.7}.
%	that this restriction is a Serre functor.
%	For any complex 
%	$P$ from $\per(\A)$
%	and
%	any complex $F$ from $\Db(\fdmod \A)$  
%	there is a bifunctorial isomorphism 
%	\begin{align*}
%		\HomD(P,F) \cong \MD \HomD(F,\SD(P)),\quad \text{where }	\MD \colonequals \Hom_{\kk}(\blank,\kk)\,.
%	\end{align*}
%The functor $\SD$ restricts to a \emph{Serre functor} on the Hom-finite subcategory $\perfd(\A)$ by \cite{IR}*{Theorem~3.7}.
%, that is, for any objects
%$X$ and $Y$ from $\perfd(\A)$
%	and
%	any complex $F$ from $\Db(\fdmod \A)$  
%	there is a bifunctorial isomorphism 
%	\begin{align} \label{eq:S-dual}
%			\HomD(X,Y) \cong \MD \HomD(Y,\SD(X)),\quad \text{where }	\MD \colonequals \Hom_{\kk}(\blank,\kk)\,.
%		\end{align}
%	This statement follows from \cite{IR}*{Theorem~3.7}.
\end{enumerate} 
We refer to \cites{Ginzburg, Keller, van-den-Bergh} for similar results on properties of the functor $\SD$.
%\begin{rmk}
%	\todo{new comment}
%	By \cite{Ooishi} $\kk$-linear duality can be viewed as \emph{Matlis duality}, that is, there is an isomorphism of the contravariant functors 
%	$$
%	\MD = \Hom_{\kk}(\blank,\kk) \cong
%	\Hom_\Rx(\blank, E_{\kk})  \colon
%	\begin{td} 
%		\fdmod \Rx \ar{r}{\sim}\& \fdmod \Rx 
%	\end{td}
%	$$
%	on the category of finite-dimensional $\Rx$-modules, 
%	where $E_{\kk}$ denotes the injective hull of the field $\kk$ viewed as an $\Rx$-module. 
%\end{rmk}

%\subsection{Auslander--Reiten translation}
%\todo{specialize}
Consequently, the Auslander-Reiten translation on the category $\perfd \A$ is given by 
%Serre functor gives rise to an equivalence 
%of $\perfd(\A)$.
\begin{align*}
	\tau  \colon    \perfd(\A) \overset{\sim}{\longrightarrow} \perfd(\A), \qquad
	X \longmapsto \tau(X) \colonequals 
	\nu (X)[\krdim -1] = 
	\SD(X)[-1].
\end{align*}

\subsection{Serre functor for the singularity category}
\label{subsec:dsg}

%According to 
%it may identified with the homotopy category of acyclic complexes.

As $\A$ is a Gorenstein $\Rx$-algebra, \cite{IK}*{Proposition~6.7} implies that the derived Nakayama functor gives rise to 
an auto-equivalence of the singularity category
\begin{align*}
	\begin{td}
		\bar{\nu} \colon 
		\Dsg(\A) \ar{r}{\sim} \& \Dsg(\A) \& X \ar[mapsto]{r} \& \omega \lotimes X
	\end{td}
\end{align*}
Since the $\Rx$-order $\A$ is isolated, the $\kk$-linear category $\Dsg(\A)$ is Hom-finite 
and admits a Serre functor which is given by the composition $\bar{\nu} \circ [d-1]$ according to
a theorem by Buchweitz \cite{Buchweitz}*{Theorem~7.7.5}.

%The $\Rx$-order $\A$ is isolated if the finite-dimensional $\KK$-algebra $\KK \otimes_{\Rx}\A$ is semisimple.
%In this case, the $\kk$-linear category $\Dsg(\A)$ has finite-dimensional morphism spaces
%and admits a Serre functor given by the composition $\nu \circ [d-1]$.

\subsection{Derived invariants of orders}
                   
Since the ring $\A$ is semiperfect, its unit $1_{\A}$ can be written as the sum $e_1 + e_2 + \ldots + e_n$ of finitely many primitive orthogonal idempotents. Any indecomposable projective $\A$-module is isomorphic to $P_i \colonequals \A e_i$ for some index $1 \leq i \leq n$.
\begin{dfn}\label{dfn:Cartan}
	The matrix $C_{\A}  \in \Mat_{n \times n}(\Z)$ with
	$(C_{\A})_{ij} \colonequals \rk_{\Rx} \Hom_{\A}(P_i,P_j) = \rk_{\Rx} (e_i \A e_j)$ for all $1 \leq i,j \leq n$
	is called the \emph{Cartan matrix} of the $\Rx$-order $\A$.
\end{dfn}
%This notion is well-defined as 
%Since the $\Rx$-algebra $\A$ is a free module of finite rank over $\Rx$ the following notion is well-defined.

We collect some basic facts on derived equivalent orders.
\begin{prp}\label{prp:der-ord}
Let $\B$ be an $\Rx$-order and
	\begin{align*}
		\begin{td}
		\Phi\colon \D(\Md \A) \ar{r}{\sim} \& \D(\Md \B)\end{td}
		\end{align*} 
%and that $\B$ is an $\Rx$-order.
an equivalence of triangulated categories.
%an equivalence of triangulated categories.
	Then the following statements hold.
	\begin{enumerate}
		\item \label{der-ord1} For any object $X$ in $\D^-(\Md \A)$ there is an isomorphism
 $\Phi \, \nu_{\A}(X) \cong \nu_{\B} \, \Phi(X)$.
% \item There is a matrix $P \in \GL_n(\Z)$ such that $P^T C_{\A} P = C_\B$.
 \item \label{der-ord2} The finite-dimensional $\KK$-algebras $\KK \otimes_{\Rx} \A$ and $\KK \otimes_{\Rx} \B$ are derived equivalent.
  \item \label{der-ord3} The matrices $C_{\A}$ and $C_{\B}$ are congruent over $\Z$.
 \end{enumerate}
\end{prp}
\begin{proof}
	\eqref{der-ord1} follows from the proof of 
	the corresponding statement in the setup of finite-dimensional algebras \cite{Rickard2}*{Proposition~5.2}.
%Following \cite{Rickard2}*{Proposition 5.2}
% By \cite{Rickard2}*{Corollary~3.5} we may assume that $\Phi$ is a standard derived equivalence. 
%		Then $\Phi$ induces an equivalence $$\Phi^* \colon \begin{td} \Db(\Md \A \otimes_{R} \A^{\op})) \ar{r}{\sim} \& \Db(\Md(\B \otimes_R \B^{\op}) \end{td}$$
%		According to Zimmermann \cite{Zimmermann}*{Lemma~1} there is an isomorphism
%		$\Phi^*(\omega_{\A}) \cong \omega_{\B}$. At this point, 
%		
%		The proof of statement \eqref{der-ord1} for finite-dimensional algebras \cite{Rickard2}*{Proposition~5.2} extends to the setup above. which shows \eqref{der-ord1}
An application of \cite{Rickard2}*{Theorem~2.1} yields \eqref{der-ord2}
as well as that the finite-dimensional $\kk$-algebras
		$A \colonequals \kk \otimes_{\Rx} \A$ and $B \colonequals \kk \otimes_{\Rx} \B$ are derived equivalent.
		By \cite{Zimmermann}*{Proposition~6.8.9} there is a matrix $P \in \GL_n(\Z)$ such that $P^T C_A P = C_{B}$.
		Since $C_{\A} = C_A$ and $C_{\B} = C_{B}$, this shows \eqref{der-ord3}.
%		 Using \cite{Zimmermann}*{Proposition~6.8.9}
%		it follows that there is a matrix $P \in \GL_n(\Z)$ such that $P^T C_A P = C_{A'}$, which implies \eqref{der-ord3}.
	\end{proof}
Motivated by statement~\eqref{der-ord1} above, we 
will call an object in $\D^-(\Md \A)$
\emph{fractionally Calabi--Yau}, if it satisfies the conditions 
of Definition~\ref{dfn:frac-CY} with respect to the relative Serre functor $\SD$. 

%fractionally Calabi--Yau objects also for the category $\D^-(\Md \A)$,
% despite that $\SD$ is not a Serre functor on that e whole category $\D^-(\Md \A)$.
%Similarly 
%extend the notion of fractionally Calabi-Yau objects in the sense of Definition~\ref{dfn:frac-CY} to the category $\D^-(\Md \A)$, although
%$\SD$ is not a Serre functor on the whole category $\D^-(\Md \A)$.

\begin{cor}\label{cor:CYdim}
In the setup of Proposition~\ref{prp:der-ord}, an object $X$ in $\D^-(\Md \A)$ is fractionally Calabi-Yau if and only if the object $\Phi(X)$ in $\D^-(\Md \B)$ is fractionally Calabi-Yau.
In this case, it holds that $\CYdim X = \CYdim \Phi(X)$.
	\end{cor}
%will call an 
%object $X$ in $\D^-(\Md \A)$ a fractionally Calabi-Yau object 
%if $\SD^m(X) \cong X[n]$ for certain $(m,n)$ although the functor $\SD$ is not a Serre functor on the whole category $\D^-(\Md \A)$.

By the next statement, the main setup implies  restrictions on possible Calabi-Yau dimensions.
%\todo{placement}

\begin{lem} \label{lem:frac-bound}
	Let $X$ be a fractionally Calabi-Yau object in $\D^-(\Md \A)$ of dimension $(m,n)$.
	Then $m \geq \krdim n$, where $\krdim = \dim \Rx$.
\end{lem}
\begin{proof}
	For any object $Y \in \D^-(\Md \A)$ set $\mu(Y) \colonequals 
	\min \{ 
	i \in \Z \mid \Ho_i(Y) \neq 0
	\}$.
	For any $i,j \in \Z$ with $\mu(\omega) \geq i$ and $\mu(X) \geq j$ there is a $\A$-linear isomorphism
	$\Ho_i(\omega) \otimes_{\A} \Ho_j(X) \cong \Ho_{i+j}(\nu(X))$
	by the \emph{Künneth trick} (see e.g. 
	\cite{Yekutieli}*{Lemma~13.1.36}).
	
	As $\mu(\omega) = 0$, it follows that $\mu(\nu(X)) \geq \mu(X)$.
	Therefore $\mu(X) + (m-\krdim n) = \mu(X[m-\krdim n]) = \mu(\nu^n(X)) \geq \mu(X)$.
\end{proof}

\subsection{A stable $t$-structure on the right-bounded derived category}

In this subsection, we recall the construction of a stable $t$-structure on the right-bounded derived category $\D^-(\md \A)$ and 
some of its equivalent formulations.
We recall the former notion, coined by Miyachi \cite{Miyachi}*{Section~2}, for the convenience of the reader.
\begin{dfn}
	A \emph{stable $t$-structure} of a triangulated category $\Dcat$ is given by a pair $(\aisle, \coaisle)$ of full triangulated subcategories of $\Dcat$ such that $\HomDc(\aisle, \coaisle) = 0$ and for any object $X$ from $\Dcat$ there is a triangle
	\begin{align}
		\label{eq:triangle0}
		\begin{td}
			U_X \ar{r} \& X \ar{r} \& V_X \ar{r} \& U_X [1]
		\end{td}
	\end{align}
	with $U_X \in \aisle$ and $V_X \in \coaisle$.	
\end{dfn}

Any idempotent 
$e$ of the Noetherian ring $\A$ gives rise to a subring and a quotient ring
\begin{align*}
	\begin{td}
		\Ainf \colonequals e \A e
		\ar[hookrightarrow]{r}
		\& \A \ar[->>]{r} \& \AI \colonequals \A/\A e\A
	\end{td}
\end{align*}
as well as a diagram of abelian categories and functors
\begin{align}
	\label{eq:ab-recol}
	\begin{tikzcd}[ampersand replacement=\&, cells={outer sep=2pt, inner sep=2pt},
		column sep=2cm, 	labels={rectangle, font=\small}
		] 
		\md \Ainf		
		\ar[yshift=5pt,rightarrowtail]{r}{\mathsf{f} \colonequals \A e \underset{e\A e}{\otimes} \blank }
		\& 
		\ar[yshift=5pt]{r}{ \mathsf{q} \colonequals \AI \underset{\A}{\otimes} \blank}
		\md \A 
		\ar[yshift=-5pt]{l}{\mathsf{j} \colonequals e(\blank)} 
		\&
		\md \AI
		\ar[yshift=-5pt,rightarrowtail]{l}{
			\mathsf{r} \colonequals \mathstrut_{\A}(\blank)}
	\end{tikzcd}
\end{align}
such that $\mathsf{q} \dashv \mathsf{r}$, that is, $\mathsf{q}$ is left adjoint to $\mathsf{r}$, $\mathsf{r} \mathsf{q} \cong \id$, $\mathsf{f} \dashv \mathsf{j}$ and $\mathsf{j} \mathsf{f} \cong \id$.

In the following, let
$\Dh(\md \A)$ denote the full subcategory of complexes $X$ in $\D^-(\md \A)$
satisfying $e \Ho_d(X)= 0$, that is, 
$\Ho_d(X) \in \mathstrut_{\A}\Hcat$  for any $d \in \Z$.

\begin{prp}\label{prp:recol}
	In the notations above, 
	%	let $e \in \A$ be an idempotent.
	%	
	%	For the gentle order $\A$ and the idempotent $e$ chosen in \eqref{eq:idem} the following statements hold.
	the following statements hold.
	\begin{enumerate}
		\item\label{recol1} 
		There are fully faithful functors of triangulated categories
		\begin{align*}
			\begin{td}	\D^-(\md \Ainf)  
				\& 
				\D^-(\md\A) \ar[leftarrowtail]{l}[swap]{\jtop}
				\& \ar[rightarrowtail]{l}[swap]{\imid} \Dh(\md\A)
			\end{td}
		\end{align*}	
		where $\imid$ denotes the inclusion functor, $\jtop$ denotes the left-derived functor of $\mathsf{f}$,
		and the pair $(\aisle,\coaisle) = (\im \imid,\im\jtop)$ forms a stable $t$-structure of $\D^-(\md \A)$.
		\item \label{recol2}
		More precisely, there is a diagram of triangulated categories and functors
		\begin{align}
			\label{eq:recol}
			\begin{td}
				\D^-(\md \Ainf)
				\& \ar[yshift=5pt,twoheadrightarrow]{r}{\itop} \D^-(\md\A) 
				\ar[yshift=-5pt,twoheadrightarrow]{l}{\jmid} \ar[yshift=5pt,leftarrowtail]{l}[swap]{\jtop}
				\& \ar[yshift=-5pt,rightarrowtail]{l}{\imid}  \Dh(\md\A)  
			\end{td}
		\end{align}	
		forming a \emph{left recollement},
		which means that
		$\jtop \dashv \jmid$, $\jmid \jtop \cong \id$, $\itop \dashv \imid$, $\itop \imid \cong \id$, $\itop \jtop =	\jmid \imid =  0$
		%		\begin{align*}
			%						\jtop \dashv \jmid, &&			\jmid \jtop \cong \id, && 
			%			\itop \dashv \imid, && \itop \imid \cong \id,  &&
			%		\itop \jtop =	\jmid \imid =  0
			%		\end{align*}
		and that for each object $X$ from $\D^-(\md \A)$ there is a triangle
		\begin{align}\label{eq:triangle}
			\begin{td} \jtop\jmid(X) \ar{r}{\varepsilon_X} \& X \ar{r}{\eta_X} \&\imid\itop(X) \ar{r}{\delta_X}\& \jtop \jmid(X)[1]
			\end{td}  
		\end{align}
		where $\varepsilon_X$ denotes the counit of the adjunction $\jtop \dashv \jmid$ and $\eta_X$ the unit of the adjunction $\itop \dashv \imid$, respectively.
		The functor $\jmid$ is the derived functor of the exact functor $\mathsf{j}$, while $\itop$ is defined via the triangle \eqref{eq:triangle} on objects, and similarly on morphisms.
		\item \label{recol3}
		In particular, the following statements hold.
		\begin{enumerate}
			%	\item It holds that  $\ker \itop \supseteq \im \jtop$  and  $\im \imid \subseteq \im \jmid$.
			%	\item The pair $(\aisle,\coaisle) = (\im \imid,\im\jtop)$ forms a stable $t$-structure of $\D^-(\md \A)$.
			\item \label{recol3a} For any $X \in \D^-(\md \A)$ and any triangle 
			given by the top row
			\begin{align*}
				\begin{td}
					U_X \ar{r} \ar[dashed]{d}{\phi} \& X \ar{r}\ar[equal]{d} \& V_X \ar{r} \ar[dashed]{d}{\psi} \& U_X [1] \ar[dashed]{d}{\phi[1]} \\
					\jtop\jmid(X) \ar{r}{\varepsilon_X} \& X \ar{r}{\eta_X} \&\imid\itop(X) \ar{r}{\delta_X}\& \jtop \jmid(X)[1]
				\end{td}
			\end{align*}
			with $U_X \in \aisle$ and $V_X \in \coaisle$,
			there are unique isomorphisms $\phi,\psi$ making the diagram above commutative.
			% is isomorphic to~\eqref{eq:triangle}.
			\item \label{recol3b} 
			Set  $B \colonequals \A e$, $f \colonequals 1-e$ and $S \colonequals \head \A f$.
			Then 
			$\aisle
			= \ker \itop = \mathstrut^{\perp}\coaisle = \mathstrut^{\perp} S$
			and
			$\coaisle = \ker \jmid = \aisle^{\perp} = B^{\perp}$,
			where each orthogonal subcategory is defined with respect to the category $\D^-(\md \A)$.
			%	\item If $\A$ is basic, it holds that $\aisle = \mathstrut^{\perp} S$ with $S \colonequals \top \A (1-e)$.
		\end{enumerate}
	\end{enumerate}
\end{prp}
\begin{proof}
	The statements in \eqref{recol2} were shown in \cite{Parshall}*{(1.7)}.
	It is straightforward to check that any left recollement gives rise to a stable $t$-structure as claimed in~\eqref{recol1}.
	A proof of \eqref{recol3a} can be found in \cite{Krause2}*{Proposition~4.11.2}.
	
	It remains to show \eqref{recol3b}.
	The condition
	$\Hom(\aisle, \coaisle ) = 0$ is equivalent to $\aisle \subseteq \mathstrut^{\perp}\coaisle$ and $\coaisle \supseteq \aisle^{\perp}$.  The condition $\jmid \imid = \itop \jtop = 0$ translates into
	$\aisle \subseteq \ker \jmid$ and $\coaisle \subseteq \ker \itop$.
	The converse four inclusions can be shown using~\eqref{eq:triangle0}. 
	
	%The equality  holds because there is an 
	%$\Rx$-linear 
	%isomorphism
	For any complex $P \in \Hot^-(\proj \A)$ and any integer $i \in \Z$
	there is an isomorphism
	$\Hom_{\Hot(\proj \A)}(B[i], P) \cong e \Ho_i(P)$, which yields the equality
	$B^{\perp}= \coaisle$.
	
	%	\todo{improve}
	Since $\A$ is basic, it holds that $\A e \A f \subseteq \rad \A f$ 	
	and thus $S \in \coaisle$, which leads to $\aisle = \mathstrut^{\perp} \coaisle \subseteq \mathstrut^{\perp} S$. 
	For any minimal complex $C \in \Hot^-(\proj \A)$, $d \in \Z$ and an indecomposable projective $\A$-module $P_i$ 
	the total number of indecomposable summands of the form $P_i$ at degree $d$ in $C$ is equal to $\dim \Hom_{\Hot(\proj \A)}(C,S_i [d])$, where $S_i \colonequals \head P_i$.
	Using that $\A$ is basic and semiperfect, it follows that
	%This observation implies the equality
	$\aisle = \mathstrut^{\perp} S$, which completes the proof of \eqref{recol3b}.
\end{proof}
We note that any stable $t$-structure
$(\mathcal{U},\mathcal{V})$ 
is also called a \emph{weak semiorthogonal decomposition}  \cite{Orlov}*{1.1}
and that the composition $\jtop  \jmid$ of functors in \eqref{eq:recol} is a \emph{localization functor}
in the sense of \cite{Krause2}*{4.9}. 
Stable $t$-structures, left recollements and these two notions yield equivalent setups.

%%\input{tex/02-basics-gentle}
%
%\newpage
\section{Combinatorial notions of quivers underlying gentle orders}
\label{sec:quiver}
From now on, we specialize the ring $\Rx$ to be the ring of formal power series
$\kk \llbracket x \rrbracket$  over a field $\kk$.
In this section, we consider a finite quiver with relations $(Q,I)$.

\subsubsection*{Conventions on quivers} 

As usual, $Q$ is given by data $(Q_0,Q_1,s,t)$ with a finite set of vertices $Q_0$, a finite set of arrows $Q_1$ and two maps $s,t \colon Q_1 \to Q_0$ assigning to each arrow  its start and target.
We do not assume that each cyclic path in $Q$ is contained in the ideal $I$, but we do assume that $I$ is generated by $\kk$-linear combinations of paths of length at least two.
%For an arrow $\alpha \in Q_1$ we denote by $s(\alpha)$ its start and by $t(\alpha)$ its target.
 Paths are composed from right to left.
%A {path} in $Q$ may be contained in the ideal $I$. 
Given a path $p = \alpha_n \ldots \alpha_2 \alpha_1$ with $n \geq 1$ arrows we set $t(p) \colonequals t(\alpha_n)$ and $s(p) \colonequals s(\alpha_1)$, and denote by $\ell(p) \colonequals n$ its length.
 
\subsection{Gentle quivers, threads and cycles}\label{subsec:gentle}
A vertex $i \in Q_0$ will be called 
\begin{itemize}[label=--]
	\item a \emph{transition vertex} if there is precisely one arrow $\alpha$ ending at $i$, there is precisely one arrow $\beta$ starting at $i$, and $\beta \alpha \notin I$; the local situation around such a vertex is shown in the diagram on the left in \eqref{eq:vert};
\item  a \emph{crossing vertex} if there are precisely two different arrows $\alpha_1, \alpha_2$ ending at $i$, there are precisely two different arrows $\beta_1, \beta_2$ starting at $i$, and for any $i,j \in \{1,2\}$ it holds that $\beta_j \alpha_i \in I$ if and only if $i = j$; such a vertex is depicted on the right in \eqref{eq:vert} with the zero relations indicated by dotted lines.
\end{itemize}
The set of transition vertices and that of crossing vertices is denoted by $Q_0^t$ and $Q_0^c$, respectively.
%The local diagram around a transition vertex is shown on the left, that around a crossing vertex on the right in the diagram below.
%In the depiction of quivers, zero relations are indicated by dotted lines.
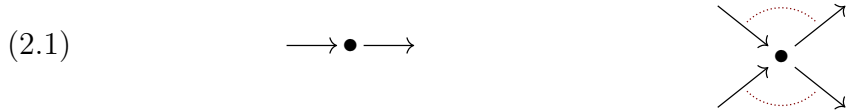
\begin{align}\label{eq:vert}
		\begin{tikzcd}[row sep=scriptsize, column sep=scriptsize, cells={outer sep=1pt, inner sep=1pt}, baseline=0.05cm, ampersand replacement=\&]  
		\phantom{\bt} \ar{r} \& {\bt} \ar{r}  \&\phantom{\bt}
		\end{tikzcd}
&&
		\begin{tikzcd}[row sep=scriptsize, column sep=scriptsize, cells={outer sep=1pt, inner sep=1pt}, baseline=0.05cm, ampersand replacement=\&]  
		\phantom{\bt} \ar[rd,	""{name=a, inner sep=0, midway}] \& \& \phantom{\bt} \\
		\& {\bt} \ar[ru, ""{name=b, inner sep=0, midway}] 
		\ar[rd, ""{name=e, swap, inner sep=0, midway}] \\
	\phantom{\bt} \ar[ru, ""{name=c, inner sep=0, swap, midway}] \& \& \phantom{\bt}
	\arrow[densely dotted, dash, , color={darkred}, bend left=45, from=a, to=b] 
	\arrow[densely dotted, color={darkred}, dash, bend right=45, from=c, to=e]
	\end{tikzcd}
	\end{align}
The quiver $(Q,I)$ is called 
\begin{itemize}[label=--]
	\item \emph{$2$-regular gentle} if $Q_0 = Q_0^c$ and the ideal $I$ is generated by the zero relations at crossing vertices,
	\item \emph{gentle} if it can be obtained from a $2$-regular gentle quiver by deleting any choice of arrows
and removing all zero relations involving deleted arrows.
\end{itemize}
In a gentle quiver $(Q,I)$, a path $p = \alpha_n \ldots \alpha_2 \alpha_1 \notin I$  with different arrows is called
\begin{itemize}[label=--]
	\item a \emph{permitted cycle} if $p$ is a cyclic path such that $\alpha_1 \alpha_n \notin I$,
	\item  
	a \emph{permitted thread} if 
	it is not a permitted cycle and
	there 
	is neither an arrow $\alpha_{n+1} \in Q_1$ 
	with $\alpha_{n+1} \alpha_n \notin I$
	nor an arrow $\alpha_0 \in Q_1$ with $\alpha_1 \alpha_0 \notin I$.
\end{itemize}
A path $f = \beta_n \ldots \beta_2 \beta_1$ a path with $n \geq 1$ different arrows
such that
$\beta_{i+1} \beta_i \in I$ for each index $1\leq i < n$ is called
\begin{itemize}[label=--]
	\item a \emph{forbidden cycle} if $f$ is a cyclic path and $\beta_1 \beta_n\in I$.
	\item  a \emph{forbidden thread} if 
	it is not a forbidden cycle and
	there 
	is neither an arrow $\beta_{n+1} \in Q_1$ 
	with $s(\beta_{n+1}) = t(\beta_n)$ and $\beta_{n+1} \beta_n \in I$ and 
	nor an arrow $\beta_0 \in Q_1$ with $s(\beta_1) = t(\beta_0)$ and $\beta_1 \beta_0 \in I$.
\end{itemize}
%The subset of arrows appearing in forbidden cycles and that of arrows appearing in forbidden threads will be denoted by $Q_1^{fc}$ and $Q_1^{ft}$, respectively.

\begin{rmk} We give a few remarks concerning the terminology.
	\begin{itemize}[label=--]
		\item Permitted threads and forbidden threads may be cyclic paths.
		\item 
%		For a forbidden thread or a forbidden cycle $f$  it holds that $f \in I$ if and only if $\ell(f) \geq 2$. 
		Somewhat counter-intuitively, any arrow $\beta$ in $Q$ between transition vertices is also considered as a forbidden thread although $\beta \notin I$. Similarly, a loop $\beta$ in $Q$ with $\beta^2 \in I$ is considered a forbidden cycle although $\beta \notin I$.  	
	\end{itemize}
%	A forbidden thread may be a cyclic path, but it cannot be a forbidden cycle.
\end{rmk}
Let $(Q,I)$ be a gentle quiver. Let $\A$ denote the $J$-adic completion of the path algebra $\kk  Q/I$ with respect to the ideal $J$ generated by all arrows.
%paths of positive length.
 Then the following statements hold.
\begin{itemize}[label=--]
	\item The $\kk$-algebra $\A$ is finite-dimensional if and only if $(Q,I)$ has no permitted cycles.
%	In this case, $\A$ is the path algebra $\kk Q/I$, so taking completion can be omitted.
	\item The global dimension of $\A$ is finite if and only if $(Q,I)$ has no forbidden cycles.
\end{itemize}

\begin{rmk}\label{rmk:KRS}
	Taking arrow ideal completion ensures that
	$\D^{-}(\md \A)$ has the Krull-Remak-Schmidt property  in the sense that any object is a countable direct sum of indecomposable objects with local endomorphism rings \cite{BD}*{Corollary~A.6}.
%	In contrast to $\A$, 
%	If the path algebra $\kk Q/I$ is infinite-dimensional, then the endomorphism ring of each of its indecomposable projectives is not local.
\end{rmk}
%\todo{better motivation}

\subsection{Definition of gentle orders}

\begin{lem}\label{lem:go}
For any gentle quiver $(Q,I)$, the following conditions are equivalent.
\begin{enumerate}
	\item \label{go1} Every vertex in $Q$ is a transition vertex or a crossing vertex.
%	$Q_0 = Q_0^t \cup Q_0^c$.
	\item \label{go2} The quiver $(Q,I)$ has no permitted threads.  
	\item \label{go3} Every arrow in $Q$ is contained in a permitted cycle.
		\item \label{go4} The arrow ideal completion $\A$ of the path algebra $\kk Q/I$ has an $\Rx$-order structure.
%	\item \label{go5} Any marked point incident to several arcs does not lie on the boundary of $\Sigma_{(Q,I)}$. 
\end{enumerate}
If $(Q,I)$ is connected, $\A$ satisfies all assumptions of the setup in Subsection~\ref{subsec:setup}.
\end{lem}
\begin{proof}
	\eqref{go1} $\Rightarrow$ \eqref{go2} Assume that \eqref{go1} holds. Then for any (possibly stationary) path $p \notin I$ in $Q$ there is an arrow $\alpha \in Q_1$ such that $\alpha p \notin I$, which shows \eqref{go2}. 
	
	The implication \eqref{go2} $\Rightarrow$ \eqref{go3} is true because in a gentle quiver any arrow is either contained in a permitted cycle or in a permitted thread.
	
	\eqref{go3} $\Rightarrow$ \eqref{go1} 	
	The assumption \eqref{go3} excludes all possibilities other than \eqref{eq:vert} around a vertex in a gentle quiver, which are given as follows.
	
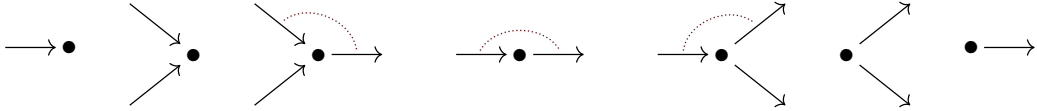
\begin{align*}
		\begin{tikzcd}[row sep=0.2cm, column sep=scriptsize, cells={outer sep=1pt, inner sep=1pt}, baseline=0.05cm, ampersand replacement=\&]  
			\phantom{\bt} \arrow[r, ""{inner sep=0, near end}] \& \bt 
		\end{tikzcd}
	&&
	\begin{tikzcd}[row sep=scriptsize, column sep=scriptsize, cells={outer sep=1pt, inner sep=1pt}, ampersand replacement=\&] 
		\phantom{\bt} \arrow[rd, ""{name=A, inner sep=0, near start}] \&   \\
		\& \bullet \\
		\phantom{\bt} \arrow[ru, ""{name=D, inner sep=0, near start}] \& \end{tikzcd}
	&&
\begin{tikzcd}[row sep=scriptsize, column sep=scriptsize, cells={outer sep=1pt, inner sep=1pt}, ampersand replacement=\&]  
	\phantom{\bt} \arrow[rd, ""{name=A, inner sep=0, midway}] \& \&  \\
	\& \bullet \arrow[r, ""{name=B, inner sep=0, midway}]  \& \phantom{\bt} \\
	\phantom{\bt} \arrow[ru, ""{name=D, inner sep=0, midway}] \& 
	\arrow[densely dotted, dash, bend left=60, color={darkred}, from=A, to=B] \& \end{tikzcd}
&&
\begin{tikzcd}[row sep=scriptsize, column sep=scriptsize, cells={outer sep=1pt, inner sep=1pt}, ampersand replacement=\&]
	\phantom{\bt} \& \\
	\phantom{\bt} \arrow[r, ""{name=E, inner sep=0.5, pos=0.45}] \& \bullet \arrow[r, ""{name=F, inner sep=0.5, pos=0.55}] \& \phantom{\bt} 
	\\
	\phantom{\bt} \& 
	\arrow[densely dotted, dash, bend left=60, color={darkred}, from=E, to=F] \end{tikzcd}
&&
\begin{tikzcd}[row sep=scriptsize, column sep=scriptsize, cells={outer sep=1pt, inner sep=1pt}, ampersand replacement=\&] 
	\& \& \phantom{\bt} \\
	\phantom{\bt} \arrow[r, ""{name=A, inner sep=0, midway}] \& \bullet \arrow[ru, ""{name=B, inner sep=0, pos=0.45}]
	\arrow[rd, ""{name=D, inner sep=0, midway}] \& \\
	\& \& \phantom{\bt}
	\arrow[densely dotted, dash, bend left=60, color={darkred}, from=A, to=B] \end{tikzcd}
&&
\begin{tikzcd}[row sep=scriptsize, column sep=scriptsize, cells={outer sep=1pt, inner sep=1pt}, ampersand replacement=\&]  
	\&   \phantom{\bt} \\
	\bullet \arrow[ru, ""{name=A, inner sep=0, near start}] \arrow[rd, ""{name=D, inner sep=0, near start}] \& \\
	\& \phantom{\bt}   \end{tikzcd}
&&
		\begin{tikzcd}[row sep=scriptsize, column sep=scriptsize, cells={outer sep=1pt, inner sep=1pt}, baseline=0.05cm, ampersand replacement=\&]  
	{\bt} \arrow[r, ""{inner sep=0, near end}]  \& \phantom{\bt}
\end{tikzcd}
	\end{align*}
	The equivalence \eqref{go3} $\Leftrightarrow$ \eqref{go4} follows from the proof of \cite{GIK}*{Proposition~8.4} with the polynomial ring $\kk[x ]$ replaced by the power series ring $\Rx = \kk \llbracket x \rrbracket$.
	
The ring $\A$ is an isolated Gorenstein algebra by \cite{GIK}*{Lemma~8.3} and \cite{Krause1}*{Proposition~6.2.9}. Verifying the remaining properties is straightforward. 
\end{proof}
\begin{dfn}\label{dfn:go}
	A basic semiperfect ring $\A$ will be called a \emph{gentle order} if $\A$ satisfies
	any of the conditions of Lemma~\ref{lem:go}.
	\end{dfn}
%In this case, the ring $\A$ will be called a \emph{gentle order}.
An $\Rx$-order structure on a gentle order $\A$ is given by the unique $\kk$-algebra morphism
%\begin{align}
%	\label{eq:order-str}
	$\Rx \to \A$
	which maps the monomial $x \in \Rx$ to the sum of all permitted cycles in $\A$.
%	, \qquad
%	x \longmapsto \sum_{\alpha \in Q_1} pc({\alpha})
%\end{align}
%\begin{align}
%	\label{eq:order-str}
%\Rx = \kk \llbracket x \rrbracket \longrightarrow \A, \qquad
%x \longmapsto \sum_{\alpha \in Q_1} pc({\alpha})
%\end{align}
%where $pc(\alpha)$ denotes the unique permitted cycle beginning with the arrow $\alpha$.

\subsection{Combinatorial invariants of quivers underlying gentle orders}
\label{subsec:inv}

In the present subsection, we assume that $\A$ is a gentle order such that its underlying quiver $(Q,I)$ is connected.
We will say that a cyclic path $c$ in $Q$ is a \emph{rotation} of a cyclic path $c'$
if $c$ can be obtained from $c'$ by a cyclic permutation of its arrows.

\subsubsection*{Bicolorability parameter}
%\begin{dfn}\label{dfn:2color}
The gentle quiver $(Q,I)$ is \emph{bicolorable} if there is a map $f\colon Q_1 \to \{1,2\}$
such that for any two arrows $\alpha,\beta\in Q_1$ with $s(\beta) = t(\alpha)$
it holds that $\beta \alpha \in I$ if and only if $f(\beta) \neq f(\alpha)$.
For the gentle order $\A$ we define a binary parameter 
%\beg
%$\bc_{\A}$ setting
\begin{align}\label{eq:bc}
	\bc_{\A} \colonequals \begin{cases} 1 &\text{if $(Q,I)$ is bicolorable},\\
	0  & \text{otherwise}.
	\end{cases}
%	, and $\bc \A \colonequals 0$ otherwise.
	\end{align}
%\end{dfn} 

\subsubsection*{The number of permitted cycles}
%\subsection{Avella-Allamino Geiss invariants of type 2''}\label{subsec:AG3}
For an arbitrary arrow $\alpha \in Q_1$ there are the following notions.
\begin{itemize}[label=--]
	\item A unique arrow $\sigma(\alpha) \in Q_1$ with $\sigma(\alpha) \alpha \notin I$.
	\item A unique permitted cycle, denoted $pc(\alpha)$, starting with $\alpha$.
	%	\item The pair $(0,0)$ will be called the AG-invariant of $p(\alpha)$. This convention allows to be consistent with \cite{LP}.
\end{itemize}
As the set $Q_1$ is finite, the first assignment defines a permutation $\sigma \colon Q_1 \overset{\sim}{\to} Q_1$.
%A permitted cycle $pc(\alpha)$ will be called a \emph{rotation} of another forbidden cycle $pc(\alpha)$ if $\alpha$ lies in the $\sigma$-orbit of $pc(\beta)$. 
We will denote by $Q_1/\!\sim_{\sigma}$ the set of $\sigma$-orbits.
The cardinality of this set, denoted $$\pc_{\A} \colonequals \lvert Q_1/\!\sim_{\sigma}\rvert, $$
 is precisely the number of permitted cycles in $(Q,I)$ up to rotation.

\subsubsection*{Avella-Allaminos Geiss invariants of first type}

For any transition vertex $j \in Q_0^t$  there are the following notions.
\begin{itemize}[label=--]
	\item A unique forbidden thread $f_1$ with $s(f_1) = j$. Then $\kappa(j) \colonequals t(f_1) \in Q_0^t$.
	%	\item Similarly, there is a forbidden thread $f_2$ starting at $\kappa(i)$ and ending in a vertex denoted $\kappa^2(i)$. 
	\item 	A unique cycle $c(j)$, called \emph{the AG-cycle of the vertex $j$}, which starts in $j$ and is composed from distinct forbidden threads. More precisely, 
 $c(j)= f_{n(j)} \ldots f_2  f_1$ with 
 $t(f_p) = \kappa^{p}(j)$ for each index $1 \leq p \leq n(j) \colonequals \min \{ q \in \N^+ \mid \kappa^q(j) = j\}$. 	 
	\item The \emph{AG-invariant of the vertex $j$}, the pair $(m(j),n(j))$ where $m(j) \colonequals \ell(c(j))$.
%		, called the \emph{AG-invariant of the vertex $j$},

\end{itemize}
The first assignment defines a permutation $\kappa\colon Q_0^t \overset{\sim}{\longrightarrow} Q_0^t$.
%We will denote by $Q_0^t/\!\sim_\kappa$ the set of $\kappa$-orbits in $Q_0^t$.
%The cycle $c(j)$ will be called the \emph{AG-cycle of $j$}.
Any two vertices which lie in the same $\kappa$-orbit have the same AG-invariant.
The cardinality $n(j)$ of the $\kappa$-orbit of $j$ is precisely the number of transition vertices which appear in the cyclic path $c(j)$.
In particular, the inequality $m(j) \geq n(j)$ holds.

%, which becomes an equality if and only if $\A$ is hereditary.

%We will call a proper cyclic path $c$ in $Q$ an \emph{AG-cycle} if $c$ is a finite composition 
%%$f_n \ldots f_2 f_1$ 
%of pairwise different, forbidden threads.
%The presentation of an AG-cycle as a composition of forbidden threads is unique, and
%any AG-cycle starts and ends at a transition vertex.

%For any two vertices $i,j \in Q_0^t$ an AG-cycle $c(i)$ will be called a \emph{rotation} of $c(j)$ if the $i$ lies in the $\kappa$-orbit of $j$.

% Such vertices have the same AG-invariant.
%it holds that $(m(i),n(i)) = (m(j),n(j))$.
Let $j_1, j_2 \ldots j_b$ denote a complete set of representatives 
%in the set $Q_0^t/\!\sim_\kappa$ 
of $\kappa$-orbits in $Q_0^t$.
Their invariants give rise to the multiset of \emph{AG-invariants of first type}
\begin{align*}
	\AG_1(\A) \colonequals \{(m(j_a),n(j_a)) \mid  1\leq a \leq b\}.
\end{align*}
Its cardinality $b$ is precisely the number of AG-cycles in $(Q,I)$ up to rotation.

\subsubsection*{Avella-Allaminos Geiss invariants of second type}
Let $Q_1^{fc}$ denote the subset of arrows in $Q$ appearing in a forbidden cycle.
For any arrow $\beta \in Q_1^{fc}$ there are the following notions.
\begin{itemize}[label=--]
	\item A unique arrow $\varrho(\beta) \in Q_1^{fc}$ with $\varrho(\beta) \beta \in I$ and $s(\varrho(\beta)) = t(\beta)$.
	\item A unique forbidden cycle $fc(\beta)$ starting with $\beta$.
	We denote by $m(\beta)$ its length.
	\item The pair $(m(\beta),0)$, called the \emph{AG-invariant of $\beta$}.
\end{itemize}
The first assignment defines a permutation $\varrho\colon Q_1^{fc} \overset{\sim}{\to} Q_1^{fc}$. 
Including $0$ in the notation $(m(\beta),0)$ is consistent with the previous notion of AG-cycles of transition vertices, as there are no transition vertices in the forbidden cycle $fc(\beta)$.  
Similar to the previous definitions, we denote by 
$Q_1^{fc}/\!\sim_{\varrho}$ the set of $\varrho$-orbits of arrows in $Q_1^{fc}$.
Any two arrows in $Q_1^{fc}$ which lie in the same $\varrho$-orbit have the same AG-invariant.
%Rotations of forbidden cycles are defined in the same way as for permitted cycles.
%Rotation of forbidden cycles is defined in the same way as that of permitted cycles.

Choosing a complete set $\beta_1, \beta_2 \ldots \beta_q$ of representatives of $\varrho$-orbits in $Q_1^{fc}$, we denote 
the multiset of \emph{AG-invariants of second type} by
\begin{align*}
	\AG_{2}(\A) \colonequals \{(m(\beta_i),0) \mid 1 \leq i \leq q\}.
	%	&&
\end{align*}	
In particular, the cardinality of this multiset is given by the number of forbidden cycles in $(Q,I)$ up to rotation.

\subsection{Hereditary gentle orders}
\label{subsec:hered}
One of the simplest
%, but nevertheless important 
subclasses of gentle orders is given by that of hereditary ones. 
These are obtained as follows.

Let $Q'$ be the equioriented quiver of Euclidean type $\widetilde{\mathbb{A}}$ with $\ell \geq 1$ vertices.
The arrow ideal completion $\Hrd$ of the path algebra $\kk Q'$ 
is isomorphic to the matrix algebra
\begin{align}\label{eq:tri-hered}
	T_{\ell}(\Rx) &\colonequals \left\{ (p_{ij})_{i,j = 1 \ldots \ell} \in \Mat_{\ell\times \ell} (\Rx)
	\mid 
	%\parbox[t]{.6\textwidth}{$
		p_{ij} \in \mx \text{ for any indices }1 \leq  i <  j \leq \ell 
		%	$}
	\right \}
\end{align}
In other terms, the quiver $Q$ and the ring $\Hrd$ can be depicted as follows.
\begin{align*}
	%	\widetilde{\mathbb{A}}_{\ell-1}
	Q\colon
	\begin{tikzpicture}[baseline=(X.base)]
		\def \radius{1cm}
		\def \n {5}	
		\def \margin {8} % margin in angles, depends on the radius
		\foreach \s in {1,2,3}
		{
			\node[circle, inner sep=0.5pt] at ({360/\n * (\s - 1)}:\radius) {$\bt$};
			\draw[<-] ({360/\n * (\s - 1)+\margin}:\radius) 
			arc ({360/\n * (\s - 1)+\margin}:{360/\n * (\s)-\margin}:\radius) ;
		}
		\node[circle, inner sep=0.5pt] at ({360/\n * (3)}:\radius) {$\bt$};
		\draw[-,densely dotted] ({360/\n * 3+\margin}:\radius) 
		arc ({360/\n * 3+\margin}:{360/\n * 5-\margin}:\radius);
		\node[circle] at ({360/\n * 0 +36}:{1.33 * \radius}) {$\alpha_2$};
		\node[circle] at ({360/\n * 1 +36}:{1.33 * \radius}) {$\alpha_1$};
		\node[circle] at ({360/\n * 2 +36}:{1.33 * \radius}) (X) {$\alpha_\ell$};
		\node[circle] at ({360/\n * 0}:{1.25 * \radius}) {\scriptsize $3$};
		\node[circle] at ({360/\n * 1}:{1.25 * \radius}) {\scriptsize $2$};
		\node[circle] at ({360/\n * 2}:{1.25 * \radius}) {\scriptsize $1$};
		\node[circle] at ({360/\n * 3}:{1.25 * \radius}) {\scriptsize $\ell$};
	\end{tikzpicture}
	&&
	\Hrd \cong
	\begin{pmatrix}
		\Rx& \mx & \ldots &  \mx \\
		\Rx & \Rx &  \ddots &    \vdots \\
		\vdots &&\ddots&	\mx \\
		\Rx & \Rx & \ldots  &\Rx 
	\end{pmatrix}
\end{align*}
The $\Rx$-order structure on $\B$ 
proposed after Definition~\ref{dfn:go} can be identified with the diagonal embedding $\Rx \to T_{\ell}(\Rx)$, $r \mapsto r \mathbf{1}_\ell.$

%\begin{proof}
%	%	\eqref{h1} $\Rightarrow$ \eqref{h2}	is immediate.
%	\eqref{h1} $\Rightarrow$ \eqref{h3}
%	It holds that $\gldim \Hrd = \prdim S_j = 1$ for any $j \in Q_0$.
%	An $\Rx$-order structure on $T_{\ell}(\Rx)$ can be defined by \eqref{eq:order-str}.
%	
%	\eqref{h3} $\Rightarrow$ \eqref{h1} Given a hereditary $\Rx$-order $\Hrd$ as above, it follows that $\Hrd \cong T_{\ell}(\Rx)$ from the structure theorem for hereditary orders (\cite{Reiner}*{Theorem~39.14}).
%\end{proof}

To describe the canonical bimodule of $\B$, we introduce an automorphism and a $\B$-bimodule.
\begin{itemize}[label=--]
	\item
Let $\rho \colon \Hrd \overset{\sim}{\to} \Hrd$ denote the 
the unique $\kk$-algebra morphism 
satisfying
%\begin{align*}
$e_{j} \mapsto e_{j+1}$ and $\alpha_j \mapsto \alpha_{j+1}$
for any index $1 \leq j < \ell$ 
%	where we identify 
%	$n+1$ with $1$.
%\end{align*}
as well as $e_{\ell} \mapsto e_1$ and $\alpha_{\ell} \mapsto \alpha_1$.
%where we set $e_{\ell+1} \colonequals e_1$ and $\alpha_{\ell+1} \colonequals \alpha_1$.
Then $\rho$ is an $\Rx$-algebra automorphism of $\B$ of order $\ell$, which can be visualized as a clockwise rotation of the quiver $Q$.
\item 
Let $\mathstrut_{\rho } \Hrd$
denote the $\Hrd$-bimodule with underlying set $\Hrd$
such that  the left $\Hrd$-module action 
is  twisted by the automorphism $\rho$ in the sense that $a *_{\rho} x \colonequals \rho(a) x$  for any with $a,x\in \Hrd$, and
the right $\Hrd$-module action 
$\mathstrut_{\rho }\Hrd$ is the regular one.
\end{itemize}
By the next statement, the derived Nakayama functor
\begin{align*}
	%	\label{eq:nak}
\begin{td}	\nu \colonequals \omega \lotimes \blank \colon \D(\Md \Hrd) \ar{r}{\sim} \& \D(\Md \Hrd). \end{td}
\end{align*}
is induced by the automorphism $\rho$.
\begin{lem}\label{lem:hered-CY}
	%	The following statements hold.
	%	\begin{enumerate}
		%		\item 
		The canonical bimodule $\omega = \Hom_{\Rx}(\Hrd,\Rx)$ is isomorphic to 
		the twisted bimodule $\mathstrut_{\rho} \Hrd$. In particular, $\D(\Md \Hrd)$ is
		fractionally $(\ell,\ell)$-Calabi--Yau.
		%	\item $\D^-(\md \Hrd)$ has countably many indecomposable objects up to isomorphism.
		%	\end{enumerate}	
\end{lem}

\begin{proof}
	%	\begin{enumerate}
		%		\item 
		There are 
		isomorphisms of $\Hrd$-bimodules
		\begin{align*}		
			\omega = \Hrd^\vee
			\cong  
			\begin{pmatrix}
				\Rx^\vee&  \ldots  & \Rx^\vee & \Rx^\vee  \\
				\mx^\vee &  \ddots & & \vdots \\
				\vdots  & \ddots  &\Rx^\vee &    \Rx^\vee\\
				\mx^\vee &  \ldots & \mx^\vee & \Rx^\vee 
			\end{pmatrix}
			\cong  
			\begin{pmatrix}
				\mx&  \ldots  & \mx& \mx  \\
				\Rx &  \ddots & & \vdots \\
				\vdots  & \ddots  &\mx &    \mx\\
				\Rx &  \ldots & \Rx & \mx 
			\end{pmatrix}
			\cong
			\mathstrut_{\rho } \Hrd 
		\end{align*}
		where $(\blank)^\vee$ denotes $\Hom_{\Rx}(\blank, \Rx)$.
		Since $\rho$ has order $\ell$, the last statement follows.
		%%\item 	We recall that 
		%%any indecomposable object in 
		%%$\D^-(\md \Hrd)$ is isomorphic to the projective resolution of an indecomposable $\Hrd$-module up to shift, because $\Hrd$ is hereditary.
		%%Since $\Hrd$ is uniserial, 
		%%any indecomposable $\Hrd$-module is
		%%isomorphic to $P_j$ or $P_j/\rad^m P_j$ for certain $j \in Q_0$ and $m \in \N$.
		%\todo{reference}
		%\end{enumerate}
	\end{proof}

	%\begin{prp}
	%	Two indecomposable objects $X,Y$ in $\D-(\md \Hrd)$ are isomorphic if and only if 
	%	$\Ho_d(X) \cong \Ho_d(Y)$ for each $d \in \Z$.
	%	\end{prp}
%\begin{proof}
%	In particular, any indecomposable object in $\D^-(\md \Hrd)$ is isomorphic
%	to an indecomposable $\Hrd$-module.
%	\end{proof}

\begin{lem}\label{prp:hered}
An object $X$ in $\D(\Md \Hrd)$ appears in an exceptional cycle if and only if $X$ is isomorphic to a simple $\Hrd$-module 
%$S_j$ for certain $j\in Q_0$ 
up to shift.
%In this case, $X$ is fractionally $(n,n)$-Calabi-Yau.
\end{lem}
\begin{proof}
Since $\Hrd$ is hereditary, any indecomposable object in $\D(\Md \Hrd)$ is isomorphic to the projective resolution of an indecomposable $\Hrd$-module $M$ up to shift by \cite{Krause1}*{Proposition~4.4.15}.

%$\Rightarrow$ 
To show the `only if'-implication, 
assume that $X$ appears in an exceptional cycle. Lemma~\ref{lem:hered-CY} implies that $\CYdim(X) = (\ell,\ell) \neq (0,1)$, and thus
$\End_{\D(\Hrd)}(X) \cong \kk$. In particular,
%any $X$ indecomposable object in $\D(\Md \Hrd)$
$X$ is isomorphic to the projective resolution of an indecomposable $\Hrd$-module $M$ up to shift by \cite{Krause1}*{Proposition~4.4.15}.
%Since $\End_{\Hrd}(M) \cong \kk$, $M$ is finite-dimensional.

As $\Hrd$ is uniserial it follows that $X [i] \cong M_{j,n} \colonequals P_j/\rad^{n} P_j$ for certain $i\in \Z$, $j \in Q_0$ and $n \in \N^+$. It is sufficient to show that $n = 1$.
\begin{itemize}
	\item Let $\ell > 1$. 
	Assume that $n >1$ for a proof by contradiction. 
	Then there is a non-zero morphism
	$M_{j+1,n} \to M_{j,n}$, $m \mapsto m \alpha_{j} $.
	Thus $\HomD(\tau(X),X) \neq 0$. By Lemma~\ref{lem:vanish} it follows that $\ell = 1$, a contradiction. 
	%Thus $m = 1$, which shows the claim.  
	\item Let $\ell = 1$. Then $n = \dim M_{j,n} = \dim \End_{\D(\Hrd)}(X) = 1$. 
\end{itemize}
To show the `if'-implication note that the minimal projective resolution of a simple module $S_j$ is given by a two-term complex $P_{j+1} \to P_j$. It is then straightforward to verify \ref{E2} of Definition~\ref{dfn:ex-cycle}.
%Thus \begin{align*}
%	{\sum_{i=0}^{n-1}\sum_{d\in\Z}\Hom_{\Hrd}(S_j,\SD^i(S_j)[d]) =
%		\dim \End_{\Hrd}(S_j) + \Ext^1_{\Hrd}(S_j,S_{j+1}) = 2.}
%	%\flushright{\qedhere}
%\end{align*}
%\todo{qed-symbol}
\end{proof}
In summary, the sequence 
$(S_1,S_2\ldots, S_\ell)$ of simple $\Hrd$-modules
is the only exceptional sequence in $\D(\Md \Hrd)$ up to equivalence.

\begin{prp}
\label{prp:hered-char} 
Let $(Q,I)$ be a connected gentle quiver and $\A$ the arrow ideal completion of its path algebra.
Then the following conditions are equivalent.
\begin{enumerate}
	\item \label{hc3} There is a vertex $j \in Q_0^t$ such that its AG-invariant 
%	$(m(j),n(j))$ 
	satisfies $m(j) = n(j)$.
	\item \label{hc1} The gentle order $\A$ is hereditary.
	\item \label{hc2} Any object in $\D^-(\md \A)$ is $\nu$-periodic and $\A$ has finite global dimension.
	\end{enumerate}
\end{prp}
%\begin{rmk}
%	A ring-indecomposable gentle order is hereditary if and only if there is a vertex $j \in Q_0^t$ such that its AG-invariant $(m(j), n(j))$ satisfies $m(j)=n(j)$.
%\end{rmk}
\begin{proof}
\eqref{hc3} $\Rightarrow$ \eqref{hc1} Assume that \eqref{hc3} holds. Since $Q$ is connected, it follows that $Q_0 = Q_0^t$, thus $Q$ is a quiver of type $\widetilde{\mathbb{A}}$ and $\A$ is hereditary.

The implication \eqref{hc1} $\Rightarrow$ \eqref{hc2} follows from Lemma~\ref{lem:hered-CY}.

To show \eqref{hc2} $\Rightarrow$ \eqref{hc3}, assume that any object in $\D^-(\md \A)$ is $\nu$-periodic  and $\gldim \A < \infty$. 
By the second assumption, $\A$ is not a ribbon graph order, that is,
there exists a vertex $j \in Q_0^t$. Then the simple module $S_j$ is fractionally Calabi--Yau with $\CYdim(S_j)=(m(j),n(j))$ where $m(j) \geq n(j)$. As $S_j$ is $\nu$-periodic, it follows that $m(j) = n(j)$. 
\end{proof}
%
%
%%Hereditary orders admit the following characterization.
%For the sake of completeness, we note that
%hereditary gentle orders admit a characterization in a more general context.
%%In this subsection, we will briefly describe the exceptional cycles in the right-bounded derivred category 
%\begin{rmk}
%	Let $(Q,I)$ be a connected quiver and $\Hrd$ the arrow ideal completion of its path algebra $\kk Q/I$. Then the following conditions are equivalent.
%	\begin{enumerate}
%		\item \label{h1} The quiver $Q$ is equioriented of type  $\widetilde{\mathbb{A}}$ and $I = 0$.
%		%		\item \label{h2} $Q_0 = Q_0^t$ and $I=0$.
%		\item \label{h3} The $\kk$-algebra $\Hrd$ is a hereditary $\Rx$-order such that $\KK \otimes_{\Rx} \Hrd \cong \Mat_{\ell \times \ell}(\KK)$.
%		%		\item The associated surface $\Sigma_{\Hrd}$ is a disc.
%	\end{enumerate}
%This statement follows from the structure theorem for hereditary orders, see e.g. \cite{Reiner}*{Theorem~39.14}.
%\end{rmk}
%

\subsection{Normalization and rational hull of a gentle order}
\label{subsec:norm}
Let $(Q,I)$ be a connected gentle quiver.
The quiver $(Q,I)$ can be viewed as a `gluing' of quivers of equioriented Euclidean type $\widetilde{\mathbb{A}}$ as follows.

Let $Q' = (Q'_0,Q'_1,s',t')$ be the quiver defined by $Q'_0 \colonequals  \{i(\alpha) \mid \alpha \in Q_1\}$, $Q'_1 \colonequals Q_1$, $s'(\alpha) \colonequals i_{\alpha}$ and $t'(\alpha) \colonequals i_{\sigma(\alpha)}$ for any $\alpha \in Q'_1$.
Set $\pc \colonequals \pc_{\A}$ and let $\alpha_1, \ldots,\alpha_{\pc}$ denote a complete set of representatives of $\sigma$-orbits in $Q_1$.
Then $Q'$ is the product $\prod_{j=1}^{\pc} \widetilde{\mathbb{A}}_{\ell(j)-1}$,
%of equioriented quivers of type $\widetilde{\mathbb{A}}$,
where $\ell(j)$ denotes the length of the unique permitted cycle $pc(\alpha_j)$ beginning with $\alpha_j$.
The arrow ideal completion $\B$ of the path algebra $\kk  Q'$
is a hereditary order and isomorphic to the product $\prod_{j=1}^{\pc} T_{\ell(j)}(\Rx)$ such that each ring factor is a `triangular' matrix algebra of the form~\eqref{eq:tri-hered}.

\begin{lem}\label{lem:rad-emb}
	Let $\iota\colon \A \hookrightarrow \B$ be the unique $\kk$-algebra homomorphism satisfying $e_{i} \mapsto \sum_{\alpha \in Q_1\colon s(\alpha) = i} e_{i(\alpha)} $ for any $i \in Q_0$ and $\alpha \mapsto \alpha$ for any $\alpha \in Q_1$.
	Then $\iota$ is an $\Rx$-algebra monomorphism such that $\iota(\rad \A) = \rad(\B)$.	
	\end{lem}
\begin{proof}
	The completion of the arrow ideal in $\kk Q/I$ is isomorphic to $\rad \A$, and a similar statement holds for $Q'$. Because of this, the statement follows  by taking completion of the $\kk$-algebra monomorphism in \cite{GIK}*{Lemma~8.2~(2)}.
	\end{proof}

\begin{cor}\label{cor:rhull}
	The $\KK$-algebra $\KK \otimes_{\Rx} \A$ is Morita equivalent to the product $\KK^{\times \pc_{\A}}$. 
	\end{cor}
\begin{proof}
	Since the cokernel of $\iota$ has finite length as an $\Rx$-module, there are $\KK$-algebra isomorphisms $\KK \otimes_{\Rx} \A \cong \KK \otimes_{\Rx} \B \cong \prod_{j=1}^{\pc_{\A}} \Mat_{\ell(j)\times\ell(j)}(\KK)$, which imply the claim.
	\end{proof}

%\newpage
\section{Cartan matrix via graph theory}
%\todo{section intros}
\label{sec:cartan}
In this section, we introduce a variation of the notion of an undirected graph by allowing so-called \emph{truncated edges}.
%We determine the rank and the determinant of certain natural matrices associated to the graph-theoretic data.
There is a natural notion of a truncated graph $\Gr_{\A}$ associated to a gentle order $\A$, which allows to compute the rank and the determinant of its Cartan matrix. This section is essentially self-contained and requires only a few notions on gentle orders.

\subsection{Truncated graphs}
\label{subsec:trunc-gr}
Throughout this subsection, let $\Gr$ be a \emph{truncated graph} by which we mean a triple $(V,E,\mu)$ comprised from a finite set of vertices $V$, a finite set of edges $E$ 
and a map 
%\begin{align*}
$	\mu \colon E \longrightarrow \left\{m \colon V \to \N_0 \mid 
	\sum_{v \in V} m(v) \in \{1,2\} 
%	\text{ for any }e \in E	
	\right\}.$
	%	e \mapsto \mu(e) \equalscolon m_e
%\end{align*}

An edge $e \in E$ is called \emph{incident} to a vertex $v \in V$ if $\mu(e)(v) \neq 0$. 
For any edge $e \in E$ one of the following cases occurs.
\begin{itemize}
	\item There are distinct vertices $v_1$, $v_2 \in V$ with
	$\mu(e)(v_1)=\mu(e)(v_2) = 1$. In this case, the edge $e$ is called \emph{ordinary}.
\item There is a vertex $v\in V$ with $\mu(e)(v) = 2$.  Then $e$ is called a \emph{loop}.
\item Otherwise, 
$\sum_{v\in V} \mu(e)(v) = 1$.
%In other terms,
% $e$ is incident to only one vertex $v$ and
% $\mu(e)(v) = 1$. 
 This is when the edge $e$ is called \emph{truncated}.
\end{itemize}
To simplify the wording, we have called the datum $\Gr$ a truncated graph also if $\Gr$ has no truncated edges.
The goal of this subsection is to provide formulas for the rank and the determinant of certain matrices associated to the truncated graph $\Gr$.

To define the \emph{edge-vertex incidence matrix} $B_{\Gr}$ of the truncated graph $\Gr$, we need to choose an enumeration of edges 
$E = \{e_1,\ldots e_n\}$ and an enumeration of vertices $V = \{v_1,\ldots, v_p\}$.
For any indices $1\leq i \leq n$ and $1\leq j \leq p$
we set 
\begin{align}
	\label{eq:inc-mat}
	b_{ij} \colonequals \mu(e_i)(v_j) = \begin{cases}
		2 & \text{if $e_i$ is a loop incident to $v_j$},\\
		1 & \text{if $e_i$ is not a loop and $e_i$ is incident to $v_j$},\\
		0 &  \text{otherwise, that is, $e_i$ is not incident to $v_j$}.
	\end{cases}
\end{align}
This prescription defines the matrix $B_{\Gr} = (b_{ij}) \in \Mat_{n \times p}(\Z)$.

\begin{rmk}
	For any matrix $A \in \Mat_{n \times p}(\Z)$ we define $\rk A$ by the rank of the image of 
	the 
%	$\Z$-linear 
	map
	$\ell_A\colon \Z^p \to \Z^n$, $x \mapsto Ax$. 
	Since $\Q \otimes_{\Z} \im(\ell_A)
	\cong \im(\Q^p \overset{A \cdot}{\to} \Q^n)$
	it holds that $\rk A = \rk_{\Q} A$, and thus $\rk A = \rk A^T$.
\end{rmk}
We note that a different enumeration of edges or vertices of $\Gr$ yields another matrix $B'_{\Gr}$ which is related to $B_{\Gr}$ by permutations of rows or permutations of columns. In particular, $B_{\Gr}$ and $B'_{\Gr}$ have the same rank, but their determinants may differ by a sign.

%If the truncated graph $\Gr$ is connected, it holds that $\bc(\Gr) \in \{0,1\}$.
%This number is related rank of $B_{\Gr}$ is related to $\bc(\Gr)$
%So far we do not assume that $\Gr$ is connected.

The following statement is an adaptation of a result from graph theory concerning \cite{GR}*{Theorem~8.2.1}.
\begin{prp}\label{prp:rkB}
	Let $\bc(\Gr)$ denote the number of connected components of $\Gr$ which are bipartite graphs without truncated edges.
Then $\ker B_{\Gr} \cong \Z^{\bc(\Gr)}$.
\end{prp}
\begin{proof}
	For simplicity, we assume that $\Gr$ is connected. Set $B \colonequals B_{\Gr}$.
	For any $x \in \Z^p$ it holds that $B x=0$ if and only if for any indices $1 \leq i,j\leq p$ it holds that
	$x_i =0 $ in case there is a loop or a truncated edge incident to $v_i$
	and
	$x_i + x_j =0$ in case there is an each edge incident to $v_i$ and $v_j$, or, equivalently, 
	$x_i = (-1)^{\ell} x_j$ for any walk from $v_i$ to $v_j$ of length $\ell$.
	\begin{itemize}[label=--]
		\item If the graph $\Gr$ is not bipartite, there is a cyclic walk of odd length $\ell$ starting at a vertex $v_i$ for certain $1 \leq i \leq p$. Then $x_i = (-1)^{\ell} x_i = 0$.
		Since the graph $\Gr$ is connected, it follows that
		$x = 0$.
		\item If the graph $\Gr$ has a truncated edge,
		there is also an index $1 \leq i \leq p$ with $x_i = 0$, which implies again that $x=0$.
		\item Otherwise the graph $\Gr$ is bipartite without truncated edges. Let $1 \leq i \leq p$. In this case, the lengths of any two walks from vertex $v_i$ to vertex $v_1$ have the same parity. It follows that 
		$x \in \ker B$ if and only if 
		$x_i = x_1$ in case of even parity and $x_i = -x_1$ in case of odd parity. This shows that $\ker B \cong \Z$. \qedhere
	\end{itemize}
\end{proof}

%\begin{lem}
%	Let $\Gr$ be a connected truncated graph with $n = p$. Then
%	\begin{enumerate}
	%		\item If $\Gr$ is a tree with a unique truncated edge, then $	|\det B_{\Gr}| = 1$.
	%		\item If $\Gr$ has a cycle of odd length and no truncated edges},\\
%		\end{enumerate}
%\begin{align*}
%	|\det B_{\Gr}| = \begin{cases}
%1 & \text{if $\Gr$ is a tree with a unique truncated edge}, \\		
%2 & \text{if $\Gr$ has a cycle of odd length and no truncated edges},\\
%0 & \text{otherwise}.
%		\end{cases}
%	\end{align*}
%	\end{lem}
The next proofs require an operation of vertex deletion for truncated graphs. For a vertex $v$ of $\Gr$ let  $\Gr\backslash\{v\}$
 denote the graph obtained from $\Gr$ by removing vertex $v$, removing all truncated edges and loops incident to $v$ and transforming any of the remaining edges incident to $v$ into truncated ones.
\begin{prp}\label{prp:detB}
Let $\Gr$ be a connected, truncated graph with $n = p$.
Then 
\begin{align*}
\lvert\det B_{\Gr}\rvert = \begin{cases}
2 & \text{if $\Gr$ has a cycle of odd length, no truncated edges},\\
1 & \text{if $\Gr$ is a tree without truncated edges}, \\
0 & \text{otherwise}.
\end{cases}
\end{align*}
\end{prp}
\begin{proof}
	Let $b$ denote the number of truncated edges in $\Gr$. Set $g \colonequals n-b$ and $B \colonequals B_{\Gr}$.
%	Since $\Gr$ is connected it holds that $g \geq p-1$, which is an equality if and only if $\Gr$ is a tree.
\begin{enumerate}
\item Assume that $\Gr$ is a tree with a unique truncated edge, enumerated as $e_1$.
\begin{itemize}[label=--]
\item If $g = 0$, it holds that $\det B=1$.
\item In case $g \neq 0$, let $v_1$ denote the vertex incident to $e_1$.
Subtracting the first row of the matrix $B$ from each other row with a non-zero entry in the first column 
yields a block-diagonal matrix, which shows that $\det B = \det B_{\Gr\backslash\{v_1\}}$.
Since each connected component of $\Gr\backslash\{v\}$ is a tree with a unique truncated edge, we may iterate this procedure until all proper edges have been removed. Therefore, it follows that $\det B = 1$ as well.
\end{itemize}
\item Assume that $\Gr$ has a cycle of odd length $\ell$,
$b = 0$ and $n = p$.
We choose an enumeration of $V$ and $E$  such that
\begin{align*}
\begin{td}
	v_1 \ar[-]{r}{e_1} \& v_2 \ar[-]{r}{e_2}\& v_3  \ar[densely dotted,-]{rr}\& \&  v_{\ell} \ar[-]{r}{e_{\ell}} \&  v_1\end{td}
\end{align*}   
forms an odd cycle. Since $n = p$, removing any edge $e_i$ with $1 \leq i \leq \ell$ in  $\Gr$ yields a tree.
\begin{itemize}[label=--]
\item Assume that $\ell = 1$, that is, there is a loop at $v_1$.
Viewing $B$ as a matrix over $\Q$
and using that each connected component of  $B_{\Gr\backslash\{v_1\}}$ is a tree with a unique truncated edge, we may use elementary row transformations
to conclude that $\det B = 2 \det B_{\Gr\backslash\{v_1\}} = 2$.
\item If $\ell > 1$, the graph $\Gr$ has no loops
and Laplace expansion of the first row 
yields that
$\det B = \det B_{11} - \det B_{12}$.
Each of the minors $B_{11}$
and $B_{12}$ is the incidence matrix
of a tree with a unique truncated edge, which implies that
$\lvert \det B \rvert \in \{0,2\}$.
Since $\bc(\Gr) = 0$, Proposition~\ref{prp:rkB} yields that $\rk B = p = n$, and thus $\lvert \det B \rvert =2$.
\end{itemize}
In both cases it follows that $\lvert \det B \rvert = 2$.
\item Assume that $\det B \neq 0$. 
Since $n = p$, Proposition~\ref{prp:rkB}  implies that $\bc(\Gr) = 0$.
So $p = \rk B = n = b + g \geq b + p-1$, which is equivalent to $1 \geq b$. Thus there are only the two possibilities
$b \in \{1,0\}$,  
which correspond to the previous two cases.
This shows that $\det B = 0$ if none of these two cases occurs. \qedhere
\end{enumerate} 
\end{proof}
For the truncated graph $\Gr$ 
the \emph{signless edge-based Laplacian matrix} is 
defined by the symmetric matrix
 $A_{\Gr} \colonequals B B^T \in \Mat_{n \times n}(\Z)$.
% Then $A_{\Gr}  is a symmetric matrix.
%\begin{rmk}
%	In different graph-theoretic terms, 
%	it holds that
%	$A_{\Gr} = 2 \mathbf{1}_{n} + A_{\mathbf{L(G)}}$ where $A_{\mathsf{L}(G)}$ denotes the adjacency matrix of the line graph $L(G)$ of $G$. 
%	\end{rmk}

\begin{cor}\label{cor:rk-det}
For any connected truncated ribbon graph $\Gr$ with $p$ vertices and $n$ edges the rank and the determinant of the matrix $A_{\Gr}$
are given by
\begin{align}
\label{eq:rk}
\rk A_{\Gr} &= p - \bc(\Gr),\\
\label{eq:det}
\det (A_{\Gr}) &= \begin{cases}
n+1 & \text{if $\Gr$ is a tree
	without truncated edges} \\
4 &  \parbox[t]{.6\textwidth}{if $\Gr$ has a cycle of odd length, no truncated edges  and as many vertices as edges,} 
\\
1 & \text{if $\Gr$ is a tree with a unique truncated edge} \\
0 & \text{otherwise}.
\end{cases}
\end{align}  
\end{cor}
\begin{proof}
We set $B \colonequals B_{\Gr}$ and $C \colonequals A_{\Gr}$.
Since
$\rk C = \rk B^T = \rk B$, the first claim follows from Proposition~\ref{prp:rkB}.
\begin{itemize}[label=--]
\item Assume that $\Gr$ is a tree without truncated edges.
In this case, $n = g = p -1$.
For each index $1 \leq j \leq p$
let $B_j$ denote the square matrix obtained from $B$ by deleting column $j$. Each matrix $B_j$ corresponds to the incidence matrix of the graph $\Gr\backslash\{v_j\}$, which is given by a union of trees, each one of which has a unique truncated edge.
%Since each connected component of the graph $\Gr\backslash\{v_j\}$
%is a tree with a unique truncated edge,
Therefore,  Proposition~\ref{prp:detB} yields that  $\det B_j = 1$. 
By the Cauchy-Binet formula it follows that
%\begin{align*}
$\det C = \det (B B^T )= \sum_{j=1}^p (\det B_j)^2 = 
%p 
%=
 n+1.$
%\end{align*}
\item In the second and third case in \eqref{eq:det}, it holds that $n = p$ and 
%	If $\Gr$ has a cycle of odd length $\ell$,
%	$b = 0$ and $n = p$, or if $\Gr$ is a tree
%	with a unique truncated edge,
%	 it holds that 
$\det C = (\det B)^2 =4$ respectively $1$ by Proposition~\ref{prp:detB}.
%	\item If $\Gr$ is a tree 
%	Since  $n = b + g = p$, it follows that $\det C = (\det B)^2 = 1$.
\item Assume that $\det C \neq 0$. 
\begin{itemize}[label=--]
\item If $n = p$, it follows that $\det B \neq 0$, which leads to the second or the third case.
\item Assume that $n \neq p$.
%		Note that $\rk C = \rk B^T = \rk B = p - \bc(\Gr)$.
Since $n = \rk C = p - \bc(\Gr)$, it follows that $\bc(\Gr) = 1$, which corresponds to the first case.
\end{itemize}
%	This shows that $\det C = 0$ if none of these three cases occurs. 
%	Assume that $\det C \neq 0$. Then $p - \bc(\Gr) = \rk C_\A = n = b + g \geq b + p-1$, which is equivalent to $1 \geq b + \bc(\Gr)$. It follows that there are only the three possibilities
%	$(b, \bc(\Gr)) \in \{(1,0),(0,1),(0,0)\}$,  
%	%		In the first two cases, we have $g = p -1$, that is, $\Gr$ is a tree with $b \in \{1,0\}$.
%	%		In the last case, we have $n = p$ 
%	which correspond to the previous three cases.
This shows that $\det C = 0$ if none of these three cases occurs. \qedhere
\end{itemize} 
\end{proof}

\subsection{Application to the Cartan matrix of a gentle order}

Let $\A$ be a gentle order $\A$ such that its underlying quiver $(Q,I)$ is connected.
The gentle order $\A$  gives rise to a truncated graph $\Gr=\Gr_{\A}$ as follows.

Let $\alpha_1, \alpha_2, \ldots, \alpha_p$ denote a complete set of representatives of $\sigma$-orbits in $Q_1$
and $\{1,2,\ldots, n\}$ an enumeration of the vertices $Q_0$.
Then $\Gr = (V,E,\mu)$ with $V \colonequals \{ v_1, \ldots v_p\}$, $E \colonequals \{ e_1, \ldots e_n\}$
and 
%for any $1 \leq i \leq n$ and $1\leq j \leq p$ the number
$\mu(e_i)(v_j)$ is given by the number of arrows in $\beta \in Q_1$ such that $s(\beta) = i$
and $\beta \sim_{\sigma} \alpha_j$
%and are in the same $\sigma$-orbit as $\alpha_j$
for any indices $1 \leq i \leq n$ and $1 \leq j \leq p$.

Since $(Q,I)$ was assumed to be connected, the graph $\Gr$ is connected as well.
There is an equality $\bc_{\A} = \bc(\Gr)$
of the parameters which were introduced in
\ref{eq:bc}
and Proposition~\ref{prp:rkB}.

We may now view relate
Cartan matrix $C_{\A}$ introduced in Definition~\ref{dfn:Cartan}
to the matrix $A_{\Gr}$.

\begin{prp}\label{prp:rkC}
	For any gentle order $\A$ and its truncated graph $\Gr = \Gr_{\A}$ it holds that $C_{\A} = A_{\Gr}$.
%	B_{\Gr} B_{\Gr}^T$. 
\end{prp}
\begin{proof}
Let $\iota\colon \A \to \B$
denote the monomorphism from Lemma~\ref{lem:rad-emb}.
Let $1 \leq j \leq p$ and $L_j \colonequals \A \alpha_{j}$. 
In the notation of Subsection~\ref{subsec:norm}, the indecomposable projective $\B$-module $\B \otimes_{\A} L_{\alpha}$ lies in the $j$-th ring factor of the overring $\B$. Accordingly, $V_j \colonequals \KK \otimes_{\Rx} L_{j}$ denotes the $j$-th simple module over $\KK \otimes_{\Rx} \A$.
%This prescription yields an enumeration $V_1, \ldots V_p$ of the simple $\KK \otimes_{\Rx}$
%Since $p = |V| = \pc_{\A}$
%it holds that $\simp \KK \otimes_{\Rx} \A = \{V_1, \ldots, V_p\}$.
%Then there is a unique index $1\leq j(\alpha) \leq p$ such that $\KK \otimes_{\Rx} L_{\alpha} = V_{j(\alpha)}$.
By the considerations above, there are natural bijections
\begin{align*}
\begin{array}{l} E \overset{\sim}{\longrightarrow} \ind\proj \A,   
%	\\
\quad
e_i \longmapsto  P_i
\end{array}
&&
\begin{array}{l} V \overset{\sim}{\longrightarrow} \simp \KK \otimes_{\Rx} \A,  
%	\\
\quad
v_j \longmapsto  V_{j}.
\end{array}
\end{align*}
For any index $1 \leq i \leq n$ we consider the decomposition of $\KK \otimes_{\Rx} P_i$.
% into simple modules.
\begin{itemize}[label=--]
\item If $e_i$ is an ordinary edge incident to distinct vertices $v_{j}$ and $v_{k}$, then $\iota(e_i)$ is a sum of two primitive idempotents, which appear in components $\B_{j}$ and $\B_{k}$ such that $j \neq k$, and thus
$\KK \otimes_{\Rx} P_i \cong V_{j} \oplus V_{k}$ decomposes in two non-isomorphic simple modules.
\item If $e_i$ is a loop incident to a vertex $v_{j}$, then 
$\iota(e_i)$ is a sum of two primitive idempotents, which both appear in the component $\B_{j}$, and thus $\KK \otimes_{\Rx} P_i \cong V_{j}^2$.
\item Otherwise, $e_i$ is a truncated edge incident to a vertex $v_j$, then the idempotent $\iota(e_i)$ is primitive as well, and appears in component $\B_{j}$, which yields that $\KK \otimes_{\Rx} P_i \cong V_{j}$ is simple.
\end{itemize}
For any index $1\leq i \leq n$ it follows that
$\KK \otimes_{\Rx} P_i \cong V_1^{b_{i1}} \oplus V_2^{b_{i2}} \ldots \oplus V_p^{b_{ip}}$
with $b_{ij}$ defined via \eqref{eq:inc-mat}.
In particular, the incidence matrix $B_{\Gr}$ coincides with the \emph{decomposition matrix} 
$D_{\A} \in \Mat_{n \times m}(\Z)$
of the order $\A$.
Since the algebras $\A/\rad\A$ and $K\otimes_{\Rx} \A$ are split semisimple, 
\cite{Zimmermann}*{Proposition~2.6.7} implies that $C_{\A} = D_{\A} D^T_{\A}$, and thus the claim holds.
%Since $\ker C_{\A} = \ker B^T_{\A}$ 
%it holds that $\rk C_{\A} = \rk D^T_{\A} = \rk D_{\A} = \rk B_{\Gr}$, so 
%The remaining claims follow then from
\end{proof}

%	In particular,  $\rk C_{\A} = \pc_{\A} - \bc_{\A}$ and $\det C_{\A}$ is determined by 
%\eqref{eq:det}.
\begin{cor}\label{cor:rk-det2}
	For any gentle order $\A$ it holds that $\rk C_{\A} = \pc_{\A} - \bc_{\A}$. Moreover, $\det C_{\A}$ is given by 
	\eqref{eq:det}.
	\end{cor}
\begin{proof}
	These claims follow from Proposition~\ref{prp:rkC} and Corollary~\ref{cor:rk-det}.
	\end{proof} 

%\begin{rmk}
%For the subclass of ribbon graph orders,
%	In the case that $\A$ is ribbon graph order, 
	For the subclass of ribbon graph orders,
the statements of the last corollary
%	the formula for the rank and for the determinant of $C_{\A}$ 
	are equivalent to corresponding results of Antipov~\cites{Antipov1, Antipov2} specialized to multiplicity-free Brauer graph algebras.
%	\end{rmk}

%\newpage

\section{A factorization of the Serre functor}
\label{sec:factor}
From now on, we fix a gentle order $\A$. 
The goal of this section is to prove Theorem~\ref{thm:A} on the factorization of the relative Serre functor.
%In other terms, $\A$ is isomorphic to the arrow ideal completion of the path algebra $\kk  Q/I$ for a connected quiver $Q$.

\subsection{Serre duality for gentle orders}

To recall an explicit description of the derived Nakayama functor from \cite{GIK} we consider a certain automorphism $\inv$ of the gentle order $\A$.
\begin{dfn}\label{dfn:inv}
	We define a $\kk$-algebra involution $\inv \colon \A \to \A$ by 
	$\inv(e_i) = e_i$ for any $i \in Q_0$ and the following prescription on arrows.
	\begin{itemize}[label=--]
		\item If the base field $\kk$ has characteristic two, we set $\inv(\alpha) \colonequals \alpha$ for any $\alpha \in Q_1$.
		\item Assume that $\kk$ has characteristic other than two.
		First, we need to choose a map $\sgn \colon Q_1 \to \{-1,1\}$
		%	, \alpha \mapsto \varepsilon_{\alpha}$
		such that $\sgn(\alpha) \neq \sgn(\beta)$ for any two arrows $\alpha \neq \beta$ in $Q_1$ with $s(\alpha) = s(\beta)$.
		For each arrow $\alpha \in Q_1$ we set $\inv(\alpha) \colonequals \sgn(\sigma(\alpha)) \, \sgn(\alpha)\, \alpha$
		where $\sigma(\alpha)$ denotes the unique arrow in $Q_1$ such that $\sigma(\alpha) \alpha \notin I$.
		
	\end{itemize}
\end{dfn}
Let $\mathstrut_{\inv} \A$ 
denote the $\A$-bimodule with underlying set $\A$, the regular right $\A$-module structure and the left $\A$-module structure given by 
$a \cdot x \colonequals \inv(a) x$ for any $a\in\A$ and $x\in\mathstrut_{\inv}\A$.
Let $\mathstrut_{\inv} \A^\circ$ denote the subbimodule generated by the idempotents $e_i$ with $i \in Q_0^c$ and arrows $\alpha \in Q_1$  such that $s(\alpha) \in Q_0^t$.
We recall that the canonical bimodule  $\omega \colonequals \Hom_{\Rx}(\A,\Rx)$ gives rise to the derived Nakayama functor $\nu$, which restricts to an auto-equivalence of $\D^-(\md \A)$.

\begin{thm}[\cite{GIK}*{Theorem~9.7, Corollary~9.8}]\label{thm:SD}
	There is an isomorphism of $\A$-bimodules $\mathstrut_{\inv}\A^{\circ} \cong \omega$. 
	In particular, 
%	for any vertex $i \in Q_0$ there is a $\A$-linear isomorphism
		for any arrow  $\alpha \colon i  \to j $ in $Q$ it holds that
	\begin{align*}
\begin{td}
			(P_{j} \ar{r}{\cdot \alpha} \& P_{i}) \ar[mapsto]{r}[font=\normalsize]{\nu}\&
%		\end{td}) \overset{\mbox{\huge \nu}}{\longmapsto}  
%		(\begin{td}
			(L_{j} \ar{r}{\cdot \inv(\alpha)} \& L_{i})
		\end{td}&&
	\text{with}\quad		L_j \colonequals \begin{cases}
		\rad P_j & \text{if }j \in Q_0^t,\\
		P_j & \text{if }j \in Q_0^c.
	\end{cases}
	\end{align*}
%	where $\rad P_i$ is the unique maximal submodule of the indecomposable projective $P_i$ 
\end{thm}	
The last statement allows to compute $\nu(\CP)$ for any complex $\CP$ from $\Hot^-(\proj \A)$.
%\todo{remark on signs}
%We will see later that                                                                                                                                                                                                                                                          the signs of the arrows $\inv(\alpha)$ can be neglected when computing the action of the derived Nakayama functor on objects (Proposition~\ref{prp:nu}).

%\newpage
%\section{Exceptional cycles from simples at transition vertices}

\subsection{Exceptional cycles from simples at transition vertices}
\begin{notation}\label{not:lin-string}
	A complex of projective $\A$-modules of the form
	\begin{align*}
		\begin{td}
			P_{i^{(m)}} \ar{r}{\cdot p_m}\& P_{j^{(m-1)}} 
			\ar[densely dotted,-]{rr} \& \& P_{j^{(2)}} 
			\ar{r}{\cdot p_2} \& P_{j^{(1)}} \ar{r}{\cdot p_1} \& P_{i^{(0)}}
		\end{td}
	\end{align*} with the projective $P_j$  located at degree zero,
	paths $p_1, p_2, \ldots p_m, \notin I$,
	and
	vertices $i^{(0)}$, $i^{(m)} \in Q_0^t$
	and
	$j^{(1)}, j^{(2)},\ldots j^{(m-1)} \in Q_0^c$
	will be abbreviated by the diagram
	\begin{align*}
		\begin{td}
			\circ \ar{r}{ p_m} \& \bt \ar[densely dotted,-]{rr} \& \&  \bt \ar{r}{ p_2} \& \bt \ar{r}{ p_1} \& \circ 
		\end{td}
	\end{align*} 
%	In particular, hollow circles represent indecomposable projectives at transition vertices, while full circles indicate crossing vertices.
\end{notation}
%AG-cycles of transition vertices were defined in Subsection~
%AG-cycles at transition vertices were defined in Subsection~\ref{subsec:inv}.
%For the next statement, we recall from Subsection~\ref{subsec:inv} the notation $(m(j), n(j))$ for the AG-invariant of a transition vertex $j$.
%we use the notation introduced in  for AG-cycles at transition vertices.
%According to the next observation, their homological counterparts are given by projective resolutions of simple modules at transition vertices.
%We recall from 

The next statement uses the notions of AG-invariants of first type
and the permutation $\kappa$ on the set $Q_0^t$ of transition vertices
 introduced in Subsection~\ref{subsec:inv}.
%Forbidden threads and AG-cycles are combinatorial counterparts to distinguished objects in 
%the category $\perfdH \A$.

\begin{lem}\label{lem:sim}
%	The following statements hold.
%	\begin{enumerate}
%		\item
	For any vertex $j \in Q_0^t$ 
	the simple module $S_j$ 
%	has finite projective dimension, 
	is fractionally Calabi-Yau with $\CYdim S_j = (m(j), n(j))$
%	the AG-invariant of vertex $j$, 
	and gives rise 
%	has finite projective dimension and 
%	is fractionally Calabi-Yau of dimension $(m(j),n(j))$ and
to a repetition-free exceptional $n(j)$-cycle
%	Let $f = f_{n} \ldots f_2 f_1$ denote the unique AG-cycle starting at $j$
%	and set $m \colonequals \ell(f)$.
%	Then 
% \begin{align*}
$		\mathcal{S}(j) \colonequals (S_j, S_{\kappa(j)},\ldots 
		S_{\kappa^{n(j)-1}(j)}
)$
%		\end{align*} 
	 in the category $\perfd (\A)$.
% such that $\CYdim S_j$ is given by the AG-invariant $(m(j),n(j))$.
%	at the mouth of an AR-component 
%	\item There is an injective map
%	\begin{align} 
%		\label{ag1-map}
%		\begin{td}
%		\AG_1(\A) \ar[hookrightarrow]{r} \& \CY(\perfdH \A)
%%		\& 
%%		(m(j),n(j)) \ar[mapsto]{r}  \& \CYdim S_j
%		\end{td}
%		\end{align}
%	\end{enumerate}
%	Moreover, the inequality $m \geq n$ holds true, which becomes an equality if and only if $\A$ is hereditary.
\end{lem}
\begin{proof}
	Let $f_1 = \beta_\ell \ldots \beta_2 \beta_1$ 
	denote the forbidden thread starting at $j$.
	In terms of Notation~\ref{not:lin-string},
	the minimal projective resolution  $P$ of the simple module $S_j$ 
	is given by the diagram
%	isomorphic to
%	 the string complex associated to the diagram
	\begin{align*}
		\begin{td}
			\circ \ar{r}{\beta_\ell} \& \bt  \ar[densely dotted,-]{rr} \&\&   \bt \ar{r}{\beta_2} \& \bt \ar{r}{\beta_1} \& \circ 
		\end{td}
	\end{align*}
%has the form
%	\begin{align*}
%			\begin{td}
%	P_{\kappa(j)} \ar[hookrightarrow]{r}{\cdot \beta_{m} } \& B_* \ar{r} \& \ldots B_{*} \ar{r}{\beta_2} \& B_* \ar{r}{\cdot \beta_1} \& P_j
%				\end{td}
%		\end{align*}
%	such that its arrows satisfy $\beta_\ell \ldots \beta_2 \beta_1 = f_1$.
%	and each term of the form $B_{*}$ is isomorphic to $P_{i'}$ for certain $i' \in Q_0^c$.
	 In particular, $\prdim S_j = \ell(f_1) <\infty$, and thus $S_j \in \perfd(\A)$.
	 
%According to Theorem~\ref{thm:SD}, the complex $\nu(\CP)$ 
Applying $\nu$ to the projective resolution $P$ via Theorem~\ref{thm:SD}
	yields a complex $L$ with only one non-zero homology, which 
	is given by $\Ho_{\ell-1}(L) \cong \head (\A
	\beta_\ell) \cong S_{\kappa(j)}$ with $\kappa(j)=t(f_1)$.
%It follows that
%	A direct computation shows that
This shows that	$\SD(S_j) = \nu(S_j)[1] \cong S_{\kappa(j)}[\ell(f_1)]$.

Let	$c(j)
= f_{n(j)} \ldots f_2 f_1$ denote the AG-cycle starting at $j$.	
	As $\kappa$ preserves $Q_0^t$, an iteration of the argument above yields   that 
	$\SD^{p}(S_j) \cong S_{\kappa^{p}(j)}[\ell(f_1) + \ldots \ell(f_{p})]$  for any index $1\leq p\leq n(j)$. 
Since $m(j) = \ell(c(j))$ and $n(j) = \min\{ p \in \N \mid \kappa^p(j) = j \}$, 
	the sequence $\mathcal{S}(j)$ satisfies condition \ref{E1} of Definition~\ref{dfn:ex-cycle}, $S_j$ is fractionally Calabi-Yau of dimension $(m(j),n(j))$ and $S_1 \not\cong S_{\kappa^p(j)}[i]$ for any $1 \leq p < n(j)$ and $i \in \Z$. 
	
%		c(j) = f_{n(j)} \ldots f_2 f_1$ denote the AG-cycle starting at $j$.

	Using that $\dim \HomD(P,S_{j'}[i])$
	is given by the number of indecomposable projective summands $\A e_{j'}$ at degree $i$ in the projective resolution $P$
	and the diagram above, it follows that the module $S_j$ satisfies condition \ref{E2} of Definition~\ref{dfn:ex-cycle}.
%	Since the simple modules are pairwise non-isomorphic, the exceptional cycle $\mathcal{S}(j)$ is repetition-free.
%	$H_0$-complexes
%%	
%	Since the diagram above has two transition vertices, which correspond to $j$ and $\kappa(j)$, it holds that
%	$\sum_{i \in Q_0^t} \sum_{d\in\Z}\Hom(S_j, S_i[d]) = \sum_{d\in\Z}\Hom(S_j, S_j \oplus S_{\kappa(j)}[d]) = 2$, thus the sequence $\mathcal{S}(j)$ is exceptional $n$-cycle.
\end{proof}
In fact, the previous proof shows that there is a bijection between forbidden threads in $(Q,I)$ 
and simple $\A$-modules at transition vertices.

\subsection{A monoid of tilting ideals}

We call a subset $\Xset \subseteq Q_0^t$ be a \emph{$\kappa$-stable} if $\kappa(\Xset) = \Xset$.
For such a subset
%For a $\kappa$-stable subset $\Xset \subseteq Q_0^t$ 
we consider 
%the $\A$-bimodule given by 
the two-sided $\A$-ideal 
\begin{align}
	\label{eq:gen}
	I(\Xset) \colonequals \langle e_i, \alpha \mid i \in Q_0\backslash \Xset, \alpha \in Q_1\colon s(\alpha) \in \Xset \rangle 
%	\ \triangleleft \  \A\,. 
\end{align}
	where $I(\emptyset) = \A$ in case  $\Xset$ is empty.
	
\begin{lem} \label{lem:ideals} Let $\Xset$, $\Yset$ be $\kappa$-stable
	subsets of $Q_0^t$.  
	\begin{enumerate}
\item \label{ideals2} It holds that $\Xset \supseteq \Yset$ if and only if $I(\Xset) \subseteq I(\Yset)$.
\item \label{ideals1}
If $\Xset$ and $\Yset$ are disjoint, then $I(\Xset \cup \Yset) = I(\Xset) I(\Yset) = I(\Yset) I(\Xset)$.
\end{enumerate}
 		\end{lem}
 		\begin{proof}
 			\begin{enumerate}
	 \item To verify the `only if'-implication let $\Xset \supseteq \Yset$. For any $e_i \in I(\Xset)$ it follows directly that $e_i \in I(\Yset)$.
Let $\alpha \in Q_1 \cap I(\Xset)$. If $s(\alpha)\in \Yset$, it holds that $\alpha \in I(\Yset)$.
Otherwise,  $\alpha = \alpha e_{s(\alpha)} \in \A I(\Yset) = I(\Yset)$. Thus $I(\Xset) \subseteq I(\Yset)$.

Vice versa, let $I(\Xset) \subseteq I(\Yset)$. If there was an element $i\in \Yset \backslash \Xset$, it would follow that $e_i \in I(\Xset) \subseteq I(\Yset)$, and thus $i \notin \Yset$, a contradiction. 
So $\Yset \supseteq \Xset$.
 				\item First, we claim that any generator of $I(\Xset \cup \Yset)$ is contained in the two-sided ideal $I(\Xset) I(\Yset)$.
For any $i \in Q_0$ with $i \notin \Xset \cup \Yset$ it holds that  
$e_i^{\phantom{2}} = e_i^2 \in I(\Xset) I(\Yset)$.
Let $\alpha \in Q_1 \cap I(\Xset \cup \Yset)$.
% one of the following occurs. 
%\begin{itemize}[label=--]
%	\item  
If $s(\alpha) \in \Xset\backslash\Yset$, then $\alpha = \alpha e_{s(\alpha)} \in I(\Xset) I(\Yset)$.
Otherwise, $s(\alpha) \in \Yset \backslash \Xset$. In this case, if follows that $t(\alpha) \notin \Xset$ using that
 $\Xset$ and $\Yset$ are disjoint and $\kappa$-stable, and thus $\alpha = e_{t(\alpha)} \alpha \in I(\Xset) I(\Yset)$.
It follows that $I(\Xset \cup \Yset) \subseteq I(\Xset) I(\Yset)$. The inclusion 
$\supseteq$ follows
 from \eqref{ideals2}. Interchanging $\Xset$ and $\Yset$ yields the remaining equality in \eqref{ideals1}. \qedhere
%The other inclusion $\supseteq$ follows from \eqref{ideals2}.
%It is straightforward to check that all non-zero products
%	 $e_j e_i$, $e_j \alpha$, $\beta e_i$, $\beta \alpha$ with $e_j, \beta \in I(\Xset)$, and $e_i, \alpha \in I(\Yset)$ are contained in $I(\Xset \cup \Yset)$.
	 \end{enumerate}
 			\end{proof}
Such two-sided ideals form a monoid under multiplication.
\begin{rmk}
	Let $\mathcal{I} \colonequals \{ I(\Xset) \mid \Xset \subseteq Q_0^t \text{ $\kappa$-stable}\}$
	and let $\mathcal{P}$ denote the power set of $Q_0^t/\!\sim_{\kappa}$.
	By Lemma~\ref{lem:ideals} there is an isomorphism of posets
	$(\mathcal{P}, \supseteq) \overset{\sim}{\longrightarrow} (\mathcal{I},\subseteq)$.
\end{rmk}

\begin{lem}\label{lem:tilt}
		For any $\kappa$-stable subset $\Xset \subseteq Q_0^t$ the ideal $I(\Xset)$
		defines a left tilting $\A$-module of finite projective dimension.
	\end{lem}
\begin{proof}
		In this proof, for any $\A$-modules $M,N$ we write $M \geq N$ 
	if $\Ext^n_{\A}(M,N)=0$ for all $n \in \N^+$, or, equivalently, all $n \neq 0$. 
There is an isomorphism of  left $\A$-modules 
\begin{align}
	\label{eq:left-dec}
	T \colonequals I(\Xset) \cong L_{\Xset} \oplus P_{\Xset^c} \quad \text{with}\quad
	L_{\Xset} \colonequals
	\bigoplus_{i \in \Xset} \rad P_{i}
	\quad\text{and}\quad
	P_{\Xset^c} \colonequals \bigoplus_{j \in Q_0\backslash \Xset} P_{j}.
	%		\quad
	%		B \colonequals \bigoplus_{i'' \in Q_0^c} P_{i''},
\end{align}
For any $i \in \Xset$ 
there is an exact sequence of $\A$-modules
\begin{align}
	\label{eq:ex-rad}
	\begin{td}
		0 \ar{r} \&
		P_{\kappa(i)} \ar{r}{\cdot \beta_{\ell(i)}} \&
		P_{j^{\ell(i)-1}} \ar[densely dotted,-]{rr} \& \& 
		P_{j^{(2)}} \ar{r}{\cdot \beta_{2}} \& P_{j^{(1)}} \ar{r}{\cdot \beta_1} \& \rad P_{i} \ar{r}\& 0
	\end{td}	
\end{align}
where $j^{(1)}, j^{(2)} \ldots j^{\ell(i)-1} \in Q_0^c$
and
the path $\beta_{\ell(i)} \ldots \beta_2 \beta_1$ comprises the forbidden thread starting at $i$.
The sequence above implies that $\rad P_i$ has projective dimension $\ell(i)-1$ 
and that $P_{\kappa(i)}$ has an $\add T$-resolution of the same length.
It follows that $I(\Xset)$ has finite projective dimension $\ell
 =
\max \{ \ell(i) -1 \mid i \in \Xset \}$
and that $\A$ has an $\add T$-resolution of length $\ell$.
%$I(\Xset)$ has an 
%The sequence above yields a minimal projective resolution of $\rad P_i$ via truncation
%and shows that $P_{\kappa(i)}$ has an $\add T$-resolution 
%of length $\ell(i)-1$.
%It follows that 
%$\ell \colonequals \prdim T = 
%\max \{ \ell(i) \mid i \in \Xset \} - 1$ is finite.
                                                                                                                                                                                                                                                                                            
It remains to show that $T \geq T$.
It holds that $P_{\Xset^c} \geq L_{\Xset} \oplus P_{\Xset^c}$ because
$P_{\Xset^c}$ is projective.
Viewed as a left $\A$-module the canonical bimodule $\omega$ is a tilting module.
As $L_{\Xset}$ is a direct summand of $\omega$, it follows that $L_{\Xset} \geq L_{\Xset}$.
For any $i \in \Xset$, $j \in Q_0$ and $n \in \N^+$, a computation in the homotopy category
using the minimal projective resolution of $\rad P_i$ shows that 
$\Ext^n_{\A}(\rad P_i,P_{j})
\neq 0$ if and only if $j = \kappa(i)$ and $n = \ell(i)$. This implies that $L_{\Xset} \geq P_{\Xset^c}$.
which completes the proof that
%Summarized, we obtain that
%it follows that
 $T \geq T$.
%conclude that $T$ is a tilting module of projective dimension $n$.
	\end{proof}

\begin{prp}
	Any $\kappa$-stable set $\Xset \subseteq Q_0^t$ induces
 an auto-equivalence
 \begin{align*}
 	\begin{td}
% 	\Tw_{\Xset} \colonequals
 	I(\Xset)\lotimes \blank \colon \D(\Md \A) \ar{r}{\sim} \&\D(\Md \A)
 	\end{td}
 	\end{align*}
% which restricts to an auto-equivalence of $\Db(\md \A)$.
\end{prp}
\begin{proof}
 For any $a \in \A$ let $r_a \colon I(\Xset) \to I(\Xset)$ denote the left $\A$-linear
	morphism given by
	right multiplication with $a$.
	We claim that the  $\Rx$-algebra morphism \begin{align*}
	\begin{td}
		r \colon \A \ar{r} \& \Gamma \colonequals \End_{\A}(I(\Xset))^{\op}
		\& a \ar[mapsto]{r}\&
		r_a
%		 \left(r_a \colon I(\Xset) \rightarrow I(\Xset), x \mapsto xa\right)
	\end{td}
\end{align*}
is bijective. 
There is a factorization
$\omega = I(\Xset^c) I(\Xset)$. Moreover, the $\Rx$-module $I(\Xset)/\omega$ has finite length.
This implies that the
restriction map
\begin{align*}
	\begin{td}
(\blank)|_{\omega}\colon
%	\End_\A(I(\Xset))^{\op}
\Gamma
	\ar{r} \& \End_{\A}(\omega)^{\op}\&
	\phi \ar[mapsto]{r} \& \phi|_{\omega} 
	\end{td}
\end{align*}
is a well-defined monomorphism of $\Rx$-algebras.
The composition $r|_{\omega} \colonequals  (\blank)|_{\omega } \circ r$ maps any element $a \in \A$ to 
the right multiplication with $a$ on the bimodule $\omega$.
% $r_a(y) = y a$ for any $y \in \omega$. 
By the proof of \cite{IK}*{Theorem~4.5} 
the map $r|_{\omega}$ is bijective.
It follows that the maps $(\blank)|_{\omega}$ and $r$ are bijective as well. 	

The functor in the claim is isomorphic to the composition
\begin{align*}
	\begin{td}
	\D(\Md \A) \ar{rr}{\Gamma \lotimes \blank }[swap]{\sim} \& \&
	\D(\Md \Gamma)  \ar{rr}{ I(\Xset) \overset{\mathbf{L}}{\otimes}_{\Gamma} \blank}[swap]{\sim}\&\& \D(\Md \A) 
%	\\
%	X \ar[mapsto]{r} \& \Gamma \lotimes X \ Y \ar[mapsto]{r} \&
%	I(\Xset) \overset{\mathbb{L}}{\otimes}_{\Gamma} Y
	\end{td}
	\end{align*}
where the first functor is an equivalence, since $r$ is an isomorphism, and the
second functor is an equivalence by Lemma~\ref{lem:tilt} and \cite{Keller2}*{8.1.4}.
\end{proof}

For any $\kappa$-stable subset we denote
by 
\begin{align*}
	\begin{td} \Tw_{\Xset} \colonequals I(\Xset) \lotimes \blank \colon \D^-(\md \A) \ar{r}{\sim} \& \D^-(\md \A) \end{td} \end{align*}
the restriction of the auto-equivalence of the previous statement to the right-bounded derived category.
For any complex $P \in \Hot^{-}(\proj \A)$ 
the complex
$\Tw_{\Xset}(P)$ is given by taking the radical at each indecomposable projective $P_i$ with $i \in \Xset$.

In Definition~\ref{dfn:inv}, an involution $\inv$ of the order $\A$ was introduced, which gives rise to an auto-equivalence of order two
$$\begin{td}\inv^* \colonequals \mathstrut_{\inv} \A \lotimes \blank  \colon \D^-(\md \A) \ar{r}{\sim} \& \D^-(\md \A)\end{td}.$$
where $\mathstrut_{\inv} \A $ denotes the $\A$-bimodule with underlying set $\A$, the left multiplication twisted by $\xi$ and regular right $\A$-module action.
%underlying set $\A$ setting
%$ a \cdot x \colonequals \inv(a) x$ and $x \cdot a = xa$ 
%for any $a,x \in \A$.

\begin{lem}\label{lem:tw-comm}
	Let $\Xset$ and $\Yset$ be disjoint $\kappa$-stable subsets of $Q_0^t$.
	Then there are isomorphisms of auto-equivalences of $\D^-(\md \A)$
	\begin{align*}
%\begin{td}
	\inv^* \circ \Tw_{\Xset} \cong \Tw_{\Xset} \circ \inv^*&&
%	\colon \D^-(\md \A) \ar{r}{\sim} \& \D^-(\md \A) \\ 
	\Tw_{\Xset \cup \Yset} \cong \Tw_{\Xset} \circ \Tw_{\Yset} \cong
	\Tw_{\Yset} \circ \Tw_{\Xset}
%	\colon \D^-(\md \A) \ar{r}{\sim} \& \D^-(\md \A).
%	\end{td}
	\end{align*}
	\end{lem}
	\begin{proof}
	For any automorphism $\phi$ of the ring $\A$ 	and any
	$\A$-bimodule $M$ 
	it holds that
	$$
	\mathstrut_\phi \A \lotimes M 
	\cong \mathstrut_{\phi} \A \otimes M \cong \mathstrut_{\phi} M \cong M_{\phi^{-1}} \cong M \otimes \A_{\phi^{-1}} \cong M \lotimes \A_{\phi^{-1}}.
	$$
	This implies the first isomorphism using that $\xi$ is an involution. 
		
		It remains to show that $\Tw_{\Xset \cup \Yset} \cong \Tw_{\Yset} \circ \Tw_{\Xset}$ as the other isomorphism then follows by interchanging $\Xset$ and $\Yset$.
		Set $S(\Yset) \colonequals \A/I(\Yset)$.
	Using that $\Xset \cap (\Yset \cup Q_0^c) = \emptyset$ and the minimal projective resolution of $\rad P_i$
	it follows that
		$$S(\Yset) \lotimes \rad P_i \cong 0 \quad \text{for any $i \in \Xset$}.$$
		Using the decomposition of $I(\Xset)$ as left $\A$-module in \eqref{eq:left-dec}
%		Since $I(\Xset) \cong \bigoplus_{i \in \Xset} \rad P_i \oplus \bigoplus_{j \in Q_0\backslash \Xset} P_{j}$
%		as left $\A$-module 
		it holds that
		$$\Tor_n^\A(I(\Yset),I(\Xset) \otimes_{\A} P) \cong \Tor_{n+1}^\A(S(\Yset), I(\Xset)\otimes_{\A} P) = 0$$
		for any $n \in \N^+$ and any module $P \in \proj \A$.
%		
%		 as well as 
%		$\Tor_1^\A(S(\Yset), I(\Xset)) = 0$. 
		This yields the isomorphism of functors 
		$$ I(\Yset) \lotimes (I(\Xset) \lotimes \blank ) \cong (I(\Yset) \otimes_{\A} I(\Xset)) \lotimes \blank \colon 
		\begin{td} \D^-(\md \A) \ar{r}{\sim} \&  \D^-(\md \A). \end{td}$$
		The previous consideration implies also that $\Tor_1^\A(S(\Yset), I(\Xset)) = 0$, which yields an isomorphism of $\A$-bimodules
		$$ I(\Yset) \otimes_{\A} I(\Xset)
		\cong I(\Yset) I(\Xset) = I(\Xset \cup \Yset).$$
		The last equality holds by Lemma~\ref{lem:ideals}~\eqref{ideals1}.
		\end{proof}

%\subsection{A factorization of the derived Nakayama functor}
%For a vertex $j \in Q_0^t$
%we denote by $[j]_{\kappa}$ its $\kappa$-orbit.

Let $j_1, \ldots j_b$ denote a complete set of representatives of $\kappa$-orbits of transition vertices.
For each index $1 \leq a \leq b$ we consider the auto-equivalence 
\begin{align*}
	\begin{td} \Tw_a \colonequals I(\Xset_a) \lotimes \blank \colon \D^-(\md \A) \ar{r}{\sim} \& \D^-(\md \A) \end{td}
	\end{align*}
where $\Xset_a$ denotes the $\kappa$-orbit of the vertex $j_a$.

The first main result of this article is the following factorization of the derived Nakayama functor.
\begin{thm}\label{thm:nu-factors}
	In the notations above, there is an isomorphism of functors
	\begin{align}
		\begin{td}
		\nu \cong \inv^* \circ \Tw_b \circ \ldots \Tw_2 \circ \Tw_1 \colon \D^-(\md \A) \ar{r}{\sim} \& \D^-(\md \A) \end{td}
	\end{align}
	where any two functors on the right hand side commute with each other.
\end{thm}
\begin{proof}
	By 
	Theorem~\ref{thm:SD}
	there are isomorphisms of $\A$-bimodules
	$\mathstrut_{\xi}\omega\cong\A^{\circ} \cong I(Q_0^t)$.
		Using that $Q_0^t = \bigcup_{1 \leq a \leq b} \Xset_a$, Lemma~\ref{lem:tw-comm} implies the claims. 
	\end{proof}

%\todo{dual twist functor}
\begin{rmk}
	The definition of the $\A$-bimodule $I(\Xcat)$ in~\eqref{eq:gen}
	is motivated by the isomorphism $\Tw_{\Scat(j_a)}^{-1}(\A) \cong I(\Xcat_a)$ in $\perfd \A$
	and \cite{IR}*{Theorem~6.14}.
%%	For any index $1 \leq a \leq b$ the restriction of the functor $\Tw_{a}$ to the subcategory $\perfd(\A)$ is isomorphic to the
%the right twist functor $\Tw_{\mathcal{S}(a)}$ associated to the exceptional cycle  induces by the simple module at transition vertex $j_a$.
%
%	there is a triangle
%	\begin{align*}
%		\begin{td}
%			\Tw_{j}(P) \ar{r} \& P \ar{r}{\mathrm{coev}} \& S \ar{r} \& \Tw_{j}(P)[1]  \end{td}
%	\end{align*}
\end{rmk}

\subsection{The square of the derived Nakayama functor}
In this subsection, we consider the full subcategory 
%\begin{align*}
$	\Scat \colonequals \add \{ \ S_j [i] \mid j\in Q_0^t, \ i\in\Z \}$
%\end{align*}
of the category $\perfd \A$.
In the following, 
$\Scat * \Scat$ denotes the full subcategory of all objects in $\perfd \A$  which are isomorphic to
cones of morphisms of the subcategory $\Scat$.

\begin{prp}\label{prp:nu2}
For any $X \in \D^-(\md \A)$ there is a triangle
\begin{align}\label{eq:nu2}
	\begin{td}
	\nu^2(X) \ar{r}{\triangle_X} \& X \ar{r} \& C_X \ar{r} \& \nu^2(X)[1]
	\end{td}
	\end{align}
	with $C_X \in \Scat * \Scat$ which is functorial in $X$.
	\end{prp}
\begin{proof}
Throughout this proof we abbreviate the derived tensor product $\lotimes$ by $\otimes$.
	The two-sided ideal $I \colonequals I(Q_0^t)$ gives rise to  
	the triangle 
	\begin{align*}
		\begin{td}
			I \ar{r} \& \A \ar{r} \& \A/I \ar{r} \& I[1]
			\end{td}
		\end{align*}
		in the derived category of $\A$-bimodules.
	Applying $\blank \otimes X$ and 
	$\blank \otimes I \otimes X$
	yields 
 the triangles in the first and the third row of the diagram
 \begin{align*}
 	\begin{td}
 I \otimes I \otimes X \ar[equal]{d} \ar{r}{\phi_{I \otimes X}} \& I \otimes X \ar{d}{\phi_X} \ar{r} \& \A/I \otimes I \otimes X \ar[dashed]{d} \ar{r} \& I \otimes I \otimes X [1] \ar[equal]{d}\\
 		 		 		I \otimes I \otimes X  \ar{d}{\phi_{I \otimes X}} \ar{r} \&  X \ar{r} \ar[equal]{d} \& C'_X \ar{r} \ar[dashed]{d} \& I \otimes I \otimes X [1] \ar{d}
 		\\
I \otimes X \ar{r}{\phi_X}  \& X  \ar{r} \& \A/I \otimes X \ar{r}  \& I \otimes X [1] 
 	\end{td}
 	\end{align*}
 in $\D^-(\md \A)$. 
 Using the octahedral axiom, the third column can be continued to a triangle.
 Since $\A/I \otimes Y \in \Scat$ for any $Y \in \D^-(\md \A)$ this shows that $C'_X \in \Scat * \Scat$.
 By Theorem~\ref{thm:nu-factors} and Lemma~\ref{lem:tw-comm},
 there is an isomorphism $I \otimes I \otimes \blank \cong \nu^2(\blank)$ of auto-equivalences of $\D^-(\md \A)$. Therefore, 
 the triangle in the second row 
 is isomorphic to the triangle in \eqref{eq:nu2}.
 As $\phi_X$ and $\phi_{I \otimes X}$ are functorial in $X$, so is the morphism $\triangle_X$.
	\end{proof}

\begin{lem}\label{lem:lvl2}
The following statements hold.
	\begin{enumerate}
		\item \label{lvl2a} 	For any $j \in Q_0^t$  the middle term $Y_j$ of the Auslander--Reiten triangle
		\begin{align*}
			\begin{td} \nu(S_j) \ar{r} \& Y_j \ar{r} \& S_j \ar{r}{\delta_j} \&  \nu(S_j)[1] \end{td}
		\end{align*}
		in $\perfd(\A)$ is indecomposable and satisfies	$\HomD(\nu(Y_j),Y_j) \neq 0$.
		%		In particular, $S_j$ lies at the mouth of an AR-component of type $\Z \mathbb{A}_{\infty}$ in $\perfd(\A)$.
		\item The indecomposable objects of $\Scat * \Scat$ are given by  $\{  S_j[i], Y_j[i] \mid j \in Q_0^t, i \in \Z\}$.
	\end{enumerate}
\end{lem}
\begin{proof}
	\begin{enumerate}
		\item 
		Let $\alpha_{\ell} \ldots \alpha_2 \alpha_1$, $\beta_{m} \ldots \beta_2 \beta_1$ and $\gamma_{n} \ldots \gamma_2 \gamma_1$ denote the  forbidden threads starting at $\kappa^{-1}(j)$, $j$ and $\kappa(j)$ respectively.
		Because $\nu(S_j)[1] \cong S_j[m]$ and
		$\MD \Hom(S_j,\nu(S_j)[1])\cong  \End(S_j) \cong \kk$, 
		we may assume that the morphism $\delta_j$ has only one non-zero component which is given by the identity of $P_{\kappa(j)}$ at degree $m$.
		Using Notation~\ref{not:lin-string},
		it follows that $Y_j$ is
		isomorphic to the projective complex 
%%		$X_{f_3 f_2}$ 
		depicted by the first row, the projective resolution of $S_j$ is given by the second row and 
		$\nu^{-1}(Y_j)$ corresponds to the last row in the diagram
		\begin{align*}
			\begin{td}
				\circ \ar{r}{\gamma_{n}} \& \bt \ar[densely dotted,-]{r}  \& \bt\ar{r}{\gamma_2} \&  \bt \ar{d}{\gamma_1} \ar{r}{\gamma_1 \beta_{m}} \& \bt \ar[equal]{d} \ar[densely dotted,-]{rr} \& \& \bt \ar[equal]{d} \ar{r}{\beta_1} \& \circ 
				 \ar[equal]{d}\\
 				 \&  \&  \& \circ \ar[equal]{d} \ar{r}{\beta_{m}} \& \bt \ar[densely dotted,-]{rr} \ar[equal]{d}\&  \& \bt \ar[equal]{d} \ar{r}{\beta_1} \& \circ \ar{d}{\alpha_\ell} \\
% 				 \nu^{-1}
 				 \&  \&  \& \circ \ar{r}{\beta_{m}} \& \bt \ar[densely dotted,-]{rr}\& \& \bt \ar{r}{\beta_1 \alpha_\ell} \& \bt \ar{r}{\alpha_{\ell-1}} \& \bt \ar[densely dotted,-]{r} \& \bt \ar{r}{\alpha_1} \& \circ
			\end{td}
		\end{align*}
The vertical arrows of this diagram yield a composition of morphisms of complexes $Y_j \to S_j \to \nu^{-1}(Y_j)$ which is not homotopic to zero. Therefore, it holds that $\HomD(\nu(Y_j),Y_j) \neq 0$.
A computation in the homotopy category shows that $\EndD(Y_j) \cong \kk$, which yields that the complex $Y_j$ is indecomposable.
%		\cong \HomD(Y_j,\nu^{-1}(Y_j)) \neq 0$.
		\item
		Any non-zero morphism of indecomposable objects in $\Scat$ 
is either given by  $\phi \colonequals \lambda \id \colon S_j \to S_j$ 
		or $\psi \colonequals \mu \,\delta_j \colon S_j \to \nu(S_j)[1] = \SD(S_j)$ with $\lambda,\mu \in \kk^*$.
		Any morphism $S_j \to S_j \oplus \SD(S_j)$ or $S_j \oplus \SD(S_j) \to \SD(S_j)$ is decomposable since
		\begin{align*}
		\begin{bmatrix}
			\phi^{-1}  & 0 
			\\ -\psi & \phi
			\end{bmatrix}
			\begin{bmatrix}
				\phi \\ \psi
				\end{bmatrix}
				=
				\begin{bmatrix}
					\id \\ 0
				\end{bmatrix},
				&&
				\text{respectively}
				&&
				\begin{bmatrix}
					\psi & \phi
				\end{bmatrix}
				\begin{bmatrix}
					\phi  & 0 
					\\ -\psi & \phi^{-1}
				\end{bmatrix}
				=
				\begin{bmatrix}
					0 & \id
				\end{bmatrix}.
				\end{align*}
Using that $\HomD(\SD(S_j), S_j) = 0$, similarly, any endomorphism of $S_j \oplus \SD(S_j)$ decomposes as well.
Extending these considerations to more summands it follows that for any morphism 
		$f \colon X \to Y$ 
		of objects in $\Scat * \Scat$
%		decomposes into 
		there are automorphisms $\alpha$ of $X$ and $\beta$ of $Y$ such that $\alpha  f\beta$ is the direct sum 
		of 
		shifts of morphisms of the form $S_j \to 0$, $0 \to S_j$,
		the identity of $S_j$ or the morphism $\delta_j$
		with $j \in Q_0^t$. 	
		Together with \eqref{lvl2a}, this yields the claim on indecomposables in $\Scat * \Scat$. \qedhere
	\end{enumerate}	
\end{proof}

%\begin{prp}
%	For any two derived equivalent $\Rx$-orders $\A$ and $\B$ it holds that
%	 
%	\end{prp}
%	\begin{proof}
%		
%		\end{proof}

%\subsection{Two applications to fractionally Calabi-Yau objects}

\begin{prp}\label{prp:T1}
%		 Assume that the gentle order $\A$ is not hereditary.
For an object $X$ in $\D^-(\md \A)$ we consider the following conditions.
				\begin{enumerate}
			\item \label{T1a} $X \in \mathstrut^{\perp}\Scat$.
%			 is isomorphic to an object in $\Hot^-(\add B)$. 
			\item \label{T1b} The natural transformation $\Delta_X$ from \eqref{eq:nu2} is an isomorphism.
			\item \label{T1c} The object $X$ is $\nu^2$-invariant.
			\item \label{T1d} The object $X$ is $\nu$-periodic.
		\end{enumerate}
	Then the implications \eqref{T1a} $\Rightarrow$
	\eqref{T1b} $\Rightarrow$ \eqref{T1c} $\Rightarrow$ \eqref{T1d} hold. If the gentle order $\A$ is not hereditary, all four conditions are equivalent.
\end{prp}
\begin{proof}
	Assume that \eqref{T1a} holds. Then $X \in \mathstrut^\perp \Scat = \mathstrut^{\perp }\tria(\Scat) \subseteq \mathstrut^{\perp} \Scat * \Scat$, and Proposition~\ref{prp:nu2} implies \eqref{T1b}. 
%	In fact, the implications  \eqref{T1a} $\Rightarrow$
The implications \eqref{T1b} $\Rightarrow$ \eqref{T1c} $\Rightarrow$ \eqref{T1d} follow directly..

 To show \eqref{T1d} $\Rightarrow$ \eqref{T1a} assume that  $\A$ is not hereditary and that $X$ is $\nu$-periodic. Let $j \in Q_0^t$.
Then  $S_j \in \perfd(\A)$ is $(m(j), n(j))$-Calabi--Yau by Lemma~\ref{lem:sim} and $m(j) > n(j)$ by Proposition~\ref{prp:hered-char}.
In particular, $\HomD(X,S_j[i])=0$ for $i \ll 0$. 
As $X$ is $(p,p)$-Calabi--Yau for certain $p \in \N^+$, Lemma~\ref{lem:CY} yields that $X \in \mathstrut^\perp S_j$. 
%%Varying $j 
%\in Q_0^t$, it follows that $X \in \mathstrut^{\perp} \Scat$.
%	Therefore the minimal complex isomorphic to $X$ does not have direct summands of the form $P_j$ at any degree.
%Since this is true for any $j \in Q_0^t$, it follows that $X \in \Hot^-(\add B)$.
%	
%	Equivalently, the minimal projective complex $P \in \Hot^-(\proj \A)$ isomorphic to $X$ 
%	may not contain an indecomposable projective of the form $P_j$. As this holds for any $j \in Q_0^t$ it follows that $P \in \Hot^-(\add \A e)$, which shows the claim.
\end{proof}

\section{Invariants on derived equivalent gentle orders}
\label{sec:invar}

As in the previous section, $\A$ denotes a gentle order such that its underlying quiver $(Q,I)$ is connected.
%According to the initial convention, any equivalence of triangulated categories is assumed to be an exact functor.

\subsection{A stable $t$-structure invariant under the Serre functor}

For the gentle order $\A$ we consider the idempotent determined by all crossing vertices, and its corresponding  projective $\A$-module,  idempotent subalgebra and quotient algebra
\begin{align}\label{eq:idem}
	e \colonequals \sum_{i \in Q_0^c} e_i, && B \colonequals \A e, && \begin{td} \Ainf \colonequals e \A e \ar[hookrightarrow]{r} \& \A \ar[->>]{r} \& \AI \colonequals \A/ \A e\A.\end{td}
\end{align}

%Let us recall that
%the category 
%$\D^-(\md \A)$
%admits a stable $t$-structure $(\aisle,\coaisle)$ 
%with $\aisle$ equivalent to 
%$\D^-(\md e\A e)$
%of an idempotent subalgebra and
%$\coaisle$ given by the full subcategory
%$\Dh(\md \A)$
%of complexes with homology in a certain hereditary category $\Hcat$.

%This idempotent yields the maximal direct summand $\A e$ of $\A$ such that $\A e$ is $\nu$-invariant, or, equivalently, a projective-injective lattice.
\begin{rmk}
We make a few basic remarks.
\begin{enumerate}
	\item The subalgebra $\Ainf$ is a ribbon graph order or zero.
	\item The quotient algebra $\AI$ is hereditary and uniserial. 	
	More precisely, the quiver of $\AI$ is either a product of quivers of equioriented Dynkin type 
	$\mathbb{A}$ or a quiver of equioriented Euclidean type $\widetilde{\mathbb{A}}$.
	\item 
	If $e \neq 0$, 
	it holds that $\gldim \Ainf = \infty \geq \gldim \A > \gldim \AI = 1$.
%	In particular, the ring $\A$ is ``less singular'' than $\Ainf$. 
	%	Because of this, we will \Hcat \colonequals \md \AI
	%has a hereditary ring factor.
	\item The module $B$ is a \emph{bijective $\A$-lattice}, that is, a projective-injective object in the \emph{category of $\A$-lattices}, which can be defined as the full subcategory given by submodules of finitely generated projective $\A$-modules.
\end{enumerate}
\end{rmk}
According to Proposition~\ref{prp:recol}
the two subcategories
$$
\begin{td}
\aisle  \colonequals	\Hot^-(\add B) 
 \ar[yshift=0pt,hookrightarrow]{r}
	\& 
	\D^-(\md \A)
	\ar[yshift=0pt,<-,hookleftarrow]{r} \&  
 \Dh(\md \A)\equalscolon \coaisle 
\end{td}$$
form a stable t-structure $(\aisle,\coaisle)$ of $\D^-(\md \A)$. 
%The stable $t$-structure in \eqref{recol1} can also be phrased in terms of homotopy  categories.
%\begin{align*}
%	\begin{td}
%		\aisle = \Hot^-(\add 
%		B)
%		 \Hot^-_{\Hcat}(\proj \A) \ar[yshift=0pt,hookrightarrow]{r}
%		\& 
%		\ar[yshift=0pt,<-,hookleftarrow]{r} \&   		\Hot^-(\proj \A)  \cong \coaisle.
%	\end{td}
%\end{align*}
	Let $\Tcat$ denote the full subcategory of $\nu$-periodic objects,
and $\Fcat$ the full subcategory of all non-$\nu$-periodic fractionally Calabi-Yau objects
in $\D^-(\md \A)$.
A priori, Lemma~\ref{lem:CY} yields that $\Hom(\Tcat,\Fcat)=0$.
Next, we compare the stable $t$-structure $(\aisle,\coaisle)$ with the pair $(\Tcat,\Fcat)$.

\begin{prp} \label{prp:recol4} 
In the notations above, the following statements hold.
	\begin{enumerate}
		\item It holds that  $\aisle = \Tau$. In particular, the relative Serre functor $\SD$ restricts to an auto-equivalence of $\aisle$ and the subcategory $\Tau$ is triangulated. 
		\item \label{recol4b} It holds that $\Fcat \subseteq \coaisle$ and the functor $\SD$ restricts to an auto-equivalence of $\coaisle$.
	\end{enumerate}
\end{prp}
\begin{proof}
	Theorem~\ref{thm:SD} implies that 
	$\nu$ preserves 
	$\Hot^-(\add B)$
	and acts as an involution on its objects. 
	This shows that $\SD(\aisle) = \nu(\aisle)[1] = \aisle \subseteq \Tcat$.
	To show $\aisle \supseteq \Tcat$, let $X$ be $\nu$-periodic.
	\begin{itemize}
		\item If $\A$ is hereditary, then $e = 0$ and $X \in \D^-(\md \A) = \aisle$.
%		\item If $\A$ is non-hereditary, Proposition~\ref{prp:T1} yields that $X \in \mathstrut^{\perp} S$,  which is equal to $\Ucat$ by Proposition~\ref{prp:recol}~\eqref{recol3b}.
		\item Assume that $\A$ is not hereditary. 
Let $j \in Q_0^t$.
Then  $S_j \in \perfd(\A)$ is $(m(j), n(j))$-Calabi--Yau with $m(j) \geq n(j)$ by Lemma~\ref{lem:sim}.   In particular, $\HomD(X,S_j[i])=0$ for $i \ll 0$. Proposition~\ref{prp:hered-char} ensures that $m(j) \neq n(j)$.
As $X$ is $(p,p)$-Calabi--Yau for certain $p \in \N^+$, Lemma~\ref{lem:CY} yields that $X \in \mathstrut^\perp S_j$. 
It follows that $X \in \mathstrut^{\perp} \Scat$ which is equal to $\aisle$ by Proposition~\ref{prp:recol}~\eqref{recol3b}.  
	\end{itemize}
		Since $\nu^{\pm 1}(B) \cong B $ and $\coaisle = B^\perp$ it follows that $\SD^{\pm 1}(\coaisle) \subseteq \coaisle$, and thus $\SD(\coaisle) = \coaisle$.
	As $B \in \per(\A)$, it follows that  $\Fcat \subseteq B^{\perp} = \coaisle$ by Lemma~\ref{lem:CY}.
\end{proof}

\begin{prp}\label{prp:der-gen}
Let $\B$ be a gentle order derived equivalent to the gentle order $\A$. Let $f$ denote the sum of the idempotents of the crossing vertices in $Q^{\B}_0$, set $\B_f \colonequals f \B f$ and $B \colonequals \B/\B f \B$.
Then the following statements hold.
\begin{enumerate}
	\item \label{der-gen1} There is a diagram of left recollements
\begin{align*}
%	\label{eq:recol}
	\begin{td} 
		\aisle_{\A} = \D^-(\md \Ainf)
		\ar[dashed]{d}{\Phi_{\aisle}}
		  \& \ar[yshift=5pt,twoheadrightarrow]{r}{\itop_{\A}} \D^-(\md\A) \ar{d}{\Phi}
		\ar[yshift=-5pt,twoheadrightarrow]{l}{\jmid_{\A}} \ar[yshift=5pt,leftarrowtail]{l}[swap]{\jtop_{\A}}
%		\ar{d}{\Phi_2}
		\& \ar[yshift=-5pt,rightarrowtail]{l}{\imid_{\A}}  \Dh(\md\A) =   \coaisle_{\A}  \ar[dashed]{d}{\Phi_{\coaisle}} \\
				\aisle_{\B} = \D^-(\md \B_f)
		\& \ar[yshift=5pt,twoheadrightarrow]{r}{\itop_{\B}} \D^-(\md\B) 
		\ar[yshift=-5pt,twoheadrightarrow]{l}{\jmid_{\B}} \ar[yshift=5pt,leftarrowtail]{l}[swap]{\jtop_{\B}}
		\& \ar[yshift=-5pt,rightarrowtail]{l}{\imid_{\B}}  \D_B^{-}(\md\B) =   \coaisle_{\B}   
	\end{td}
\end{align*}
such that $\Phi_{\aisle}$, $\Phi$ and $\Phi_{\coaisle}$ are equivalences
and there are isomorphisms of functors
\begin{align*}
	\Phi \circ\jtop_{\A} \cong \jtop_{\B}\circ \Phi_{\aisle},
	&&
	\jmid_{\B} \circ \Phi \cong \Phi_{\aisle} \circ \jmid_{\A},&&
	\imid_{\B} \circ\Phi_{\coaisle} \cong \Phi \circ \imid_{\A},&&	
	\Phi_{\coaisle}\circ \itop_{\A} \cong  \itop_{\B}\circ \Phi.&&
\end{align*}
	\item \label{der-gen2} $\A$ is hereditary if and only if $\B$ is hereditary.
	\item \label{der-gen3} $\A$ is a ribbon graph order if and only if $\B$ is a ribbon graph order.
\end{enumerate} 
\end{prp}
\begin{proof}
	The top and bottom recollement exist by Proposition~\ref{prp:recol}, and the functor $\Phi$ exists by Theorem~\ref{thm:der-eq}.
Proposition~\ref{prp:der-ord}~\eqref{der-ord1}
%	states that $\Phi \circ \nu_{\A}(X) \cong \nu_{\B} \circ \Phi$, which 
	implies that $\Phi$ preserves $\nu$-periodic objects.
%	right-bounded complexes.
Using the identification of such objects with aisle subcategories in Proposition~\ref{prp:recol4} it follows that
	$\Phi(\Ucat_{\A}) = \Ucat_{\B}$.
	Since $\gldim \A$ is finite if and only if so is $\gldim \B$, 
	Proposition~\ref{prp:hered-char} implies that $\A$ is hereditary if and only if $\B$ is hereditary.
	
	\begin{itemize}
		\item If $\A$ and $\B$ are hereditary, it holds that $e = f = 0$, the functors  $\itop_{\A}$, $\imid_{\A}$ 
		$\itop_{\B}$, and $\imid_{\B}$
		are identity functors, and $\Phi = \Phi_{\coaisle}$.
		\item Assume that neither $\A$ nor $\B$ is hereditary.
 Proposition~\ref{prp:recol4} yields that
		$\aisle_{\A} = \im F_{\A} = \Tau$ and similarly $\aisle_{\B} = \Tau$.
		The previous consideration shows that
%	 Proposition~\ref{prp:nu-comm}
%	 states that $\Phi \circ \nu_{\A} \cong \nu_{\B} \circ \Phi$, which implies that 
	 $\Phi(\aisle_{\A}) = \aisle_{\B}$.
		Therefore, $\Phi$ restricts to an equivalence $\Phi_{\aisle}$ such that $\Phi  \circ \jtop_{\A} \cong \jtop_{\B} \circ \Phi_{\aisle}$.
		Since $\jmid_{\A} \vdash \jtop_{\A}$ 
		and $\jtop_{\B} \dashv \jmid_{\B}$,
		it holds that
		$\Phi_{\aisle} \jmid_{\A} \vdash \jtop_{\A} \Phi_{\aisle}^{-1} \cong \Phi^{-1} \jtop_{\B} \dashv \jmid_{\B}\Phi$.
		As left adjoints are unique it follows that $\Phi_{\aisle} \jmid_{\A} \cong \jmid_{\B} \Phi$.

		Because $\Phi$ is an equivalence it holds that $\Phi(\coaisle_{\A}) = \Phi(\aisle_{\A}^{\perp}) = \Phi(\aisle_{\A})^{\perp} = \aisle_{\B}^{\perp} = \coaisle_{\B}$. 
		Therefore, $\Phi$ restricts also to an equivalence $\Phi_{\coaisle}$ with $\Phi \circ \imid_{\A} \cong \jmid_{\B} \circ \Phi_{\coaisle}$.
		Similar to the previous argument, the adjointness properties
	 $\itop_{\A} \dashv \imid_{\A}$ 
		and $\imid_{\B} \vdash \itop_{\B}$ yield that 
		$\Phi_{\coaisle} \itop_{\A} \dashv \imid_{\A} \Phi^{-1}_{\coaisle} \cong \Phi^{-1} \imid_{\B} \vdash \itop_{\B} \Phi$, and thus $\Phi_{\coaisle} \itop_{\A} \cong \itop_{\B} \Phi$.
%		 \qedhere
		\end{itemize}
	This shows all isomorphisms of compositions.
	To show the remaining claim, assume that $\A$ is a ribbon graph order. Then $e=1$ and $\coaisle_{\A} = 0$.
	Since $\Phi_{\coaisle}$ is an equivalence, it follows that $\coaisle_{\B}=0$, $f=1$ and $\B$ is a ribbon graph order.
	\end{proof}
\begin{rmk}
Proposition \ref{prp:der-gen}~\eqref{der-gen2} recovers an observation by König and Zimmermann \cite{KZ}. 
		\eqref{der-gen3} confirms an expectation by Kauer and Roggenkamp \cite{KR}. In a forthcoming work \cite{G} it will be shown that \eqref{der-gen3} is true without assuming that $\A$ and $\B$ are $\Rx$-orders.
		\end{rmk}
%
%For the gentle order $\A$
%we define the multiset of AG-invariants   
%\begin{align*}
%	\AG(\A)\colonequals \AG_1(\A) \cup \AG_{2'}(\A) \cup \AG_{2''}(\A)
%	\end{align*}
%Due to our conventions, the unions are disjoint.

\subsection{Homological interpretation of AG-invariants}

\begin{prp}\label{prp:exc}
	An object $X \in \D^-(\md \A)$ appears in an exceptional cycle 
	if and only if 
%	one 
%	the following holds.
%	\begin{itemize}
%		\item 
$X$ is isomorphic to a shift of a simple module
$S_j$
%with
% = \{ \ S_j[i] \ \mid 
for certain
$j \in Q_0^t$,
%and $ i \in \Z$
%\item 
or
		$X \in \aisle$ and $X$ is $1$-spherical or induces an exceptional $2$-cycle.
%		\item $X \in \Scat = \{ S_j[i] \mid j \in Q_0^t, i \in \Z\}$.
%		\end{itemize}
%	$X$ is $1$-spherical, appears in an exceptional $2$-cycle
%	or $X \in \Scat$.
%\begin{enumerate}
%		\item  \label{excV} $X \in \coaisle$ and $X$ is isomorphic to $S_j$ for certain $j \in Q_0^t$ up to shift.
%		\item \label{excV} $X \in \coaisle$ and $X$ is isomorphic to $S_j$ for certain $j \in Q_0^t$ up to shift.
%\end{enumerate}
\end{prp}
\begin{proof}
As above, let 
$	\Scat \colonequals \add \{ \ S_j [i] \mid j\in Q_0^t, \ i\in\Z \}$.
The `if'-implication follows from Lemma~\ref{lem:sim}.
To show the converse, assume that $X$ appears in an exceptional cycle. Set $(m,n) \colonequals \CYdim(X)$.
% and $X$ is indecomposable.
\begin{itemize}[label=--]
	\item Assume that $(m,n) \notin \{(1,1),(2,2)\}$. 
	By Lemma~\ref{lem:frac-bound} it holds that $m \geq n$. In particular, it holds that $(m,n) \neq (0,1)$ and $X$ is indecomposable.
	Lemma~\ref{lem:vanish} implies that $\HomD(\nu^q(X),X)=0$ for each $q \in \{1,2\}$.
	Proposition~\ref{prp:nu2} and Lemma~\ref{lem:lvl2} yield that $X \in \ind \Scat$.
	%			 \cong S_j [d]$ for certain $j\in Q_0^t$, $d \in Z$.
	\item Assume that $(m,n) \in \{(1,1),(2,2)\}$.
	\begin{itemize}[label=--]
		\item If $\A$ is hereditary, 
		%		it holds that $|Q_0| = n$, $\coaisle = \D^-(\md \A)$
		then
		%		and
		$X \in \simp \A \subseteq \Scat$ by Proposition~\ref{prp:hered}.
		\item Otherwise, it follows that $X \in \mathstrut^{\perp} \Scat = \aisle$ by Propositions~\ref{prp:T1} and \ref{prp:recol} \eqref{recol3b}. \qedhere
	\end{itemize}
\end{itemize}

\end{proof}

Let us recall that two exceptional cycles are equivalent if they are the same up to shift and rotation, see Definition~\ref{dfn:eq-exc}.
%We will call two exceptional cycles $\Exc_1$ and $\Exc_2$ \emph{strongly equivalent}, denoted by $\Exc_1 \approx \Exc_2$, if $\tria(\Exc) = \tria(\Exc')$,
%that is, the objects of $\Exc_1$ generate the same triangulated category as the objects $\Exc_2$.

%Let $(E^{(i)})_{i\in I}$ denote a complete set of representatives of
%fractionally Calabi-Yau objects satisfying \eqref{E2} in $\perfdH \A$
%up to shift and $\SD$.
%% exceptional sequences in $\Dcat$ up to shift and $\SD$.
%We will denote the by 
%\begin{align*}
%	\CY(\perfdH \A) \colonequals f
%	[ \CYdim E^{(i)} \mid i \in I ]
%\end{align*}
%the multiset of Calabi--Yau-dimensions of such objects.

\begin{cor}\label{cor:AG1}
%The map $\mathsf{s} \colon Q_0^t \to  \exc(\coaisle)$, $j \mapsto 
%\mathcal{S}(j) = \left(S_j,S_{\kappa(j)}, \ldots S_{\kappa^{n(j)-1}(j)}\right)$
The following statements hold.
\begin{enumerate}
	\item \label{AG1a} There is a bijection between the set
$Q_0^t/\!\sim_{\kappa}$ of $\kappa$-orbits in $Q_0^t$
and the set $\exc(\coaisle)/\!\sim$ of equivalence classes of repetition-free exceptional cycles in $\coaisle$  
given by
\begin{align}
	\label{eq:AG1-bij}
	\begin{td}
	Q_0^t/\!\sim_{\kappa} \ar{r}{\sim} \& \exc(\coaisle)/\!\sim 
	\&
	\overline{j} \ar[mapsto]{r} \& \overline{\mathcal{S}(j)} = \overline{\left(S_j,S_{\kappa(j)}, \ldots S_{\kappa^{n(j)-1}(j)}\right)}
	\end{td}
	\end{align}
%\begin{itemize}[label=--]
%	\item the set 
%	\item 
%	%		\item 
%	%		the set $\exc(\coaisle)/\approx$ of strong equivalence classes of exceptional cycles in $\coaisle$  
%\end{itemize}
\item \label{AG1b} There is an equality of multisets
\begin{align*}\AG_1(\A) = 
	%\CY_1(\coaisle)
	%\colonequals 
	\left\{\CYdim \Exc \mid \overline{\Exc} \in \exc(\coaisle)/\!\sim\right\}.
\end{align*}
%\item \label{AG1c} If $\B$ is a gentle order derived equivalent to $\A$, then $\AG_1(\A) = \AG_1(\B)$.
\end{enumerate}
\end{cor}
\begin{proof}
%	\begin{enumerate}
%		\item
By Lemma~\ref{lem:sim}
the map $ Q_0^t \to  \exc(\coaisle)$, $j \mapsto 
\mathcal{S}(j)$
%, $j \mapsto 
%\mathcal{S}(j) = \left(S_j,S_{\kappa(j)}, \ldots S_{\kappa^{n(j)-1}(j)}\right)$
is well-defined.
%\ldots S_{\kappa^{n(j)-1}}(j))
%\begin{align*}
%	\begin{td} 
%		Q_0^t \ar{r} \& \exc(\coaisle)
%		\&		
%		j \ar[mapsto]{r} \& \mathcal{S}(j) = (S_j,\ldots S_{\kappa^{n(j)-1}}(j))
%	\end{td}
%\end{align*}
For any $i,j \in Q_0^t$
it holds that
$i \sim_{\kappa} j$ if and only if $S_i \cong S_{\kappa^p(j)}[q]$ for some $p,q \in \Z$
if and only if $\mathcal{S}(i) \sim \mathcal{S}(j)$. 
Because of this, the induced map in \eqref{eq:AG1-bij} is well-defined and injective.
%$
%\bar{\mathsf{s}}\colon
% Q_0^t/\!\sim_{\kappa} \ \longrightarrow \  \exc(\coaisle)/\!\sim$,
%\begin{align*}
%	\begin{td} Q_0^t/\!\sim_{\kappa} \ar{r} \& \exc(\coaisle)/\!\sim \&
%		{[j]_{\kappa}} \ar[mapsto]{r} \& {[\mathcal{S}(j)]}
%	\end{td}
%\end{align*}
%\begin{align*}
%	\begin{td} Q_0^t/\!\sim_{\kappa} \ar{r} \& \exc(\coaisle)/\!\sim \&
%		{[j]_{\kappa}} \ar[mapsto]{r} \& {[\mathcal{S}(j)]}
%	\end{td}
%\end{align*}
%is well-defined and injective. 
It is is surjective by Proposition~\ref{prp:exc}.
which shows
\ref{AG1a}.
%\item 
The equality of multi-sets in \eqref{AG1b} follows then using that
 $\CYdim S_j = (m(j),n(j)$
for any $j \in Q_0^t$ from Lemma~\ref{lem:sim}.
\end{proof}

%\subsection{Singularity category}
%Each arrow $\alpha \in Q_1$ appears either in a forbidden thread or a forbidden cycle.
%Let $Q_1^{fc}$ denote the subset of arrows of $Q$ appearing in a forbidden cycle.
%For any such arrow $\alpha$, there is a unique arrow $\phi(\alpha) \in Q_1^{fc}$ with $s(\phi(\alpha)) = t(\alpha)$ and $\phi(\alpha) \alpha \in I$. Iterating this procedure yields 
%the unique forbidden cycle $f_{\alpha}$ starting with $\beta_i = \phi^{i-1}(\alpha)$.
To give a similar interpretation of AG-invariants of second type of the gentle order $\A$, 
we recall the description of its singularity category.
By the general considerations in \ref{subsec:dsg}, the singularity category $\Dsg(\A)$ is Hom-finite and has a Serre functor $\overline{\nu}$, which is induced by the derived Nakayama functor.

\begin{thm}[\cite{GIK}*{Theorem~10.13}] \label{thm:dsg} 
For the gentle order $\A$ the following statements hold.
%	The singularity category $\Dsg(\A)$ of the gentle order $\A$ has the following description.

\begin{enumerate}
	\item \label{dsg1} 
	There is a bijection between the set $Q_1^{fc}$ of arrows in $Q$ which appear in forbidden cycles 
	and the set of isomorphism classes of indecomposable objects in the category $\Dsg(\A)$ given by
	$$ \begin{td} Q_1^{fc} \ar{r}{\sim} \& \ind \Dsg(\A), \& \beta \ar[mapsto]{r} \& L_{\beta} \colonequals \A \beta \end{td} $$
	where $L_{\beta}$ is viewed as an object in the category $\Dsg(\A)$.
	\item \label{dsg2} For any arrow $\alpha \in Q_1^{fc}$ there are isomorphisms
\begin{align*}
%	\label{eq:shift}
	\overline{\nu}(L_{\beta}) \colonequals \omega \lotimes L_{\beta} \cong L_{\beta} \quad
	\text{and} \quad L_{\beta}[-1] \cong L_{\varrho(\beta)}.
\end{align*}
	where $\varrho(\beta)$ denotes the unique arrow with $s(\varrho(\beta))=t(\beta)$ and $\varrho(\beta)\beta \in I$.
	\item \label{dsg3} For any arrows $\alpha, \beta \in Q_1^{fc}$ it holds that $\dim \Hom_{\Dsg(\A)}(L_{\alpha},L_{\beta}) = \delta_{\alpha \beta}$.
\end{enumerate}
\end{thm}
We refer to \cite{Kalck} for the counterpart of Theorem~\ref{thm:dsg} for finite-dimensional gentle algebras, and to 
\cite{Bennett-Tennenhaus} for a statement equivalent to \eqref{dsg1}.

\begin{rmk}
	By arguments due to Buchweitz \cite{Buchweitz}*{Section~6},
the singularity category $\Dsg(\A)$ is equivalent to the homotopy category $\Hotac(\proj \A)$ of acyclic projective complexes.
Using Notation~\ref{not:lin-string},
the acyclic complex corresponding to the left ideal $L_{\beta}$ of an arrow $\beta \in Q_1^{fc}$ can be represented by the periodic diagram
\begin{align*}
%	\label{eq:periodic-res}
	\begin{td}
		\mathstrut \ar[densely dotted,-]{rr} \& \& \bt \ar{r}{\beta_2} \& \bt \ar{r}{\beta_1} \& \bt \ar{r}{\beta_m} \& \bt \ar[densely dotted,-]{r}  \& 
%		\bt 
%		\ar{r}{\beta_3} 
%		\& 
		\bt \ar{r}{\beta_2} \& \bt \ar{r}{\beta_1} \& \bt \ar{r}{\beta_m} \& \bt \ar[densely dotted,-]{rr} \& \& \mathstrut
	\end{td}
\end{align*}
with $\beta_{p+1} \colonequals \varrho^{p}(\beta)$ for each index $0 \leq p < m$, where $m \colonequals m(\beta) = \ell(fc(\beta))$.
%such that $fc(\beta) = \varrho^{m-1}(\beta) \ldots \varrho(\beta) \beta = \beta_m \ldots \beta_2 \beta_1$.
\end{rmk}

%Note that the last theorem states the category $\Dsg(\A)$
%is essentially determined by the multiset $\AG_{2'}(\A)$.

\begin{cor}\label{cor:AG2}
	The following statements hold.
	\begin{enumerate}
		\item \label{AG2a}
There is a bijection between the set of $\varrho$-orbits in $Q_1^{fc}$
and the set of shift orbits of indecomposable objects in the category $\Dsg(\A)$ given by
$$ \begin{td} Q_1^{fc}/\sim_{\varrho} \ar{r}{\sim} \& \ind \Dsg(\A)/[1]\end{td}$$
\item \label{AG2b} There is an equality of multisets
%	$Q_1^{fc}/\!\sim_{\rho}$
%	and 
%	$\ind \Dsg(\A)$ and 
%	fractionally Calabi-Yau 
%	such that
\begin{align*}
	\AG_{2}(\A)
	=  \{\CYdim X \mid 
	X \in \ind \Dsg(\A)/[1]
	\}.
\end{align*}
%\item \label{AG2c} Given an additional gentle order $\B$, the categories $\Dsg(\A)$ and $\Dsg(\B)$ are equivalent if and only if 
%	$\AG_{2}(\A) = \AG_2(\B)$.
	\end{enumerate}
\end{cor} 
\begin{proof}
	\eqref{AG2a} follows directly from Theorem~\ref{thm:dsg} \eqref{dsg1} and \eqref{dsg2}.
For any arrow $\beta \in Q_1^{fc}$, Theorem~\ref{thm:dsg}~\ref{dsg2}	implies  that
	$m(\beta) = \ell(fc(\beta)) = \min\{ p \in \N^+ \mid L_{\beta}[p] = L_{\beta}\}$, and thus $(m(\beta),0) = \CYdim L_{\beta}$, which shows \eqref{AG2b}.	
%	To show \eqref{AG2c} let $\B$ be a gentle order.
%	If $\A$ is singularly equivalent to $\B$, it follows that $\AG_2(\A) = \AG_2(\B)$ using that the equivalence commutes with the shift functor. Vice versa, assume that $\AG_2(\A) = \AG_2(\B)$. 
%		By Theorem~\ref{thm:dsg}~\eqref{dsg3} the triangulated categories $\Dsg(\A)$ and $\Dsg(\B)$ are semisimple
%		with the same finite number of shift-orbits and the same multi-set of shift-orbit lengths.
%Therefore a bijection on their sets of indecomposable objects can be extended to a singular equivalence, which completes the proof of \eqref{AG2c}.
	\end{proof}
\label{rewrite}

\subsection{Combinatorial derived invariants}

\begin{thm}\label{thm:der-inv}
	Assume that $\A$ and $\B$ are derived equivalent ring-indecomposable gentle orders.
Then there are equalities
	\begin{align*}
	\AG_1(\A) = \AG_1(\B), && \AG_{2}(\A) = \AG_{2}(\B),
 &&\pc_{\A} = \pc_{\B}, && \bc_{\A} = \bc_\B,
	\end{align*}
that is, $\A$ and $\B$ have the same multi-sets of AG-invariants, the same number of permitted cycles, and 
the same bicolorability parameter.
\end{thm}
\begin{proof}
Because $\A$ and $\B$ are derived equivalent, 
		their coaisles $\coaisle_{\A}$ and $\coaisle_{\B}$
		are equivalent by Proposition~\ref{prp:der-gen},
		their singularity categories are equivalent 
		by Theorem~\ref{thm:der-eq}, and 
		their semisimple rational hulls $\KK\otimes \A$ and $\KK \otimes_{\Rx} \B$ are derived equivalent 
and their Cartan matrices have the same rank and
		by Proposition~\ref{prp:der-ord}.
		Now the equalities follow from Corollaries~\ref{cor:AG1}~\eqref{AG1b} with~\ref{cor:CYdim}, \ref{cor:AG2}~\eqref{AG2b}, \ref{cor:rhull} and \ref{cor:rk-det2}.
	\end{proof}

\section{Truncated ribbon graphs of gentle orders}
\label{sec:tru}
In this section we introduce truncated ribbon graphs and review the combinatorial invariants associated to the quiver of a gentle order in graph-theoretic terms.

%\begin{rmk}
%	There are bijections between the set of $\rho$-orbits in $Q_1^{fc}$, the set of forbidden cycles up to rotation, the set of punctured faces of $\Gr$,
%	and
%	the set of shift-orbits in $\Dsg(\A)$.
%	In particular,
%	the multiset $\AG_{2'}$ is empty if and only if $(Q,I)$ has no forbidden cycles if and only if the global dimension of $\A$ is finite if and only if $\Surf$ has no punctures if and only if $\Dsg(\A)$ is trivial.
%\end{rmk}
%
%\begin{rmk}
%	There are bijections between the set of $\kappa$-orbits of $Q_0^t$, the set of AG-cycles up to rotation 
%	and the set of boundary faces of the truncated ribbon graph $\Gr_{\A}$. 
%	We note that $\AG_1$ is empty if and only if $Q_0^t$ is empty if and only if $\A$ is a ribbon graph order 
%	if and only if $\Sigma$ has no boundary components.
%\end{rmk}

\subsection{Truncated ribbon graphs}
\label{subsec:trRG}

Let $\sigma, \theta \colon H \to H$ be permutations of a finite set $H$ such that $\theta$ is an involution. 
The pair $(\sigma,\theta)$ gives rise to graph-theoretic data $\Gr_{(\sigma,\theta)} = (\sigma,\theta,\phi, V,E,F,i)$ as follows.

The composition $\phi \colonequals \theta \sigma$ defines a third permutation on the set $H$.
Elements of the set $H$ are called \emph{half-edges}.
The sets $V,E,F$ of \emph{vertices}, \emph{edges} and \emph{faces} are defined as the sets of $\sigma$-orbits, $\theta$-orbits and $\phi$-orbits in $H$, respectively. The \emph{incidence map} $i \colon H \to V$ assigns to each half-edge $h$ its $\sigma$-orbit, denoted by $v_h$.

The data $\Gr_{(\sigma,\theta)}$ will be called a \emph{truncated ribbon graph}.
\subsubsection*{Refined notions of edges and faces}
It will be convenient to use the following terminology.
For any half-edge $h \in H$ its $\theta$-orbit $e_h$ will be called a \emph{truncated edge} if $\theta(h) = h$,
and a \emph{glued edge} otherwise.
A glued edge $e_h$ is a \emph{loop} if $h \neq \theta(h)$ and $v_h = v_{\theta(h)}$,
and otherwise an \emph{ordinary edge} .  
%In the last two cases, we will view $e_h$ as being glued from the half-edges $h$ and $\theta(h)$.
The set of truncated edges and that of glued edges will be denoted by $E_{t}$ and $E_{g}$.

For $h \in H$ its $\phi$-orbit $f_h$ will be called a \emph{boundary face} if 
$f_h$ contains a fixed point of $\theta$, and a \emph{punctured face} otherwise.

For $h \in H$ the permutation $\sigma$ yields a cyclic order on the set of half-edges incident to $v_h$, which is given by
%\begin{align*}
$h < \sigma(h) < \ldots \sigma^{p-1}(h) < \sigma^{p}(h) = h$, where $p$ is the cardinality $|v_h|$ of the $\sigma$-orbit of $v_h$.
%\end{align*}
%When depicting a ribbon graph, the half-edges incident to $v$ will be depicted in clockwise order around $v$ matching the cyclic order above.
The face $f_h$ may be viewed as a polygon enclosed by edges $e_h, e_{\phi(h)}, \ldots, e_{\phi^{m-1}(h)}$ with $m \colonequals f_h$. 
%denotes the cardinality of the $\phi$-orbit $||$. 
In case $f_h$ is a boundary face, the corresponding polygon has $m$ ordinary edges, otherwise the polygon contains vertices incident to two glued edges and one truncated edge. The set of boundary faces and that of punctured faces will be denoted by $F_{p}$ and $F_{b}$, respectively.

The involution $\alpha$ has no fixed points if and only if $\Gr_{(\sigma,\alpha)}$ has no truncated edges if and only if each face of $\Gr_{(\sigma,\alpha)}$ is punctured. In this case, $\Gr_{(\sigma,\alpha)}$ is called a \emph{ribbon graph}. 
To simplify the terminology, we call $\Gr_{(\sigma,\alpha)}$ a truncated ribbon graph even if it has no truncated edges.

%Finally, we denote by $\bc(\Gr)$ the number of bipartite connected components of $\Gr$.
%Let $\Gr^{\circ}$ denote the ribbon graph obtained from $\Gr$ by deleting all truncated edges.
%Then $\bc(\hat{\Gr}) = \bc(\Gr) = \bc(\Gr^{\circ})$.

\begin{rmk}
A vertex in a ribbon graph is called a \emph{leaf} if it is adjacent to only one glued edge and no loop.
Any truncated ribbon graph $\Gr_{(\sigma,\theta)}$ can be obtained from a ribbon graph ${\Gr}_{(\hat{\sigma},\hat{\theta})}$ by deleting any choice of the leaves of ${\Gr}_{(\hat{\sigma},\hat{\theta})}$ together with their incident half-edges, which results in truncated edges in the sense above.
\end{rmk}

\subsubsection*{The underlying truncated graph}

Let $\Gr_{(\sigma,\theta)} = (\sigma,\theta,V,E,F,H,i)$
be a truncated ribbon graph.
$\Gr_{(\sigma,\theta)}$ has an underlying truncated graph $\Gr = (V,E,\mu)$ in the sense of Subsection~\ref{subsec:trunc-gr}, 
which is defined by setting $\mu(e)(v) \colonequals \lvert i^{-1}(v) \rvert$ 
for any edge $e \in E$ and any vertex $v$. 
%Any truncated, ordinary edge or loop of the truncated ribbon graph $\Gr_{(\sigma,\theta)}$ yields the same type of edge in its underlying truncated graph $\Gr$.
The terminology of ordinary edges, truncated edges and loops
of the truncated ribbon graph $\Gr_{(\sigma,\theta)}$ of the previous subsection is consistent with the terminology of these types of edges in the truncated graph $\Gr$.
% of  Subsection~\ref{subsec:trunc-gr}

\subsubsection*{The surface of a truncated ribbon graph}
%\todo{surface and arcs}

Any truncated ribbon graph $\Gr_{(\sigma,\theta)}$ can be embedded in an oriented surface
$\Surf \colonequals \Surf_{(\sigma,\theta)}$ with non-empty boundary and a finite set of marked points $M \colonequals M_{(\sigma,\theta)}$ as follows.
\begin{itemize}[label=--]
	\item The vertices of the truncated ribbon graph are viewed as marked points in the interior of the surface.
	\item Each punctured face of $\Gr_{(\sigma,\theta)}$ corresponds as a punctured polygon. 
	More precisely, a punctured face $f_h$ is viewed as a polygon which is enclosed by edges $e_h, e_{\phi(h)}, \ldots, e_{\phi^{m-1}(h)}$ with $m \colonequals \lvert f_h \rvert$ and has a boundary component without marked points in its interior.
	\item Each boundary face of  $\Gr_{(\sigma,\theta)}$ corresponds to a
	a polygon enclosing a boundary component with marked points.
	More precisely, the glued edges of any boundary face of $\Gr_{(\sigma,\theta)}$ are viewed as the outer edges of the polygon, while each of its truncated edges is incident to a vertex and an associated marked point on the inner boundary of the polygon.
\end{itemize}
%In particular, the surface $\Surf_{\Gr}$ may have any non-zero number of internal marked points. 
The genus $g(\Surf)$ of the surface $\Surf$ is determined by the Euler-Poincar\'{e} formula
\begin{align}
	\label{eq:EP}
	2 - 2 g(\Surf) = \lvert V \rvert - \lvert E \rvert + \lvert F \rvert
	\end{align}
The graph $\Gr_{(\sigma,\theta)}$ is a ribbon graph if and only if its associated surface $\Surf$ has no boundary components with marked points.

\subsection{The truncated ribbon graph of a gentle order}
Let $\A$ be a gentle order and $(Q,I)$ its underlying quiver.

\subsubsection{Pair of permutations}
The quiver $(Q,I)$ gives rise to a pair $(\sigma,\theta)=(\sigma,\theta)_{(Q,I)}$ as follows.
Let $\alpha \in Q_1$. As noted previously,  there is a unique arrow $\sigma(\alpha)$ with $\sigma(\alpha) \alpha \notin I$. 
If $s(\alpha) \in Q_0^c$, there is a unique arrow $\theta(\alpha) \neq \alpha$ with $s(\theta(\alpha)) = s(\alpha)$.
If $s(\alpha) \in Q_0^t$, let $\theta(\alpha) \colonequals \alpha$. 
These prescriptions define a permutation $\sigma$ and involution $\theta$ on the set of arrows $Q_1$.
For formal reasons, it will be convenient to introduce the set $H \colonequals \{h_{\alpha} \mid \alpha \in Q_1\}$
whose elements are permuted via $h_{\alpha} \mapsto h_{\sigma(\alpha)}$ and $h_{\alpha} \mapsto h_{\theta(\alpha)}$. 
We denote these permutations again by $\sigma,\theta \colon H \to H$ and consider 
the associated truncated ribbon graph
$\Gr_{(\sigma,\theta)}$.

\subsubsection{Identifications}\label{subsec:ident}
Next, we consider the truncated ribbon graph  $\Gr_{(\sigma,\theta)}$ associated to the pair $(\sigma,\theta)=(\sigma,\theta)_{(Q,I)}$.
By definition there are bijections
$$
\begin{td}
	Q_1 \ar{r}{\sim} \& H, \& Q_1/\!\sim_{\sigma} \ar{r}{\sim} \& V \&
	Q_0 \ar{r}{\sim} \& Q_1/\!\sim_{\theta} \ar{r}{\sim} \& E
	%	\& Q_1/\!\sim_{\phi} \ar{r} \& F 
\end{td}
$$
where the last composition restricts to bijections 
$Q_0^t \overset{\sim}{\to} E_{t}$ and $Q_0^c \overset{\sim}{\to} E_{g}$.
By definition there is a bijection 
$$ \begin{td} Q_1/\!\sim_{\phi} \ar{r}{\sim} \& F,
	\& {[\alpha]_{\phi}} \ar[mapsto]{r} \& f_{\alpha} \end{td}$$
For any arrow $\alpha \in Q_1$
one of the following occurs.
\begin{itemize}[label=--]
	\item If $j \colonequals s(\alpha) \in Q_0^t$, the arrow $\alpha$ is contained in a forbidden thread starting  at $j$, its $\phi$-orbit is given by the arrows in the AG-cycle $c(j)$  and $f_{\alpha} \in F_{b}$.
%	In this case, $|f_{\alpha}| = \ell(c(j))$.
	\item	If $j \in Q_0^c$, the arrow	$\alpha$ is contained in a forbidden cycle $fc(\alpha)$ starting at $j$ and $f_{\alpha} \in F_{p}$.
%	 In this case, $|f_{\alpha}| = \ell(fc(\alpha))$.
\end{itemize}
In each case, the cardinality $|f_{\alpha}|$ of the $\phi$-orbit matches the length of a cycle naturally associated with $\alpha$.
Let $Q_1^{ft}$ denote the set of arrows appearing in forbidden threads.
Since  $\phi$ restricts to the permutation $\rho$ on $Q_1^{fc}$ and preserves $Q_1^{ft}$, it follows that
%		 Note that $\rho$ is the restriction $\phi|_{Q_1^{fc}}$.
there are bijections
$$
\begin{td}
	Q_0^t/\!\sim_{\kappa} \ar{r}{\sim} \& Q_1^{ft}/\!\sim_{\phi} \ar{r}{\sim} \& F_{b} \&
	Q_1^{fc}/\!\sim_{\rho} \ar{r}{\sim} \& F_{p}
\end{td}
$$
Summarized, the sets of rotation classes of permitted cycles, AG-cycles and forbidden cycles in $(Q,I)$
have graph-theoretic interpretations as vertices, boundary faces and punctured faces of $\Gr_{(\sigma,\theta)}$.

	\begin{table}
				\caption{Dictionary of invariants}
		\label{tab:dict}
		\begin{center}
		$
	\renewcommand{\arraystretch}{1.5}
	\begin{array}{|@{}w{c}{0.115\textwidth}|@{}w{c}{0.115\textwidth}|@{}w{c}{0.22\textwidth}|
		|@{}w{c}{0.125\textwidth}|@{}w{c}{0.115\textwidth}|@{}w{c}{0.22\textwidth}|}
		\text{quiver} & 	\text{graph} &\text{derived} & 
		\text{quiver}  & 	\text{graph} &\text{derived}
		\\[-0.25cm]
		\text{ invariant} &				\text{ invariant} &		\text{ invariant} &		\text{ invariant} &		\text{ invariant} &		\text{ invariant} 
		\\
		\hline
				\lvert Q_1\rvert & \lvert H\rvert & \lvert \rad \A\rvert 
				&
				\lvert Q_1/\!\sim_\sigma \rvert &  \lvert V\rvert & \rk K_0(\A_{\KK})  \\
		\lvert Q_0\rvert & \lvert E\rvert & \rk K_0(\A)  &						\bc_{\A} & \bc_\Gr & \ \rk C_{\A} - \rk K_0(\A_{\KK}) \\
		\lvert Q_0^t\rvert & 
\lvert E_{t} \rvert &  
\rk K_0(\AI)&
		\lvert Q_0^c\rvert & \lvert E_{g}\rvert&	  \rk K_0 (\Ainf) 
		\\
		\lvert Q_0^t/\!\sim_{\kappa} \rvert & \lvert F_{b}\rvert & \lvert\exc(\coaisle)/\!\sim\rvert & 
		\lvert Q^{fc}_1/\!\sim_{\rho}\rvert & \lvert F_{p}\rvert & \lvert \ind \Dsg(\A)/[1]\rvert  \\		
		\lvert Q_1^{ft}\rvert & 
	{\displaystyle			 \sum_{\mathclap{f \in F_b}} \lvert f\rvert }& 
	{\displaystyle  \sum_{\mathclap{{\Exc \in \exc(\coaisle)/\!\sim}}} m({\Exc})} &		
		\lvert Q_1^{fc}\rvert & 
	{\displaystyle  \sum_{\mathclap{f \in F_p}} \lvert f\rvert }&
		\lvert \ind \Dsg(\A) \rvert 
	\end{array}$
\end{center}
\end{table}
	{\footnotesize
\noindent {\it Remarks on notation}.
In the table above, $(Q,I)$ is the quiver underlying a gentle order $\A$ and $\Gr_{(\sigma,\theta)}$ 
its associated truncated ribbon graph.
	All quiver invariants have been defined in Subsection~\ref{subsec:gentle}.
	%	\item 
	All ribbon graph invariants were introduced in Subsection~\ref{subsec:trRG}.
	%	\item 
	In the column of derived invariants, $\lvert \rad \A \rvert$ denotes the number of isomorphism classes of indecomposable summands of the $\A$-module $\rad \A$.
	The overring $\A_{\KK}$ of the gentle order $\A$ was defined in Subsection~\ref{subsec:main-setup},
	the subring $\A_e$ and the quotient ring $\AI$ were introduced in \eqref{eq:idem}, 	the Cartan matrix $C_{\A}$ in Definition~\ref{dfn:Cartan}.
	If $R$ denotes any of the rings above, $\rk K_0(R)$ denotes the rank of its Grothendieck group, which is equal to $\lvert R\rvert$ as well as $\lvert \head R \rvert$.
	The notion $\exc{\coaisle}/\!\sim$
	was defined in 	Corollary~\ref{cor:AG1}, notions involving 
	$\Dsg(\A)$ in Corollary~\ref{cor:AG2}.
	For an equivalence class $\Exc \in \exc{\coaisle}/\!\sim$, the notion $m(\Exc)$ denotes the first entry in the pair $\CYdim \Exc$. 	}	

\begin{prp}\label{prp:dict}
Let $\A$ and $\B$ be derived equivalent gentle orders.
\begin{enumerate}
	\item
	Any three invariants in the same row of Table~\ref{tab:dict} are equal.
	\item
	The gentle orders $\A$ and $\B$ share any invariant listed in Table~\ref{tab:dict}. In particular, it holds that $\lvert \ind \rad \A\rvert = \lvert \rad \B\rvert$. 
	\item The associated surfaces $\Surf_{\A}$ and $\Surf_\B$ are homeomorphic.
\end{enumerate}
\end{prp}
\begin{proof}
\begin{enumerate}
	\item
	%		By construction of the ribbon graph $\Gamma$ in Subsection~\ref{subsec:trRG},
	%		there are bijections
	%		$Q_1 \cong H$, $Q_1/\!\sim_{\sigma} \cong V$, $Q_1/\!\sim_{\theta} \cong Q_0 \cong E$ 
	%		which restricts to $Q_0^{c} \cong E_{gl}$ and $Q_0^t \cong E_{tr}$
	%		 and $Q_1/\!\sim_{\phi} \cong F$.

	The equalities of quiver-theoretic invariants and invariants from ribbon graph data follow from the identifications in Subsection~\ref{subsec:ident}.
	%Finally, $\bc(Q,I) = \bc(\Gamma)$ follows by construction of $\Gamma$. 
	%This shows the equalities of invariants from quiver and invariants of ribbon graph data.	
	The homological interpretations of
	$Q_0$, $Q_0^c$ and $Q_1$ are straightforward. It holds that $|Q_0^t| = \rk K_0(\AI)$ by definition of $e$.  		  
	Corollaries~\ref{cor:AG1} 
	and~\ref{cor:AG2} imply that
	\begin{align*}
		(\lvert Q_1^{ft}\rvert,\lvert Q_0^t \rvert) = \sum_{\mathclap{(m,n) \in \AG_1(\A)}} (m,n) 
		= \sum_{\mathclap{\Exc \in \exc(\coaisle)/\!\sim}} \CYdim \Exc, &&
		\lvert Q_1^{fc}\rvert = \sum_{\mathclap{(m,0)\in\AG_2(\A)}} m = 
		\lvert\ind \Dsg(\A)\rvert.
		\end{align*}
	Proposition~\ref{prp:rkC} and Lemma~\ref{cor:rhull}
	yield the homological expressions for $\bc_{\Gr}$   and $|V|$, respectively.			
	\item
	The derived invariance of $|Q_0^c|$  and $|Q_0^t|$ in the sense above 
	follows from Proposition~\ref{prp:der-gen}.
	Using that $|Q_1| = 2|Q_0^c| + |Q_0^t|$ 
%	or $|Q_1|= |Q_1^{ft}| + |Q_1^{fc}|$
	 it follows that $|Q_1|$ is derived invariant as well. 
	The remaining derived invariance results follow from Theorem~\ref{thm:der-inv}.
	\item	The genus $g(\Surf)$ of the surface $\Surf \colonequals \Surf_{\A}$ associated to $\A$  is determined by the formula \eqref{eq:EP}.
\qedhere
\end{enumerate}
\end{proof}

\end{document}